\documentclass[11pt,letterpaper]{amsart}
\usepackage{amsfonts}
\usepackage{amsmath}
\usepackage{amssymb}
\usepackage{mathrsfs}
\usepackage{graphicx}
\usepackage{subfigure}
\usepackage{hyperref}
\usepackage{microtype}
\usepackage{enumerate}
\usepackage{geometry}
\usepackage{tikz}%% drawing
\usepackage{tikz-cd}%% drawing
\usepackage{color}
\numberwithin{equation}{section}
\allowdisplaybreaks[1]
\usepackage{titletoc}
\usepackage{amsthm}

\usepackage{cite}

\newcommand{\la}{\lambda}
\newcommand{\al}{\alpha}
\newcommand{\be}{\beta}
\newcommand{\ga}{\gamma}

\newcommand{\ve}{\varepsilon}

\newcommand{\vp}{\varphi}

\newcommand{\R}{\mathbb{R}}
\newcommand{\N}{\mathbb{N}}

\newcommand{\om}{\omega}

\newcommand{\n}[1]{\Vert #1\Vert }
\newcommand{\bn}[1]{\big \Vert #1 \big \Vert }
\newcommand{\bbn}[1]{\Big\Vert #1 \Big \Vert }
\newcommand{\lr}[1]{\left\{ #1 \right\} }
\newcommand{\lrc}[1]{\left[ #1 \right] }
\newcommand{\lrs}[1]{\left( #1 \right) }

\newcommand{\lra}[1]{\langle #1 \rangle }

\newcommand{\wt}[1]{\widetilde{#1}}

\newcommand{\pa}{\partial}

\newcommand{\cf}{{\mathcal F}}

\begin{document}

\newtheorem{theorem}{Theorem}[section]
\newtheorem{lemma}[theorem]{Lemma}

\theoremstyle{definition}
\newtheorem{definition}[theorem]{Definition}
\newtheorem{example}[theorem]{Example}
\newtheorem{remark}[theorem]{Remark}

\numberwithin{equation}{section}

\newtheorem{proposition}[theorem]{Proposition}
\newtheorem{corollary}[theorem]{Corollary}
\newtheorem{goal}[theorem]{Goal}
\newtheorem{algorithm}{Algorithm}

\renewcommand{\figurename}{Fig.}

\title[Global well-posedenss of  Boltzmann Equation]{Sharp Global Well-posedness and Scattering of the Boltzmann Equation}

\author[X. Chen]{Xuwen Chen}
\address{Department of Mathematics, University of Rochester, Rochester, NY 14627, USA}
\email{xuwenmath@gmail.com}
%\email{chenxuwen@math.umd.edu}

\author[S. Shen]{Shunlin Shen}
\address{School of Mathematical Sciences, Peking University, Beijing, 100871, China}
\email{slshen@pku.edu.cn}

\author[Z. Zhang]{Zhifei Zhang}
\address{School of Mathematical Sciences, Peking University, Beijing, 100871, China}

\email{zfzhang@math.pku.edu.cn}

%%
%%
%%
%%
%%
%\subjclass is required.
%%
%%
%%\date{}
%%
%%\dedicatory{}
%%
%%%    Abstract is required.
\begin{abstract}
We consider the 3D Boltzmann equation for the Maxwellian particle and soft potential with an angular cutoff. We prove sharp global well-posedness with initial data small in the scaling-critical space. The solution also remains in $L^{1}$ if the initial datum is in $L^{1}$, even at such low regularity. The key to existence, uniqueness and regularity criteria is the new bilinear spacetime estimates for the gain term, the proof of which is based on novel techniques from nonlinear dispersive PDEs including the atomic $U$-$V$ spaces, multi-linear frequency analysis, dispersive estimates, etc. To our knowledge, this is the first 3D sharp global result for the Boltzmann equation.
 \end{abstract}
\maketitle
%\begin{flushleft}
\tableofcontents
%\end{flushleft}
%\titlecontents{section}
%              [0cm]
%              {}%
%              {\contentslabel{0em}}%
%              {}%
%              {\titlerule*[0.5pc]{}\contentspage\hspace*{0cm}}%
%\titlecontents{subsection}
%              [0cm]
%              {}%
%              {\contentslabel{0em}}%
%              {}%
%              {\titlerule*[0.5pc]{}\contentspage\hspace*{0cm}}%

\section{Introduction}
The Boltzmann equation is a basic mathematical model in the collisional kinetic theory which describes the statistical evolution of a dilute gas. The Cauchy problem for the Boltzmann equation is of crucial importance for the physical interpretation and practical
application, and is thus one of the fundamental problems in kinetic theory. So far,
a large quantity of mathematical theories have been developed by using various methods for constructing solutions in different settings,
see for example \cite{alexandre2013local,alonso09distributional,ampatzoglou2022global,arsenio2011global,chen2019local,chen2019moments,chen2021small,chen2022well,chen2023well,diperna1989cauchy,
duan2016global,duan2018solution,guo2003classical,guo2003the,guo04boltzamnn,he2017well,he2023cauchy,illner84the,kaniel78the,alexandre2011the,alexandre2011thehard,alexandre2012the,alexandre2011global,chaturvedi2021stability,
duan2021global,gressman2011global,gressman2011sharp,imbert2022global}. Despite the significant progress, it remains an open problem to characterize the optimal regularity of initial data for the well-posedness including the global existence, uniqueness, continuity of the solution map, scattering, conservation laws, and etc. This is not only of mathematical and physical interests for perfection, but also an actual need for many related problems, such as
the derivation of the Boltzmann equation from the classical particle systems or quantum many-body dynamics, its hydrodynamic limit to fluid equations, and many others.
 In the paper, we investigate the sharp global well-posedness of the Boltzmann equation.

The general 3D Boltzmann equation takes the form
\begin{equation}\label{equ:Boltzmann}
\left\{
\begin{aligned}
\left( \partial_t + v \cdot \nabla_x \right) f (t,x,v) =& \int_{\mathbb{S}^{2}} \int_{\mathbb{R}^{3}} \lrc{f(v^{\ast })f(u^{\ast})-f(v)f(u)} B(u-v,\omega)dud\omega,\\
f(0,x,v)=& f_{0}(x,v),
\end{aligned}
\right.
\end{equation}
where $f(t,x,v)$ denotes the distribution function for the particles
at time $t\geq 0$, position $x\in \R^{3}$ and velocity $v\in \R^{3}$.
The variables $u$, $v$ can be regarded as pre-collision velocities for a pair of particles, $\om\in \mathbb{S}^{2}$ is a parameter for the deflection angle in the collision process, and the after-collision velocities $u^{*}$, $v^{*}$ are given by
\begin{align*}
u^{*}=u+\omega\cdot (v-u) \omega,\quad v^{*}=v-\omega\cdot(v-u)\omega.
\end{align*}
We adopt the usual shorthand $Q(f,g)$ to denote the nonlinear collision term of \eqref{equ:Boltzmann}, which is conventionally split into a gain term
and a loss term:
\begin{align}
Q(f,g)=&Q^{+}(f,g)-Q^{-}(f,g),\\
\text{(gain term)}\quad Q^{+}(f,g)=&\int_{\mathbb{S}^{2}} \int_{\mathbb{R}^{3}} f(v^{\ast })g(u^{\ast}) B(u-v,\omega)dud\omega,\\
\text{(loss term)} \quad Q^{-}(f,g)=&f A\lrc{g},\quad A\lrc{g}= \int_{\mathbb{S}^{2}} \int_{\mathbb{R}^{3}} g(u) B(u-v,\omega) du d\omega. \label{equ:loss term,A[g]}
\end{align}
Due to physical considerations of collision,
the Boltzmann collision kernel function $B(u-v,\omega)$ is a non-negative function depending only on the relative velocity $|u-v|$ and the deflection angle $\theta$ through $\cos \theta:=\frac{u-v}{|u-v|}\cdot \omega$. Throughout the paper, we consider
\begin{align}\label{equ:kernel function}
B(u-v,\omega)=|u-v|^{\ga}\textbf{b}(\cos \theta)
\end{align}
 under the Grad's angular cutoff assumption
\begin{align*}
0\leq  \textbf{b}(\cos \theta)\leq C|\cos\theta|.
\end{align*}
 The collision kernel \eqref{equ:kernel function} originates from the physical model of inverse-power law potentials and
 the different ranges $\ga<0$, $\ga=0$, $\ga>0$ correspond to soft potentials, Maxwellian molecules, and hard potentials, respectively. See also \cite{cercignani1988boltzmann,cercignani1994mathematical,villani2002review} for a more detailed physics background.

There have been many advancements of well-posedness theories requiring as less regularity as possible on the initial data.
 However, it is highly nontrivial to find the critical regularity of initial data for well-posedness. On the one hand,
the critical regularity for the Boltzmann equation is sometimes believed at $s=\frac{3}{2}$, the continuity threshold, see for example \cite{alexandre2013local,duan2016global,duan2021global,duan2018solution} for a more discussion. On the other hand,
from the scaling point of view,
 the Boltzmann equation \eqref{equ:Boltzmann} is invariant under the scaling
\begin{align}
f_{\la}(t,x,v)=\la^{\al+(2+\ga)\be}f(\la^{\al-\be}t,\la^{\al}x,\la^{\be}v),
\end{align}
for any $\al$, $\be\in \R$ and $\la>0$. Then in the $L^{2}$ setting, it holds that
\begin{align*}
\n{|\nabla_{x}|^{s}|v|^{r}f_{\la}}_{L_{xv}^{2}}=\la^{^{\al+(2+\ga)\be}}\la^{\al s-\be r}\la^{-\frac{3}{2}\al-\frac{3}{2}\be}
\n{|\nabla_{x}|^{s}|v|^{r}f}_{L_{xv}^{2}},
\end{align*}
which gives the scaling-critical index
\begin{align}\label{equ:scaling,L2}
s=\frac{1}{2},\quad  r=s+\ga.
\end{align}
That is, in term of scaling, a guiding principe,
one expects that the well/ill-posedness threshold in $H^{s}$ Sobolev space is $s_{c}=\frac{1}{2}$ with $r\geq 0$.

In a recent series of paper \cite{chen2019local,chen2019moments,chen2021small}, by adopting dispersive techniques on the study of the quantum many-body hierarchy dynamics, especially space-time collapsing/multi-linear estimates techniques (see for instance \cite{chen2015unconditional,
chen2014derivation,chen2013rigorous,chen2016focusing,chen2016collapsing,chen2016klainerman,chen2016correlation,
chen2019derivation,chen2022quantitative,chen2023derivation,herr2016gross,herr2019unconditional,
kirkpatrick2011derivation,klainerman2008uniqueness,sohinger2015rigorous}),
T. Chen, Denlinger, and Pavlovi$\acute{c}$ provided an alternate dispersive PDE based route for proving well-posedness of the Boltzmann equation and hierarchy.
With the introduction of dispersive techniques, the regularity index for local well-posedness, which is usually at least the continuity threshold $s>\frac{3}{2}$, has been improved to $s>1$ for both 3D Maxwellian molecules and hard potentials with cutoff in \cite{chen2019local}.
Unexpectedly in the scaling point of view, for the 3D constant kernel case, X. Chen and Holmer in \cite{chen2022well} found the well/ill-posedness threshold in $H^{s}$ Sobolev space was exactly at regularity $s=1$, and thus pointed out the actual optimal regularity for the global well-posedness problem. Subsequently, in our work \cite{chen2023well}, we moved forward from the special constant kernel case to investigate the general
kernel with soft potentials, and proved that the well/ill-posedness threshold was also $s=1$.

With the finding of this critical regularity for well-posedness, just like many other physically important equations \cite{danchin2000global,klainerman2003improved,
koch2001well,krieger2017global,smith2005sharp,wang2022rough}, a challenging problem for the Boltzmann
equation is whether or not one could prove the sharp global well-posedness even for small initial data.
Our main result provides an affirmative answer.

\begin{theorem}[Main Theorem]\label{thm:main theorem}
Let $s>1$ and $\ga\in[-\frac{1}{2},0]$.
There exists $\eta>0$, such that for all non-negative initial data $f_{0}$ satisfying the regularity condition that
\begin{align*}
\n{f_{0}}_{L_{v}^{2,s+\ga}H_{x}^{s}}:=\n{\lra{\nabla_{x}}^{s}\lra{v}^{s+\ga}f_{0}}_{L_{x,v}^{2}}<\infty,
\end{align*}
and the scaling-critical smallness condition that
\begin{align}\label{equ:gwp,smallness} \n{\lra{\nabla_{x}}^{\frac{1}{2}}\lra{v}^{\frac{1}{2}+\ga}f_{0}}_{L_{x,v}^{2}}\leq \eta,
\end{align}
the Boltzmann equation \eqref{equ:Boltzmann} is global well-posed in $C([0,\infty);L_{v}^{2,s+\ga}H_{x}^{s})$ and the solution scatters.
Furthermore, if $f_{0}\in L_{x,v}^{1}$, then for $t\in [0,\infty)$ we have
$$\n{f(t)}_{L_{x,v}^{1}}\leq \n{f_{0}}_{L_{x,v}^{1}}.$$
\end{theorem}
\begin{remark}
Theorem \ref{thm:main theorem} is sharp, as we have proven that for $s<1$, the Cauchy problem of the Boltzmann equation is ill-posed in \cite{chen2023well}. The range $\ga\in[-\frac{1}{2},0]$ is the endpoint of our method.
On the one hand, the scaling analysis \eqref{equ:scaling,L2} and the scaling-critical norm \eqref{equ:gwp,smallness} imply that $\frac{1}{2}+\ga\geq 0$. If not, the well-definiteness of the Boltzmann equation is a problem. On the other hand,
 our proof depends on a scaling-invariant estimate which does not work well if the Sobolev index is negative.

\end{remark}
Theorem \ref{thm:main theorem} is actually contained in the following theorem.
\begin{theorem}\label{thm:main theorem,2}
Let $s\in (1,\frac{3}{2})$ and $\ga\in[-\frac{1}{2},0]$. There exists $\eta>0$, such that for all non-negative initial data $f_{0}\in L_{v}^{2,s+\ga}H_{x}^{s}$ satisfying
\begin{align}
\n{\lra{\nabla_{x}}^{\frac{1}{2}}\lra{v}^{\frac{1}{2}+\ga}f_{0}}_{L_{x,v}^{2}}\leq \eta,
\end{align}
we have:
\begin{enumerate}[$(1)$]
\item $($Existence$)$
There exists a non-negative $C([0,\infty);L_{v}^{2,s+\ga}H_{x}^{s})$ solution $f(t)$ satisfying
\begin{align}
&\n{\lra{v}^{s+\ga}f}_{L_{t}^{\infty}(0,\infty;L_{v}^{2}L_{x}^{p})}<\infty ,\label{equ:f,Lx6,theorem}\\
&\n{\lra{v}^{s+\ga}Q^{+}(f,f)}_{L_{t}^{1}(0,\infty;L_{v}^{2}L_{x}^{p})}< \infty,\label{equ:Q+,f,Lx6,theorem}\\
&\n{A\lrc{f}}_{L_{t}^{2}(0,\infty;L_{x}^{\infty}L_{v}^{\infty})}<\infty,\label{equ:A,f,Lxinfty,theorem}\\
&\n{\lra{v}^{s+\ga}Q^{-}(f,f)}_{L_{t}^{1}(0,T;L_{v}^{2}L_{x}^{p})}\leq C(p,T),\label{equ:Q-,f,Lx6,theorem}
\end{align}
for all $p\in [2,\frac{6}{3-2s}]$ and all $T\in(0,\infty)$.
\item $($Uniqueness$)$ The solution $f(t)$ is unique in a larger class of all $C([0,T];L_{v}^{2,s+\ga}L_{x}^{2})$ solutions satisfying the integrability bounds \eqref{equ:f,Lx6,theorem}--\eqref{equ:Q-,f,Lx6,theorem} on $[0,T]$ with $p=6$.
\item $($Scattering$)$ The solution $f(t)$ scatters in $L_{v}^{2,s+\ga}L_{x}^{p}$ for all $p\in [2,\frac{6}{3-2s}]$. That is, there exists a function $f_{+\infty}\in L_{v}^{2,s+\ga}L_{x}^{p}$ such that
    \begin{align*}
    \lim_{t\to+\infty}\n{f(t)-S(t)f_{+\infty}}_{L_{v}^{2,s+\ga}L_{x}^{p}}=0,
    \end{align*}
    where $S(t)=e^{-tv\cdot\nabla_{x}}$.
\item $($Lipschitz continuity of the solution map$)$ The solution map
$$f_{0}\in L_{v}^{2,s+\ga}H_{x}^{s}\mapsto f\in C([0,T];L_{v}^{2,s+\ga}H_{x}^{s})$$
is Lipschitz\footnote{The solution map is actually analytic continuous as it comes from an argument of contraction map in our proof.} continuous.
\item $($Persistence of regularity$)$ Further suppose that $f_{0}\in L_{v}^{2,s+\ga+\be}H_{x}^{s+\al}$ for some $\al\geq 0$, $\be\geq 0$, then we have that
$f(t)\in C([0,\infty);L_{v}^{2,s+\ga+\be}H_{x}^{s+\al})$ and
\begin{align*}
&\n{\lra{\nabla_{x}}^{s+\al}\lra{v}^{s+\ga+\be}Q^{\pm}(f,f)}_{L_{t}^{1}(0,T;L_{x,v}^{2})}<\infty,
\end{align*}
for all $T\in[0,\infty)$.
\item $($Finite mass density$)$ Moreover, if $f_{0}\in L_{x,v}^{1}$, then for $t\in [0,\infty)$ we have
$$\n{f(t)}_{L_{x,v}^{1}}\leq \n{f_{0}}_{L_{x,v}^{1}}.$$
\end{enumerate}
\end{theorem}
Using the global well-poseness and persistence of regularity, we immediately have the following corollary for smooth initial data with respect to the spatial variable.
\begin{corollary}
Smooth datum subject to the smallness $\eqref{equ:gwp,smallness}$ generates a global smooth solution which scatters. That is, if $f_{0}\in \bigcap_{\al>0}L_{v}^{2,r}H_{x}^{\al}$ with $r>1+\ga$, then
$$f(t)\in \bigcap_{\al>0}C([0,\infty);L_{v}^{2,r}H_{x}^{\al}).$$
\end{corollary}

Continuing a great deal of efforts such as \cite{alexandre2013local,alonso09distributional,arsenio2011global,chen2019local,chen2019moments,chen2021small,chen2022well,chen2023well,diperna1989cauchy,
duan2016global,duan2018solution,guo2003classical,guo2003the,guo04boltzamnn,he2017well,he2023cauchy,illner84the,kaniel78the} devoted to the well-posedness theory of the Boltzmann equation with an angular cutoff\footnote{The non-cutoff case is of equally importance and many nice developments have been achieved, see for example \cite{alexandre2011the,alexandre2011thehard,alexandre2012the,alexandre2011global,chaturvedi2021stability,
duan2021global,gressman2011global,gressman2011sharp,imbert2022global}.}, Theorems \ref{thm:main theorem}--\ref{thm:main theorem,2} establish a sharp global well-posedness with small initial data in the
scaling-critical space for both Maxwellian molecules and soft potential cases.\footnote{The hard potential case is also interesting and the global well/ill-posedness results remain open. However, it needs a different working space even for the local well-posedness, see \cite{chen2019local}. Hence, it requires new ideas to deal with these problems which we put for further work.} To the best of our knowledge, this is the first 3D sharp global result.

\subsection{Outline of Proof}
In \cite{chen2023well}, we have proved a sharp local well-posedness in the $L_{v}^{2,s+\ga}H_{x}^{s}$ space for $s>1$ and hence provided a blow-up criterion that
\begin{align}
\lim_{t\nearrow T(f_{0})}\n{f_{loc}(t)}_{L_{v}^{2,s+\ga}H_{x}^{s}}=\infty,
\end{align}
where $T(f_{0})$ is the lifespan.
To obtain a global result,
it suffices to establish a priori regularity bound $L_{t}^{
\infty}(0,T(f_{0});L_{v}^{2,s+\ga}H_{x}^{s})$ on the local strong solution $f_{loc}(t)$. The strategy we take is divided into the following three steps.
\begin{enumerate}[\textbf{Step} 1.]
\item
Construct a global solution $f(t)$ to the Boltzmann equation. In the step, the regularity bound is not required but some good decay and integrability properties are needed for a subsequent analysis.
\item
Establish a uniqueness theory to prove that the global solution $f(t)$ coincides with the strong local solution $f_{loc}(t)$. Therefore, the global solution $f(t)$ recovers the regularity at least for a short time.
\item
 Provide a regularity criterion to prove the persistence of regularity for the global solution. Once the regularity bound of $f(t)$ is set up, by the blow-up criterion and uniqueness theorem, we conclude that $T(f_{0})=\infty$ and hence obtain the global well-posedness.

\end{enumerate}

Due to the complexity of the collision kernel, it is quite hard to solve the Boltzmann equation at critical regularity. Though the gain term and the loss term scale
the same way, they have totally different structures and hence cannot share the same critical estimates.
To beat it, we make use of a classical technique, the Kaniel--Shinbrot iteration \cite{kaniel78the}, the main point of which is to solve the gain-term-only Boltzman equation.
This method has been successful in many work such as \cite{chen2021small,he2023cauchy,illner84the}. Especially in \cite{chen2021small}, with
a novel application of this iteration scheme, the global well-posedness of 2D Boltzmann equation with a constant collision
kernel is solved for small $L_{x,v}^{2}$-critical initial data and the result is actually sharp with the ill-posedness results in \cite{chen2022well,chen2023well} and some other tools in this paper. For the physical 3D problem, to obtain the global well-posedness of the gain-term-only Boltzmann equation, we prove a completely new 3D scaling-invariant bilinear estimate for the gain term. As it requires derivatives to be scaling-invariant, the related harmonic analysis nitty-gritty technicalities come in and the proof highly relies on the latest dispersive techniques.\footnote{As shown in \cite[p.138]{cercignani1994mathematical}, it has been proved in \cite{andreasson2004blowup,illner87blowup} that the $L^1$ norm for the 3D gain-term-only Boltzmann equation might blow up and hence hinders the application of the Kaniel-Shinbrot iteration for general data. That is not the case here. Our low regularity solution for \eqref{equ:Boltzmann} stays in $L^1$ provided that the initial datum is in $L^1$ as a consequence of well-posedness in a Strichartz-type space, though the solution to the gain-term-only equation may not stay in $L^{1}$. This also addresses the question raised in \cite[Remark 1.5]{chen2021small}.}

In Section \ref{section:atomic U-V spaces}, we introduce an updated dispersive technique, the atomic $U$--$V$ spaces, which was first developed by Koch and Tataru \cite{koch05dispersive,koch07priori} and played a key role in solving critical problems. Apart from this, one important observation is that the energy conservation provides
a lower bound estimate for after-collision velocities, which enables the application of the
Littlewood-Paley theory and multi-linear frequency analysis techniques.
In Section \ref{section:Scaling-invariant Bilinear Estimate for Gain Term}, we give a subtle frequency analysis on the gain term and set up the scaling-invariant bilinear estimate
 with the help of a convolution type estimate and Strichartz estimates. Finally in Section \ref{section:Small Data Global Well-posedness for Gain-term-only Boltzmann}, we complete the proof of the global well-posedness of the gain-term-only Boltzmann equation with small initial data in the scaling-critical space.
To our knowledge, this seems to be the 1st application of the atomic $U$-$V$ spaces techniques on the study of the well-posedness of the Boltzmann equation, and we believe that the approach would be helpful in many related problems in different settings, as these spaces have made the estimates more unified and the format cleaner.

In Section \ref{section:Global Existence of the Boltzmann Equation}, we prove the global existence and scattering of the Boltzmann equation by using
the Kaniel--Shinbrot iterative method. The idea is to put the solution of the gain-term-only Boltzmann equation as an upper bound of the beginning condition in the iteration scheme. However, a key point of this method, the uniqueness of the limiting equation, cannot be easily obtained. As the limiting point only enjoys integrability bounds instead of regularity bounds, the uniqueness must hold in some integrable class. In Section \ref{section:uniqueness}, we establish the uniqueness theorem. As we work in the integrable class, an $L_{t}^{1}L_{v}^{2,s+\ga}L_{x}^{2}$ bilinear estimate which carries no regularity, plays an important role.

In Section \ref{section:Strong Solution and Blow-up Criteria}, we state the strong local well-posedness in the setting of atomic $U$-$V$ spaces and give the blow-up criterion. In Section \ref{section:Persistence of Regularity}, to propagate the regularity for the Boltzmann equation, we provide the regularity criteria based on the integrability bounds, one novel application of which is to solve \cite[Conjecture 1.1]{chen2021small}.\footnote{See Remark \ref{remark:conjecture} for details.}
The key is a new $L_{t}^{1}L_{v}^{2,s+\ga}H_{x}^{s}$ bilinear estimate for the gain term, the proof of which also highly relies on the frequency analysis techniques like in Section \ref{section:Scaling-invariant Bilinear Estimate for Gain Term}.

Putting together all the results in Sections \ref{section:Global Existence of the Boltzmann Equation}--\ref{section:Persistence of Regularity}, we finish the proof of Theorems \ref{thm:main theorem}--\ref{thm:main theorem,2}.

\section{The Gain-term-only Boltzmann Equation} \label{section:The Gain-term-only Boltzmann Equation}
In the section, we will take dispersive techniques to deal with the gain-term-only Boltzmann equation:
\begin{equation}\label{equ:Boltzmann,gain}
\left\{
\begin{aligned}
\left( \partial_t + v \cdot \nabla_x \right) f (t,x,v) =& Q^{+}(f,f),\\
f(0,x,v)=& f_{0}(x,v).
\end{aligned}
\right.
\end{equation}
To draw a connection between the analysis of \eqref{equ:Boltzmann,gain} and the theory of nonlinear
dispersive PDEs, we
take the inverse $v$-variable Fourier transform on both side of \eqref{equ:Boltzmann,gain} to get
\begin{align}
i\pa_{t}\wt{f}+\nabla_{\xi}\cdot \nabla_{x}\wt{f}=i\mathcal{F}_{v\mapsto \xi}^{-1}\lrc{Q^{+}(f,f)},
\end{align}
where $\wt{f}(t,x,\xi)=\mathcal{F}^{-1}_{v\mapsto \xi}(f)$.
The linear part of \eqref{equ:Boltzmann,gain} changes into the symmetric hyperbolic Schr\"{o}dinger equation and hence gives Strichartz estimates (see the Appendix
\ref{section:Strichartz Estimates}) that
\begin{align}
\n{e^{it\nabla_{\xi}\cdot \nabla_{x}}\wt{f}_{0}}_{L_{t}^{q}L_{x\xi}^{p}}\lesssim \n{\wt{f}_{0}}_{L_{x\xi}^{2}},\quad \frac{2}{q}+\frac{6}{p}=3,\quad q\geq 2.
\end{align}
For the nonlinear part of \eqref{equ:Boltzmann,gain},
by the well-known Bobylev identity for a general case, see for example \cite{alexandre2000entropy,desvillettes2003use}, it holds that (up to an unimportant constant)
\begin{align}
\mathcal{F}_{v\mapsto \xi}^{-1}\lrc{Q^{+}(f,g)}(\xi)=&\int_{\R^{3}\times \mathbb{S}^{2}}\frac{\wt{f}(\xi^{+}+\eta)\wt{g}(\xi^{-}-\eta)}{|\eta|^{3+\ga}}
\textbf{b}(\frac{\xi}{|\xi|}\cdot \omega)d\eta d\omega,
\end{align}
where $\xi^{+}=\frac{1}{2}\lrs{\xi+|\xi|\omega}$ and $\xi^{-}=\frac{1}{2}\lrs{\xi-|\xi|\omega}$.
For convenience, we take the notation that $\wt{Q}^{+}(\wt{f},\wt{g})=\mathcal{F}_{v\mapsto \xi}^{-1}\lrc{Q^{+}(f,g)}$.

In Section \ref{section:atomic U-V spaces}, we introduce the atomic $U$-$V$ spaces techniques into the analysis of the Boltzmann equation. In Section \ref{section:Scaling-invariant Bilinear Estimate for Gain Term}, we establish a scaling-invariant bilinear estimate for the gain term, which is the key to the global well-posedness of the gain-term-only Boltzmann equation \eqref{equ:Boltzmann,gain}. Finally in Section \ref{section:Small Data Global Well-posedness for Gain-term-only Boltzmann}, we conclude the well-posedness of \eqref{equ:Boltzmann,gain} with small initial data in the scaling-critical space.

\subsection{Atomic $U$-$V$ Spaces}\label{section:atomic U-V spaces}
We give a brief introduction to the atomic $U$ spaces introduced by Koch and Tataru \cite{koch05dispersive,koch07priori} and
the $V$ spaces of bounded $p$-variation of Wiener \cite{wiener24the}. Their properties have been further elaborated in \cite{hadac09well,koch14dispersive}.

Let $\mathcal{Z}$ be the set of finite partitions $-\infty<t_{0}<t_{1}<...<t_{K}\leq +\infty$ of the real line.
If $t_{K}=+\infty$, we use the convention that $u(t_{K}):=0$ for all functions $u:\R\mapsto H$, where $H$ is a Hilbert space.
\begin{definition}
Let $1\leq p<\infty$. For $\lr{t_{k}}_{k=0}^{K}\in \mathcal{Z}$ and $\lr{\phi_{k}}_{k=0}^{K-1}\subset H$ with
$\sum_{k=0}^{K-1}\n{\phi_{k}}_{H}^{p}=1$ we call the step function
\begin{align}
a=\sum_{k=1}^{K}\chi_{[t_{k-1},t_{k})}\phi_{k-1}
\end{align}
a $U^{p}$-atom and define the atomic space $U^{p}(\R;H)$ of all functions $u:\R\mapsto H$ such that
\begin{align}
u=\sum_{j=1}^{\infty}\la_{j}a_{j},\quad \text{for $U^{p}$-atoms $a_{j}$, $\lr{\la_{j}}\in l^{1}$,}
\end{align}
with norm
\begin{align}
\n{u}_{U^{p}(\R;H)}:=\inf\lr{
\sum_{j=1}^{\infty}|\la_{j}|:u=\sum_{j=1}^{\infty}\la_{j}a_{j},\ \la_{j}\in \mathbb{C},\ a_{j}\ \text{$U^{p}$-atom}}.
\end{align}
\end{definition}
\begin{definition}
 Define $V^{p}(\mathbb{R};H)$ as the space of all functions $u:\mathbb{R}\mapsto H$ such that
\begin{align}
\n{u}_{V^{p}(\mathbb{R};H)}=\sup_{\lr{t_{k}}_{k=0}^{K}\in \mathcal{Z}}\lrs{\sum_{k=1}^{K}\n{u(t_{k+1})-u(t_{k})}_{H}^{p}}^{\frac{1}{p}}
\end{align}
is finite.
\end{definition}

 We will work exclusively with the variants $U_{L}^{p}(\mathbb{R};H)$, $V_{L}^{p}(\mathbb{R};H)$ defined as the
norms $U_{L}^{p}(\mathbb{R};H)$, $V_{L}^{p}(\mathbb{R};H)$, respectively, after pulling-back by a linear flow $U(t)$.
\begin{definition}\label{def:UL,VL}
Let $U_{L}^{p}(\mathbb{R};H)$, $V_{L}^{p}(\mathbb{R};H)$ be the space of all functions $u:\R\mapsto H$ such that $t\mapsto U(-t)u(t)$ is in $U^{p}(\mathbb{R};H)$, $V^{p}(\mathbb{R};H)$ respectively, with norms
\begin{align}
\n{u}_{U_{L}^{p}(\mathbb{R};H)}=\n{U(-t)u}_{U_{L}^{p}(\mathbb{R};H)},\quad
\n{u}_{V_{L}^{p}(\mathbb{R};H)}=\n{U(-t)u}_{V_{L}^{p}(\mathbb{R};H)}.
\end{align}
\end{definition}

In our setting, we take $H=L_{x,\xi}^{2}$, $U(t)=e^{it\nabla_{\xi}\cdot \nabla_{x}}$.
We provide the Strichartz estimates which will be used to establish the scaling-invariant bilinear estimate for the gain term.

\begin{proposition}[Strichartz estimate]\label{lemma:strichartz estimate,U-V}
Let $(q,p)$ be a pair satisfying
$$\frac{2}{q}+\frac{6}{p}=3,\quad q\geq 2.$$
Then we have the Strichartz estimates
\begin{align}\label{equ:strichartz estimate,U}
\n{\wt{f}}_{L_{t}^{q}L_{x,\xi}^{p}}\lesssim \n{\wt{f}}_{U_{L}^{q}L_{x,\xi}^{2}},
\end{align}
and
\begin{align}\label{equ:strichartz estimate,duhamel,V}
\bbn{\int_{0}^{t}U(t-\tau)\mathcal{N}[\wt{f}](\tau)d\tau}_{V_{L}^{q'}(0,\infty;L_{x,\xi}^{2})}\lesssim \n{\mathcal{N}[\wt{f}]}_{L_{t}^{q'}(0,\infty;L_{x,\xi}^{p'})},
\end{align}
where $q'=\frac{q}{q-1}$ and $p=\frac{p}{p-1}$.
Especially, we have
\begin{align}\label{equ:strichartz estimate,duhamel,U}
\bbn{\int_{0}^{t}U(t-\tau)\mathcal{N}[\wt{f}](\tau)d\tau}_{U_{L}^{2}(0,\infty;L_{x,\xi}^{2})}\lesssim \n{\mathcal{N}[\wt{f}]}_{L_{t}^{1}(0,\infty;L_{x,\xi}^{2})}.
\end{align}
\end{proposition}
\begin{proof}
The Strichartz estimates are well-known in the dispersive literature.
The estimates \eqref{equ:strichartz estimate,U}--\eqref{equ:strichartz estimate,duhamel,U} follow from the linear Strichartz estimate \eqref{equ:strichartz estimate,linear} in the Appendix \ref{section:Strichartz Estimates} and the definition of the $U$--$V$ spaces. See \cite[Chapter 4.10]{koch14dispersive} for more details.
\end{proof}

\subsection{Scaling-invariant Bilinear Estimate for the Gain Term}\label{section:Scaling-invariant Bilinear Estimate for Gain Term}
Before getting into the analysis of the scaling invariant-bilinear estimate, we provide a convolution type inequality as follows.
\begin{lemma}\label{lemma:Q+,bilinear estimate}
Let $\frac{1}{p}+\frac{1}{q}=\frac{1}{2}$.
\begin{align}\label{equ:Q+,bilinear estimate}
\bbn{\int_{\mathbb{S}^{2}}\int_{\R^{3}}\frac{\wt{f}(\xi^{+}+\eta)\wt{g}(\xi^{-}-\eta)}{|\eta|^{3+\ga}}
\textbf{b}(\frac{\xi}{|\xi|}\cdot \omega)d\eta d\omega}_{L_{\xi}^{2}}\lesssim
\n{\wt{f}}_{L^{\frac{6p}{6-p\ga}}} \n{\wt{g}}_{L^{\frac{6q}{6-q\ga}}}.
\end{align}
In particular, we have
\begin{align}
\n{\wt{Q}^{+}(\wt{f},\wt{g})}_{L_{\xi}^{2}}
\lesssim& \n{\wt{f}}_{L_{\xi}^{2}}
\n{\wt{g}}_{L_{\xi}^{\frac{3}{-\ga}}},\label{equ:Q+,bilinear estimate,PM,f,g}\\
\n{\wt{Q}^{+}(\wt{f},\wt{g})}_{L_{\xi}^{2}}
\lesssim&  \n{\wt{f}}_{L_{\xi}^{\frac{3}{-\ga}}}
\n{\wt{g}}_{L_{\xi}^{2}},\label{equ:Q+,bilinear estimate,PM,g,f}\\
\n{\wt{Q}^{+}(\wt{f},\wt{g})}_{L_{\xi}^{2}}
\lesssim& \n{\wt{f}}_{L_{\xi}^{3}}
\n{\wt{g}}_{L_{\xi}^{\frac{6}{1-2\ga}}},\label{equ:Q+,bilinear estimate,L3,PM,f,g}\\
\n{\wt{Q}^{+}(\wt{f},\wt{g})}_{L_{\xi}^{2}}
\lesssim&  \n{\wt{f}}_{L_{\xi}^{\frac{6}{1-2\ga}}}
\n{\wt{g}}_{L_{\xi}^{3}}.\label{equ:Q+,bilinear estimate,L3,PM,g,f}
\end{align}
\end{lemma}
\begin{proof}
For the case of Maxwellian molecules, it has been established in \cite[Theorem 1]{alonso2010estimates} that
\begin{align}\label{equ:Q+,bilinear estimate,constant kernel}
\bbn{\int_{\mathbb{S}^{2}} \wt{f}(\xi^{+})\wt{g}(\xi^{-})
\textbf{b}(\frac{\xi}{|\xi|}\cdot \omega) d\omega}_{L_{\xi}^{2}}\lesssim \n{\wt{f}}_{L_{\xi}^{p}}\n{\wt{g}}_{L_{\xi}^{q}},\quad \frac{1}{p}+\frac{1}{q}=\frac{1}{2}.
\end{align}
By Cauchy-Schwarz inequality and \eqref{equ:Q+,bilinear estimate,constant kernel}, we then have
\begin{align}\label{equ:Q+,bilinear estimate,proof}
&\bbn{\int_{\mathbb{S}^{2}}\int_{\mathbb{R}^{3}}\frac{\wt{f}(\xi^{+}+\eta)\wt{g}(\xi^{-}-\eta)}{|\eta|^{3+\ga}}
\textbf{b}(\frac{\xi}{|\xi|}\cdot \omega)d\eta d\omega}_{L_{\xi}^{2}}\\
\leq& \bbn{\int_{\mathbb{S}^{2}}\lrc{\int_{\mathbb{R}^{3}}\frac{|\wt{f}(\xi^{+}+\eta)|^{2}}{|\eta|^{3+\ga}}
d\eta}^{\frac{1}{2}}
\lrc{ \int_{\mathbb{R}^{3}}\frac{|\wt{g}(\xi^{-}-\eta)|^{2}}{|\eta|^{3+\ga}}d\eta}^{\frac{1}{2}}\textbf{b}(\frac{\xi}{|\xi|}\cdot \omega) d\omega}_{L_{\xi}^{2}}\notag\\
\leq& \bbn{\lrc{\int_{\mathbb{R}^{3}}\frac{|\wt{f}(\xi+\eta)|^{2}}{|\eta|^{3+\ga}}
d\eta}^{\frac{1}{2}}}_{L_{\xi}^{p}}
\bbn{\lrc{\int_{\mathbb{R}^{3}}\frac{|\wt{g}(\xi-\eta)|^{2}}{|\eta|^{3+\ga}}d\eta}^{\frac{1}{2}}}_{L_{\xi}^{q}} \notag\\
=& \bbn{\int_{\mathbb{R}^{3}}\frac{|\wt{f}(\xi+\eta)|^{2}}{|\eta|^{3+\ga}}
d\eta}^{\frac{1}{2}}_{L_{\xi}^{\frac{p}{2}}}
\bbn{\int_{\mathbb{R}^{3}}\frac{|\wt{g}(\xi-\eta)|^{2}}{|\eta|^{3+\ga}}d\eta}^{\frac{1}{2}}_{L_{\xi}^{\frac{q}{2}}}\notag\\
\lesssim& \n{\wt{f}}_{L^{\frac{6p}{6-p\ga}}} \n{\wt{g}}_{L^{\frac{6q}{6-q\ga}}},\notag
\end{align}
where in the last inequality we have used the Hardy-Littlewood-Sobolev inequality that
\begin{align*}
\n{u*|\cdot|^{-3-\ga}}_{L^{\frac{p}{2}}}\lesssim \n{u}_{L^{\frac{3p}{6-p\ga}}}.
\end{align*}
Therefore, we complete the proof of \eqref{equ:Q+,bilinear estimate}.
Then we obtain estimates \eqref{equ:Q+,bilinear estimate,PM,f,g}--\eqref{equ:Q+,bilinear estimate,PM,g,f} by taking
$$(p,q)=(\frac{6}{3+\ga},-\frac{6}{\ga}),\quad (p,q)=(-\frac{6}{\ga},\frac{6}{3+\ga}),$$
and obtain estimates \eqref{equ:Q+,bilinear estimate,L3,PM,f,g}--\eqref{equ:Q+,bilinear estimate,L3,PM,g,f} by taking
$$(p,q)=(3,6),\quad (p,q)=(6,3).$$

\end{proof}

To prove the scaling-invariant bilinear estimate for the gain term, we need a detailed frequency analysis from
Littlewood-Paley theory.\footnote{See \cite{chen2019derivation,chen2022unconditional,CSZ22} for some examples sharing similar critical flavor but carrying completely different structures.}
Let $\chi(x)$ be a cutoff function and satisfy $\chi(x)=1$ for $|x|\leq 1$ and $\chi(x)=0$ for $|x|\geq 2$. Let $N$ be a dyadic number and define the Littlewood-Paley projector
\begin{align}
\widehat{P_{N}u}(y)=\vp_{N}(y)\widehat{u}(y), \quad \vp_{N}(y)=\chi(\frac{y}{2N})-\chi(\frac{y}{N}).
\end{align}
We denote by $P_{N}^{x}$ and $P_{M}^{\xi}$ the projector of the $x$-variable and $\xi$-variable respectively.

\begin{lemma}[Scaling-invariant bilinear estimate]\label{lemma:bilinear estimate,gain term,gwp}
Let $s_{c}=\frac{1}{2}$, $\ga\in[-\frac{1}{2},0]$, $\al\geq 0$, $\be\geq 0$, $\ve\in[0,1]$. Then we have
\begin{align}\label{equ:bilinear estimate,gain term,gwp}
&\n{\lra{\ve\nabla_{x}}^{\al}\lra{\ve \nabla_{\xi}}^{\be}\wt{Q}^{+}(\wt{f},\wt{g})}_{L_{t}^{1}H_{x}^{s_{c}}H_{\xi}^{s_{c}+\ga}}\\
\leq& C\n{\lra{\ve\nabla_{x}}^{\al}\lra{\ve \nabla_{\xi}}^{\be}\wt{f}}_{U_{L}^{2}H_{x}^{s_{c}}H_{\xi}^{s_{c}+\ga}}\n{\lra{\ve\nabla_{x}}^{\al}
\lra{\ve\nabla_{\xi}}^{\be}\wt{g}}_{U_{L}^{2}H_{x}^{s_{c}}H_{\xi}^{s_{c}+\ga}},\notag
\end{align}
where the constant $C$ is independent of $\ve$. Moreover, if $\al\in[\frac{1}{2},1]$ and $\be\in [\frac{1}{2},1]$, the constant $C$ can also be independent of $\al$ and $\be$.
\end{lemma}
\begin{proof}
We divide the parameters into two cases as follows:
\begin{enumerate}[$(1)$]
\item $s_{c}+\ga+\be=0$.
\item $s_{c}+\ga+\be>0$.
\end{enumerate}

For the case $(1)$ that $\ga=-s_{c}=-\frac{1}{2}$ and $\be=0$, by triangle inequality, we have
\begin{align}\label{equ:bilinear estimate,gain term,case1,triangle}
\n{\lra{\ve\nabla_{x}}^{\al}\wt{Q}^{+}(\wt{f},\wt{g})}_{L_{t}^{1}H_{x}^{s_{c}}L_{\xi}^{2}}\lesssim
\n{\wt{Q}^{+}(\wt{f},\wt{g})}_{L_{t}^{1}H_{x}^{s_{c}}L_{\xi}^{2}}+
\n{|\ve \nabla_{x}|^{\al}\wt{Q}^{+}(\wt{f},\wt{g})}_{L_{t}^{1}H_{x}^{s_{c}}L_{\xi}^{2}}.
\end{align}
For the first term on the right hand side of \eqref{equ:bilinear estimate,gain term,case1,triangle},
we use the fractional Leibniz rule in Lemma \ref{lemma:generalized leibniz rule}, and estimates \eqref{equ:Q+,bilinear estimate,L3,PM,f,g}--\eqref{equ:Q+,bilinear estimate,L3,PM,g,f} in Lemma \ref{lemma:Q+,bilinear estimate} to obtain
\begin{align*}
&\n{\lra{\nabla_{x}}^{s_{c}}\wt{Q}^{+}(\wt{f},\wt{g})}_{L_{t}^{1}L_{x}^{2}L_{\xi}^{2}}\\
\lesssim& \bbn{ \wt{Q}^{+}\lrs{\n{\lra{\nabla_{x}}^{s_{c}}\wt{f}}_{L_{x}^{3}},\n{\wt{g}}_{L_{x}^{6}}}}_{L_{t}^{1}L_{\xi}^{2}}+\bbn{ \wt{Q}^{+}\lrs{\n{\wt{f}}_{L_{x}^{6}},\n{\lra{\nabla_{x}}^{s_{c}}\wt{g}}_{L_{x}^{3}}}}_{L_{t}^{1}L_{\xi}^{2}}\\
\lesssim& \n{\lra{\nabla_{x}}^{s_{c}}\wt{f}}_{L_{t}^{2}L_{\xi}^{3}L_{x}^{3}}\n{\wt{g}}_{L_{t}^{2}L_{\xi}^{3}L_{x}^{6}}+
\n{\wt{f}}_{L_{t}^{2}L_{\xi}^{3}L_{x}^{6}}\n{\lra{\nabla_{x}}^{s_{c}}\wt{g}}_{L_{t}^{2}L_{\xi}^{3}L_{x}^{3}}.
\end{align*}
By Sobolev inequality that $W^{s_{c},6}\hookrightarrow L^{3}$, and Strichartz estimate \eqref{equ:strichartz estimate,U}, we have
\begin{align}\label{equ:bilinear estimate,gain term,case1,L2}
\n{\lra{\nabla_{x}}^{s_{c}}\wt{Q}^{+}(\wt{f},\wt{g})}_{L_{t}^{1}L_{x}^{2}L_{\xi}^{2}}\lesssim& \n{\lra{\nabla_{x}}^{s_{c}}\wt{f}}_{L_{t}^{2}L_{\xi}^{3}L_{x}^{3}}\n{\lra{\nabla_{x}}^{s_{c}}\wt{g}}_{L_{t}^{2}L_{\xi}^{3}L_{x}^{3}}\\
\lesssim& \n{\lra{\nabla_{x}}^{s_{c}}\wt{f}}_{U_{L}^{2}L_{x}^{2}L_{\xi}^{2}}\n{\lra{\nabla_{x}}^{s_{c}}\wt{g}}_{U_{L}^{2}L_{x}^{2}L_{\xi}^{2}}.\notag
\end{align}
In the same way, for the second term on the right hand side of \eqref{equ:bilinear estimate,gain term,case1,triangle}, we also have
\begin{align}\label{equ:bilinear estimate,gain term,case1,H}
&\n{|\ve \nabla_{x}|^{\al}\wt{Q}^{+}(\wt{f},\wt{g})}_{L_{t}^{1}H_{x}^{s_{c}}L_{\xi}^{2}}\\
\lesssim& \ve^{\al}
\n{\lra{\nabla_{x}}^{s_{c}+\al}\wt{f}}_{U_{L}^{2}L_{x}^{2}L_{\xi}^{2}}
\n{\lra{\nabla_{x}}^{s_{c}}\wt{g}}_{U_{L}^{2}L_{x}^{2}L_{\xi}^{2}}+
\ve^{\al}\n{\lra{\nabla_{x}}^{s_{c}}\wt{f}}_{U_{L}^{2}L_{x}^{2}L_{\xi}^{2}}
\n{\lra{\nabla_{x}}^{s_{c}+\al}\wt{g}}_{U_{L}^{2}L_{x}^{2}L_{\xi}^{2}}\notag\\
\lesssim&\n{\lra{\ve\nabla_{x}}^{\al}\wt{f}}_{U_{L}^{2}H_{x}^{s_{c}}L_{\xi}^{2}}\n{\lra{\ve\nabla_{x}}^{\al}
\wt{g}}_{U_{L}^{2}H_{x}^{s_{c}}L_{\xi}^{2}}.\notag
\end{align}
Putting together estimates \eqref{equ:bilinear estimate,gain term,case1,triangle}, \eqref{equ:bilinear estimate,gain term,case1,L2}, and \eqref{equ:bilinear estimate,gain term,case1,H}, we complete the proof of \eqref{equ:bilinear estimate,gain term,gwp} for the case (1).

Next we deal with the case (2). By triangle inequality, we have
\begin{align}\label{equ:bilinear estimate,gain term,case2,triangle}
&\n{\lra{\ve\nabla_{x}}^{\al}\lra{\ve \nabla_{\xi}}^{\be}\wt{Q}^{+}(\wt{f},\wt{g})}_{L_{t}^{1}H_{x}^{s_{c}}H_{\xi}^{s_{c}+\ga}}\notag\\
\lesssim&
\n{\lra{\ve\nabla_{x}}^{\al}\wt{Q}^{+}(\wt{f},\wt{g})}_{L_{t}^{1}H_{x}^{s_{c}}H_{\xi}^{s_{c}+\ga}}+
\n{\lra{\ve\nabla_{x}}^{\al}|\ve\nabla_{\xi}|^{\be}\wt{Q}^{+}(\wt{f},\wt{g})}_{L_{t}^{1}H_{x}^{s_{c}}H_{\xi}^{s_{c}+\ga}}.
\end{align}
It suffices to handle the second term on the right hand of \eqref{equ:bilinear estimate,gain term,case2,triangle} and prove that
\begin{align}\label{equ:bilinear estimate,gain term,gwp,main case}
&\n{\lra{\ve\nabla_{x}}^{\al}|\ve \nabla_{\xi}|^{\be}\wt{Q}^{+}(\wt{f},\wt{g})}_{L_{t}^{1}H_{x}^{s_{c}}H_{\xi}^{s_{c}+\ga}}\\
\lesssim& \n{\lra{\ve\nabla_{x}}^{\al}\lra{\ve \nabla_{\xi}}^{\be}\wt{f}}_{U_{L}^{2}H_{x}^{s_{c}}H_{\xi}^{s_{c}+\ga}}\n{\lra{\ve\nabla_{x}}^{\al}
\lra{\ve\nabla_{\xi}}^{\be}\wt{g}}_{U_{L}^{2}H_{x}^{s_{c}}H_{\xi}^{s_{c}+\ga}},\notag
\end{align}
since the first term on the right hand of \eqref{equ:bilinear estimate,gain term,case2,triangle} can be estimated in the same way as \eqref{equ:bilinear estimate,gain term,case1,triangle} if $s_{c}+\ga=0$, or \eqref{equ:bilinear estimate,gain term,gwp,main case} with $\be=0$ if $s_{c}+\ga>0$.
To prove \eqref{equ:bilinear estimate,gain term,gwp,main case}, by duality it is equivalent to prove
\begin{align}\label{equ:bilinear estimate,gain term,gwp,duality}
&\int \wt{Q}^{+}(\wt{f},\wt{g}) h dx d\xi dt\\
\lesssim&
\n{\lra{\ve\nabla_{x}}^{\al}\lra{\ve\nabla_{\xi}}^{\be}\wt{f}}_{U_{L}^{2}H_{x}^{s_{c}}H_{\xi}^{s_{c}+\ga}}\n{\lra{\ve\nabla_{x}}^{\al}\lra{\ve\nabla_{\xi}}^{\be}\wt{g}}_{U_{L}^{2}H_{x}^{s_{c}}H_{\xi}^{s_{c}+\ga}}\notag\\
&\n{\lra{\ve\nabla_{x}}^{-\al}|\ve \nabla_{\xi}|^{-\be}h}_{L_{t}^{\infty}H_{x}^{-s_{c}}H_{\xi}^{-s_{c}-\ga}}.\notag
\end{align}

We denote by $I$ the integral in \eqref{equ:bilinear estimate,gain term,gwp,duality},
and insert a Littlewood-Paley decomposition so that
\begin{align*}
I=\sum_{\substack{M,M_{1},M_{2}\\N,N_{1},N_{2}}} I_{M,M_{1},M_{2},N,N_{1},N_{2}}
\end{align*}
where
\begin{align*}
I_{M,M_{1},M_{2},N,N_{1},N_{2}}= \int \wt{Q}^{+}(P_{N_{1}}^{x}P_{M_{1}}^{\xi}\wt{f},P_{N_{2}}^{x}P_{M_{2}}^{\xi}\wt{g}) P_{N}^{x}P_{M}^{\xi}h dx d\xi dt.
\end{align*}
Noticing that $\wt{Q}^{+}$ commutes with $P_{N}^{x}$, we have the constraint that $N\lesssim \max\lrs{N_{1},N_{2}}$ due to that
 \begin{align} \label{equ:property,constraint,projector}
 P_{N}^{x}(P_{N_{1}}^{x}\wt{f}P_{N_{2}}^{x}\wt{g})=0, \quad \text{if $N\geq 10\max\lrs{N_{1},N_{2}}$.}
 \end{align}
 One key observation is that the property \eqref{equ:property,constraint,projector} also holds for the $\xi$-variable, that is,
\begin{align}\label{equ:property,constraint,projector,v,variable}
P_{M}^{\xi}\wt{Q}^{+}(P_{M_{1}}^{\xi}\wt{f},P_{M_{2}}^{\xi}\wt{g})=0, \quad \text{if $M\geq 10\max\lrs{M_{1},M_{2}}$.}
\end{align}
Indeed, note that
\begin{align}\label{equ:fourier transform,decomposition,Q+}
\cf_{\xi}\lrs{P_{M}^{\xi}\wt{Q}^{+}(P_{M_{1}}^{\xi}\wt{f},P_{M_{2}}^{\xi}\wt{g})}=
\vp_{M}(v)\int_{\mathbb{S}^{2}}\int_{\mathbb{R}^{3}} (\vp_{M_{1}}f)(v^{\ast })(\vp_{M_{2}}g)(u^{\ast}) B(u-v,\omega)dud\omega.
\end{align}
By the energy conservation that $|v|^{2}+|u|^{2}=|v^{*}|^{2}+|u^{*}|^{2}$, we have $|v|\leq |v^{*}|+|u^{*}|$ for all $(u,\omega)\in \R^{3}\times \mathbb{S}^{2}$. Together with $M\geq 10\max\lrs{M_{1},M_{2}}$, this lower bound implies that
\begin{align}
\vp_{M}(v)\vp_{M_{1}}(v^{*})\vp_{M_{2}}(u^{*})=0, \quad \text{for all $(u,\omega)\in \R^{3}\times \mathbb{S}^{2}$}.
\end{align}
Thus, we arrive at \eqref{equ:property,constraint,projector,v,variable} and the constraint that $M\lesssim \max\lrs{M_{1},M_{2}}$.

 Now, we divide the sum into four cases as follows:

Case A. $M_{1}\geq M_{2}$, $N_{1}\geq N_{2}$.

Case B. $M_{1}\leq M_{2}$, $N_{1}\geq N_{2}$.

Case C. $M_{1}\geq M_{2}$, $N_{1}\leq N_{2}$.

Case D. $M_{1}\leq  M_{2}$, $N_{1}\leq N_{2}$.

We only need to treat Cases A and B, as Cases C and D follow similarly.

\textbf{Case A. $M_{1}\geq M_{2}$, $N_{1}\geq N_{2}$.}

Let $I_{A}$ denote the integral restricted to the Case A.

\begin{align*}
I_{A}=&\sum_{\substack{M_{1}\geq M_{2},M_{1}\gtrsim M\\N_{1}\geq N_{2},N_{1}\gtrsim N }}\int \wt{Q}^{+}(P_{N_{1}}^{x}P_{M_{1}}^{\xi}\wt{f},P_{N_{2}}^{x}P_{M_{2}}^{\xi}\wt{g}) P_{N}^{x}P_{M}^{\xi}h dx d\xi dt\\
=&\sum_{\substack{M_{1}\geq M\\N_{1}\gtrsim N }}\int \wt{Q}^{+}(P_{N_{1}}^{x}P_{M_{1}}^{\xi}\wt{f},P_{\leq N_{1}}^{x}P_{\leq M_{1}}^{\xi}\wt{g}) P_{N}^{x}P_{M}^{\xi}h dx d\xi dt,
\end{align*}
where in the last equality we have done the sum in $M_{2}$ and $N_{2}$. Using H\"{o}lder inequality and estimate \eqref{equ:Q+,bilinear estimate,L3,PM,f,g} in Lemma \ref{lemma:Q+,bilinear estimate}, we have
\begin{align*}
I_{A}\leq& \int \sum_{\substack{M_{1}\gtrsim M\\N_{1}\gtrsim N }} \n{\wt{Q}^{+}(P_{N_{1}}^{x}P_{M_{1}}^{\xi}\wt{f},P_{\leq N_{1}}^{x}P_{\leq M_{1}}^{\xi}\wt{g})}_{L_{\xi}^{2}} \n{P_{N}^{x}P_{M}^{\xi}h}_{L_{\xi}^{2}} dx  dt\\
\lesssim& \int \sum_{\substack{M_{1}\gtrsim M\\N_{1}\gtrsim N }} \n{P_{N_{1}}^{x}P_{M_{1}}^{\xi}\wt{f}}_{L_{\xi}^{3}}\n{P_{\leq N_{1}}^{x}P_{\leq M_{1}}^{\xi}\wt{g}}_{L_{\xi}^{\frac{6}{1-2\ga}}} \n{P_{N}^{x}P_{M}^{\xi}h}_{L_{\xi}^{2}} dx  dt.
\end{align*}
By Bernstein inequality and Sobolev inequality that $W^{s_{c}+\ga,3}\hookrightarrow L^{\frac{6}{1-2\ga}}$,
\begin{align*}
I_{A}\leq& \int \sum_{\substack{M_{1}\gtrsim M\\N_{1}\gtrsim N }} \frac{M^{s_{c}+\ga} |\ve M|^{\be}}{M_{1}^{s_{c}+\ga}|\ve M_{1}|^{\be}}
\n{\lra{\nabla_{\xi}}^{s_{c}+\ga}\lra{\ve\nabla_{\xi}}^{\be}P_{N_{1}}^{x}P_{M_{1}}^{\xi}\wt{f}}_{L_{\xi}^{3}}\\
&\n{P_{\leq N_{1}}^{x}\lra{\nabla_{\xi}}^{s_{c}+\ga}\wt{g}}_{L_{\xi}^{3}} \n{\lra{\nabla_{\xi}}^{-s_{c}-\ga}|\ve\nabla_{\xi}|^{-\be}P_{N}^{x}P_{M}^{\xi}h}_{L_{\xi}^{2}} dx  dt.
\end{align*}
With $s_{c}+\ga+\be >0$, we use Cauchy-Schwarz in $M$ and $M_{1}$ to get
\begin{align*}
I_{A}\lesssim& \int \sum_{N_{1}\gtrsim N}\n{P_{\leq N_{1}}^{x}\lra{\nabla_{\xi}}^{s_{c}+\ga}\wt{g}}_{L_{\xi}^{3}}
\lrs{\sum_{M_{1}\gtrsim M}\frac{M^{s_{c}+\ga+\be}}{M_{1}^{s_{c}+\ga+\be}}
\n{\lra{\nabla_{\xi}}^{s_{c}+\ga}\lra{\ve\nabla_{\xi}}^{\be}P_{N_{1}}^{x}P_{M_{1}}^{\xi}\wt{f}}_{L_{\xi}^{3}}^{2}}^{\frac{1}{2}}\\
&\lrs{\sum_{M_{1}\gtrsim M}\frac{M^{s_{c}+\ga+\be}}{M_{1}^{s_{c}+\ga+\be}}
\n{\lra{\nabla_{\xi}}^{-s_{c}-\ga}|\ve\nabla_{\xi}|^{-\be}P_{N}^{x}P_{M}^{\xi}h}_{L_{\xi}^{2}}^{2}}^{\frac{1}{2}}dxdt\\
\lesssim &\int \sum_{N_{1}\gtrsim N}\n{P_{\leq N_{1}}^{x}\lra{\nabla_{\xi}}^{s_{c}+\ga}\wt{g}}_{L_{\xi}^{3}}
\n{\lra{\nabla_{\xi}}^{s_{c}+\ga}\lra{\ve\nabla_{\xi}}^{\be}P_{N_{1}}^{x}P_{M_{1}}^{\xi}\wt{f}}_{l_{M_{1}}^{2}L_{\xi}^{3}}\\
&\n{\lra{\nabla_{\xi}}^{-s_{c}-\ga}|\ve\nabla_{\xi}|^{-\be}P_{N}^{x}h}_{L_{\xi}^{2}} dxdt.
\end{align*}
By H\"{o}lder inequality in the $x$-variable,
\begin{align*}
I_{A}\lesssim &
\int \sum_{N_{1}\gtrsim N}\n{P_{\leq N_{1}}^{x}\lra{\nabla_{\xi}}^{s_{c}+\ga}\wt{g}}_{L_{x}^{6}L_{\xi}^{3}}
\n{\lra{\nabla_{\xi}}^{s_{c}+\ga}\lra{\ve\nabla_{\xi}}^{\be}P_{N_{1}}^{x}P_{M_{1}}^{\xi}\wt{f}}_{L_{x}^{3}l_{M_{1}}^{2}L_{\xi}^{3}}\\
&\n{\lra{\nabla_{\xi}}^{-s_{c}-\ga}|\ve\nabla_{\xi}|^{-\be}P_{N}^{x}h}_{L_{x}^{2}L_{\xi}^{2}} dt.
\end{align*}
By Minkowski inequality, Sobolev inequality that $W^{s_{c},3}\hookrightarrow L^{6}$, and Bernstein inequality,
\begin{align*}
I_{A}\lesssim &
\int \sum_{N_{1}\gtrsim N}\n{P_{\leq N_{1}}^{x}\lra{\nabla_{x}}^{s_{c}}\lra{\nabla_{\xi}}^{s_{c}+\ga}\wt{g}}_{L_{x}^{3}L_{\xi}^{3}}
\n{\lra{\nabla_{\xi}}^{s_{c}+\ga}\lra{\ve\nabla_{\xi}}^{\be}P_{N_{1}}^{x}P_{M_{1}}^{\xi}\wt{f}}_{l_{M_{1}}^{2}L_{x}^{3}L_{\xi}^{3}}\\
&\n{\lra{\nabla_{\xi}}^{-s_{c}-\ga}|\ve\nabla_{\xi}|^{-\be}P_{N}^{x}h}_{L_{x}^{2}L_{\xi}^{2}} dt\\
\lesssim &
\int \n{\lra{\nabla_{x}}^{s_{c}}\lra{\nabla_{\xi}}^{s_{c}+\ga}\wt{g}}_{L_{x}^{3}L_{\xi}^{3}}
\sum_{N_{1}\gtrsim N} \frac{N^{s_{c}}\lra{\ve N}^{\al}}{N_{1}^{s_{c}}\lra{\ve N_{1}}^{\al}}\\
&\n{\lra{\nabla_{x}}^{s_{c}}\lra{\ve \nabla_{x}}^{\al}\lra{\nabla_{\xi}}^{s_{c}+\ga}\lra{\ve\nabla_{\xi}}^{\be}P_{N_{1}}^{x}P_{M_{1}}^{\xi}\wt{f}}_{l_{M_{1}}^{2}L_{x}^{3}L_{\xi}^{3}}\\
&\n{\lra{\nabla_{x}}^{-s_{c}}\lra{\ve \nabla_{x}}^{-\al}\lra{\nabla_{\xi}}^{-s_{c}-\ga}|\ve\nabla_{\xi}|^{-\be}P_{N}^{x}h}_{L_{x}^{2}L_{\xi}^{2}} dt.
\end{align*}
With $s_{c}=\frac{1}{2}>0$, we use Cauchy-Schwarz in $N$ and $N_{1}$ to get
\begin{align*}
I_{A}\lesssim &
\int \n{\lra{\nabla_{x}}^{s_{c}}\lra{\nabla_{\xi}}^{s_{c}+\ga}\wt{g}}_{L_{x}^{3}L_{\xi}^{3}}\n{\lra{\nabla_{x}}^{s_{c}}\lra{\ve\nabla_{x}}^{\al}\lra{\nabla_{\xi}}^{s_{c}+\ga}\lra{\ve\nabla_{\xi}}^{\be}
P_{N_{1}}^{x}P_{M_{1}}^{\xi}\wt{f}}_{l_{N_{1}}^{2}l_{M_{1}}^{2}L_{x}^{3}L_{\xi}^{3}}\\
&\n{\lra{\nabla_{x}}^{-s_{c}}\lra{\ve \nabla_{x}}^{-\al}\lra{\nabla_{\xi}}^{-s_{c}-\ga}|\ve\nabla_{\xi}|^{-\be}h}_{L_{x}^{2}L_{\xi}^{2}} dt.
\end{align*}
By H\"{o}lder inequality in the $t$-variable,
\begin{align*}
I_{A}\lesssim &\n{\lra{\nabla_{x}}^{s_{c}}\lra{\nabla_{\xi}}^{s_{c}+\ga}\wt{g}}_{L_{t}^{2}L_{x}^{3}L_{\xi}^{3}}\n{\lra{\nabla_{x}}^{s_{c}}\lra{\ve\nabla_{x}}^{\al}\lra{\nabla_{\xi}}^{s_{c}+\ga}\lra{\ve\nabla_{\xi}}^{\be}
P_{N_{1}}^{x}P_{M_{1}}^{\xi}\wt{f}}_{L_{t}^{2}l_{N_{1}}^{2}l_{M_{1}}^{2}L_{x}^{3}L_{\xi}^{3}}\\
&\n{\lra{\nabla_{x}}^{-s_{c}}\lra{\ve \nabla_{x}}^{\al}\lra{\nabla_{\xi}}^{-s_{c}-\ga}|\ve\nabla_{\xi}|^{-\be}h}_{L_{t}^{\infty}L_{x}^{2}L_{\xi}^{2}}.
\end{align*}
We use Strichartz estimate \eqref{equ:strichartz estimate,U} to obtain
\begin{align*}
I_{A}\lesssim & \n{\wt{g}}_{U_{L}^{2}H_{x}^{s_{c}}H_{\xi}^{s_{c}+\ga}}
\n{\lra{\ve\nabla_{x}}^{\al}\lra{\ve\nabla_{\xi}}^{\be}
P_{N_{1}}^{x}P_{M_{1}}^{\xi}\wt{f}}_{l_{N_{1}}^{2}l_{M_{1}}^{2}U_{L}^{2}H_{x}^{s_{c}}H_{\xi}^{s_{c}+\ga}}\\
&\n{\lra{\nabla_{x}}^{-s_{c}}\lra{\ve \nabla_{x}}^{\al}\lra{\nabla_{\xi}}^{-s_{c}-\ga}|\ve\nabla_{\xi}|^{-\be}h}_{L_{t}^{\infty}L_{x}^{2}L_{\xi}^{2}}.
\end{align*}
By Minkowski inequality for the atomic $U$ space (see \cite[Lemma 4.25]{koch14dispersive}), we have
\begin{align*}
\n{\lra{\ve\nabla_{x}}^{\al}\lra{\ve\nabla_{\xi}}^{\be}
P_{N_{1}}^{x}P_{M_{1}}^{\xi}\wt{f}}_{l_{N_{1}}^{2}l_{M_{1}}^{2}U_{L}^{2}H_{x}^{s_{c}}H_{\xi}^{s_{c}+\ga}}\leq & \n{\lra{\ve\nabla_{x}}^{\al}\lra{\ve\nabla_{\xi}}^{\be}
P_{N_{1}}^{x}P_{M_{1}}^{\xi}\wt{f}}_{U_{L}^{2}l_{N_{1}}^{2}l_{M_{1}}^{2}H_{x}^{s_{c}}H_{\xi}^{s_{c}+\ga}}\\
\lesssim& \n{\lra{\ve\nabla_{x}}^{\al}\lra{\ve\nabla_{\xi}}^{\be}\wt{f}}_{U_{L}^{2}H_{x}^{s_{c}}H_{\xi}^{s_{c}+\ga}}.
\end{align*}
Hence, we arrive at
\begin{align*}
I_{A}\lesssim & \n{\lra{\ve\nabla_{x}}^{\al}\lra{\ve\nabla_{\xi}}^{\be}\wt{g}}_{U_{L}^{2}H_{x}^{s_{c}}H_{\xi}^{s_{c}+\ga}}
\n{\lra{\ve\nabla_{x}}^{\al}\lra{\ve\nabla_{\xi}}^{\be}
\wt{f}}_{U_{L}^{2}H_{x}^{s_{c}}H_{\xi}^{s_{c}+\ga}}\\
&\n{\lra{\nabla_{x}}^{-s_{c}}\lra{\ve \nabla_{x}}^{-\al}\lra{\nabla_{\xi}}^{-s_{c}-\ga}|\ve\nabla_{\xi}|^{-\be}h}_{L_{t}^{\infty}L_{x}^{2}L_{\xi}^{2}},
\end{align*}
which completes the proof of \eqref{equ:bilinear estimate,gain term,gwp,duality} for Case A.

\textbf{Case B. $M_{1}\leq M_{2}$, $N_{1}\geq N_{2}$.}

Let $I_{B}$ denote the integral restricted to the Case B.

\begin{align*}
I_{B}=&\sum_{\substack{M_{2}\geq M_{1},M_{2}\gtrsim M\\N_{1}\geq N_{2},N_{1}\gtrsim N }}\int \wt{Q}^{+}(P_{N_{1}}^{x}P_{M_{1}}^{\xi}\wt{f},P_{N_{2}}^{x}P_{M_{2}}^{\xi}\wt{g}) P_{N}^{x}P_{M}^{\xi}h dx d\xi dt\\
=&\sum_{\substack{M_{2}\gtrsim M\\N_{1}\gtrsim N }}\int \wt{Q}^{+}(P_{N_{1}}^{x}P_{\leq M_{2}}^{\xi}\wt{f},P_{\leq N_{1}}^{x}P_{M_{2}}^{\xi}\wt{g}) P_{N}^{x}P_{M}^{\xi}h dx d\xi dt
\end{align*}
where we have done the sum in $M_{1}$ and $N_{2}$. By using H\"{o}lder inequality and then estimate \eqref{equ:Q+,bilinear estimate,L3,PM,g,f} in Lemma \ref{equ:Q+,bilinear estimate}, we have
\begin{align*}
I_{B}\lesssim& \int \sum_{\substack{M_{2}\gtrsim M\\N_{1}\gtrsim N}}\n{\wt{Q}^{+}(P_{N_{1}}^{x}P_{\leq M_{2}}^{\xi}\wt{f},P_{\leq N_{1}}^{x}P_{M_{2}}^{\xi}\wt{g})}_{L_{\xi}^{2}} \n{P_{N}^{x}P_{M}^{\xi}h}_{L_{\xi}^{2}} dx dt\\
\lesssim&\int \sum_{\substack{M_{2}\gtrsim M\\N_{1}\gtrsim N}}\n{P_{N_{1}}^{x}P_{\leq M_{2}}^{\xi}\wt{f}}_{L_{\xi}^{\frac{6}{1-2\ga}}}\n{P_{\leq N_{1}}^{x}P_{M_{2}}^{\xi}\wt{g}}_{L_{\xi}^{3}} \n{P_{N}^{x}P_{M}^{\xi}h}_{L_{\xi}^{2}} dx dt.
\end{align*}
By Bernstein inequality and Sobolev inequality that $W^{s_{c}+\ga,3}\hookrightarrow L^{\frac{6}{1-2\ga}}$,
\begin{align*}
I_{B}\lesssim& \int \sum_{\substack{M_{2}\gtrsim M\\N_{1}\gtrsim N }} \frac{M^{s_{c}+\ga}|\ve M|^{\be}}{M_{2}^{s_{c}+\ga}|\ve M_{2}|^{\be}}
\n{\lra{\nabla_{\xi}}^{s_{c}+\ga}P_{N_{1}}^{x}\wt{f}}_{L_{\xi}^{3}}\\
&\n{\lra{\nabla_{\xi}}^{s_{c}+\ga}\lra{\ve\nabla_{\xi}}^{\be}P_{\leq N_{1}}^{x}P_{M_{2}}^{\xi}\wt{g}}_{L_{\xi}^{3}} \n{\lra{\nabla_{\xi}}^{-s_{c}-\ga}|\ve\nabla_{\xi}|^{-\be}P_{N}^{x}P_{M}^{\xi}h}_{L_{\xi}^{2}} dx  dt.
\end{align*}
With $s_{c}+\ga+\be>0$, we use Cauchy-Schwarz in $M$ and $M_{2}$ to obtain
\begin{align*}
I_{B}\lesssim& \int \sum_{N_{1}\gtrsim N}
\n{\lra{\nabla_{\xi}}^{s_{c}+\ga}P_{N_{1}}^{x}\wt{f}}_{L_{\xi}^{3}}
\lrs{\sum_{M_{2}\gtrsim M}\frac{M^{s_{c}+\ga+\be}}{M_{2}^{s_{c}+\ga+\be}}
\n{\lra{\nabla_{\xi}}^{s_{c}+\ga}\lra{\ve\nabla_{\xi}}^{\be}P_{\leq N_{1}}^{x}P_{M_{2}}^{\xi}\wt{g}}_{L_{\xi}^{3}}^{2}}^{\frac{1}{2}}\\
&\lrs{\sum_{M_{2}\gtrsim M}\frac{M^{s_{c}+\ga+\be}}{M_{2}^{s_{c}+\ga+\be}}
\n{\lra{\nabla_{\xi}}^{-s_{c}-\ga}|\ve\nabla_{\xi}|^{-\be}P_{N}^{x}P_{M}^{\xi}h}_{L_{\xi}^{2}}^{2}}^{\frac{1}{2}}dxdt\\
\lesssim &\int \sum_{N_{1}\gtrsim N}
\n{\lra{\nabla_{\xi}}^{s_{c}+\ga}P_{N_{1}}^{x}\wt{f}}_{L_{\xi}^{3}}
\n{\lra{\nabla_{\xi}}^{s_{c}+\ga}\lra{\ve\nabla_{\xi}}^{\be}P_{\leq N_{1}}^{x}P_{M_{2}}^{\xi}\wt{g}}_{l_{M_{2}}^{2}L_{\xi}^{3}}\\
&\n{\lra{\nabla_{\xi}}^{-s_{c}-\ga}|\ve\nabla_{\xi}|^{-\be}P_{N}^{x}h}_{L_{\xi}^{2}} dxdt.
\end{align*}
By H\"{o}lder inequality in the $x$-variable,
\begin{align*}
I_{B}\lesssim &
\int \sum_{N_{1}\gtrsim N}
\n{\lra{\nabla_{\xi}}^{s_{c}+\ga}P_{N_{1}}^{x}\wt{f}}_{L_{x}^{6}L_{\xi}^{3}}
\n{\lra{\nabla_{\xi}}^{s_{c}+\ga}\lra{\ve\nabla_{\xi}}^{\be}P_{\leq N_{1}}^{x}P_{M_{2}}^{\xi}\wt{g}}_{L_{x}^{3}l_{M_{2}}^{2}L_{\xi}^{3}}\\
&\n{\lra{\nabla_{\xi}}^{-s_{c}-\ga}|\ve\nabla_{\xi}|^{-\be}P_{N}^{x}h}_{L_{x}^{2}L_{\xi}^{2}} dt.
\end{align*}
By Minkowski inequality, Sobolev inequality that $W^{s_{c},3}\hookrightarrow L^{6}$, and Bernstein inequality,
\begin{align*}
I_{B}\lesssim &
\int \sum_{N_{1}\gtrsim N}
\n{\lra{\nabla_{x}}^{s_{c}}\lra{\nabla_{\xi}}^{s_{c}+\ga}P_{N_{1}}^{x}\wt{f}}_{L_{x}^{3}L_{\xi}^{3}}
\n{\lra{\nabla_{\xi}}^{s_{c}+\ga}\lra{\ve\nabla_{\xi}}^{\be}P_{\leq N_{1}}^{x}P_{M_{2}}^{\xi}\wt{g}}_{l_{M_{2}}^{2}L_{x}^{3}L_{\xi}^{3}}\\
&\n{\lra{\nabla_{\xi}}^{-s_{c}-\ga}|\ve\nabla_{\xi}|^{-\be}P_{N}^{x}h}_{L_{x}^{2}L_{\xi}^{2}} dt\\
\lesssim &
\int \n{\lra{\nabla_{x}}^{s_{c}}\lra{\nabla_{\xi}}^{s_{c}+\ga}\wt{f}}_{L_{x}^{3}L_{\xi}^{3}}
\sum_{N_{1}\gtrsim N} \frac{N^{s_{c}}\lra{\ve N}^{\al}}{N_{1}^{s_{c}}\lra{\ve N_{1}}^{\al}}\\
&\n{\lra{\nabla_{x}}^{s_{c}}\lra{\ve \nabla_{x}}^{\al}\lra{\nabla_{\xi}}^{s_{c}+\ga}\lra{\ve\nabla_{\xi}}^{\be}P_{N_{1}}^{x}
P_{M_{2}}^{\xi}\wt{g}}_{l_{M_{2}}^{2}L_{x}^{3}L_{\xi}^{3}}\\
&\n{\lra{\nabla_{x}}^{-s_{c}}\lra{\ve \nabla_{x}}^{-\al}\lra{\nabla_{\xi}}^{-s_{c}-\ga}|\ve\nabla_{\xi}|^{-\be}P_{N}^{x}h}_{L_{x}^{2}L_{\xi}^{2}} dt.
\end{align*}
With $s_{c}=\frac{1}{2}>0$, we use Cauchy-Schwarz in $N$ and $N_{1}$ to get
\begin{align*}
I_{B}\lesssim &
\int \n{\lra{\nabla_{x}}^{s_{c}}\lra{\nabla_{\xi}}^{s_{c}+\ga}\wt{f}}_{L_{x}^{3}L_{\xi}^{3}}\n{\lra{\nabla_{x}}^{s_{c}}\lra{\ve\nabla_{x}}^{\al}\lra{\nabla_{\xi}}^{s_{c}+\ga}\lra{\ve\nabla_{\xi}}^{\be}
P_{N_{1}}^{x}P_{M_{2}}^{\xi}\wt{g}}_{l_{N_{1}}^{2}l_{M_{2}}^{2}L_{x}^{3}L_{\xi}^{3}}\\
&\n{\lra{\nabla_{x}}^{-s_{c}}\lra{\ve \nabla_{x}}^{\al}\lra{\nabla_{\xi}}^{-s_{c}-\ga}|\ve\nabla_{\xi}|^{-\be}h}_{L_{x}^{2}L_{\xi}^{2}} dt.
\end{align*}
By H\"{o}lder inequality in the $t$-variable,
\begin{align*}
I_{B}\lesssim &\n{\lra{\nabla_{x}}^{s_{c}}\lra{\nabla_{\xi}}^{s_{c}+\ga}\wt{f}}_{L_{t}^{2}L_{x}^{3}L_{\xi}^{3}}\n{\lra{\nabla_{x}}^{s_{c}}\lra{\ve\nabla_{x}}^{\al}\lra{\nabla_{\xi}}^{s_{c}+\ga}\lra{\ve\nabla_{\xi}}^{\be}
P_{N_{1}}^{x}P_{M_{2}}^{\xi}\wt{g}}_{L_{t}^{2}l_{N_{1}}^{2}l_{M_{2}}^{2}L_{x}^{3}L_{\xi}^{3}}\\
&\n{\lra{\nabla_{x}}^{-s_{c}}\lra{\ve \nabla_{x}}^{\al}\lra{\nabla_{\xi}}^{-s_{c}-\ga}|\ve\nabla_{\xi}|^{-\be}h}_{L_{t}^{\infty}L_{x}^{2}L_{\xi}^{2}}.
\end{align*}
By Strichartz estimate \eqref{equ:strichartz estimate,U} and Minkowski inequality for the atomic $U$ space, we arrive at
\begin{align*}
I_{B}\lesssim & \n{\wt{f}}_{U_{L}^{2}H_{x}^{s_{c}}H_{\xi}^{s_{c}+\ga}}
\n{\lra{\ve\nabla_{x}}^{\al}\lra{\ve\nabla_{\xi}}^{\be}
P_{N_{1}}^{x}P_{M_{2}}^{\xi}\wt{g}}_{l_{N_{1}}^{2}l_{M_{2}}^{2}U_{L}^{2}H_{x}^{s_{c}}H_{\xi}^{s_{c}+\ga}}\\
&\n{\lra{\nabla_{x}}^{-s_{c}}\lra{\ve \nabla_{x}}^{\al}\lra{\nabla_{\xi}}^{-s_{c}-\ga}|\ve\nabla_{\xi}|^{-\be}h}_{L_{t}^{\infty}L_{x}^{2}L_{\xi}^{2}}\\
\lesssim &\n{\lra{\ve\nabla_{x}}^{\al}\lra{\ve\nabla_{\xi}}^{\be}\wt{f}}_{U_{L}^{2}H_{x}^{s_{c}}H_{\xi}^{s_{c}+\ga}}
\n{\lra{\ve\nabla_{x}}^{\al}\lra{\ve\nabla_{\xi}}^{\be}
\wt{g}}_{U_{L}^{2}H_{x}^{s_{c}}H_{\xi}^{s_{c}+\ga}}\\
&\n{\lra{\nabla_{x}}^{-s_{c}}\lra{\ve \nabla_{x}}^{\al}\lra{\nabla_{\xi}}^{-s_{c}-\ga}|\ve\nabla_{\xi}|^{-\be}h}_{L_{t}^{\infty}L_{x}^{2}L_{\xi}^{2}},
\end{align*}
which completes the proof of \eqref{equ:bilinear estimate,gain term,gwp,duality} for Case B. Therefore, we have done the proof of \eqref{equ:bilinear estimate,gain term,gwp}.

During the entire proof, the constant $C$ in \eqref{equ:bilinear estimate,gain term,gwp} is also independent of the parameters $\al$ and $\be$ if they are restricted to the finite interval that
$\al\in[\frac{1}{2},1]$ and $\be\in [\frac{1}{2},1]$.
\end{proof}

\subsection{Global Well-posedness for the Gain-term-only Boltzmann}\label{section:Small Data Global Well-posedness for Gain-term-only Boltzmann}
In this section, for small initial data in the scaling-critical space, we prove the global well-posedeness for the gain-term-only Boltzmann equation
\begin{equation}\label{equ:gain term,Boltzmann}
\left\{
\begin{aligned}
\left( \partial_t + v \cdot \nabla_x \right) f^{+} (t,x,v) =& Q^{+}(f^{+},f^{+}),\\
f^{+}(0,x,v)=& f_{0}(x,v).
\end{aligned}
\right.
\end{equation}
\begin{proposition}\label{lemma:gwp,gain-only,small data}
Let $s\in(1,\frac{3}{2})$ and $\ga\in[-\frac{1}{2},0]$.
There exists $\eta>0$, such that for all non-negative initial data $f_{0}\in L_{v}^{2,s+\ga}H_{x}^{s}$ satisfying
\begin{align}\label{equ:gwp,gain-only,smallness}
\n{\lra{\nabla_{x}}^{\frac{1}{2}}\lra{v}^{\frac{1}{2}+\ga}f_{0}}_{L_{x,v}^{2}}\leq \eta,
\end{align}
there exists a unique non-negative global solution $f^{+}$ to the gain-term-only Boltzmann equation \eqref{equ:gain term,Boltzmann} satisfying
$$f^{+}\in C([0,\infty);L_{v}^{2,s+\ga}H_{x}^{s}),\quad Q^{+}(f^{+},f^{+})\in L_{t,loc}^{1}(0,\infty;L_{v}^{2,s+\ga}H_{x}^{s}).$$
Moreover, for this solution, it holds that
\begin{align}
\n{\lra{\nabla_{x}}^{s}\lra{v}^{s+\ga}f^{+}}_{L_{t}^{\infty}(0,\infty;L_{v}^{2}L_{x}^{2})}<\infty,\label{equ:f+,regularity bound}\\
\n{\lra{\nabla_{x}}^{s}\lra{v}^{s+\ga}Q^{+}(f^{+},f^{+})}_{L_{t}^{1}(0,\infty;L_{v}^{2}L_{x}^{2})}<\infty.\label{equ:Q+,f+,regularity bound}
\end{align}

In particular, for $p\in [2,\frac{6}{3-2s}]$, we also have the integrability bounds
\begin{align}
\n{\lra{v}^{s+\ga}f^{+}}_{L_{t}^{\infty}(0,\infty;L_{v}^{2}L_{x}^{p})}<&\infty,\label{equ:f+,Lx6}\\
\n{\lra{v}^{s+\ga}Q^{+}(f^{+},f^{+})}_{L_{t}^{1}(0,\infty;L_{v}^{2}L_{x}^{p})}<&\infty,\label{equ:Q+,f+,Lx6}\\
\n{A\lrc{f^{+}}}_{L_{t}^{2}(0,\infty;L_{x}^{\infty}L_{v}^{\infty})}<&\infty,\label{equ:A,f+,Lxinfty}\\
\n{\lra{v}^{s+\ga}Q^{-}(f^{+},f^{+})}_{L_{t}^{1}(0,T;L_{v}^{2}L_{x}^{p})}\leq& C_{T},\label{equ:Q-,f+,Lx6} \quad \text{for $T\in [0,\infty)$.}
\end{align}

\end{proposition}
\begin{proof}
By Plancherel identity, it is equivalent to prove the well-posedness for the gain-term-only Boltzmann equation on the $(x,\xi)$ side, that is,
\begin{equation}\label{equ:gain term,Boltzmann,fourier}
\left\{
\begin{aligned}
\pa_{t}\wt{f}^{+}-i\nabla_{\xi}\cdot \nabla_{x}\wt{f}^{+}=&\wt{Q}(\wt{f}^{+},\wt{f}^{+}),\\
\wt{f}^{+}(0,x,\xi)=& \wt{f}_{0}(x,\xi).
\end{aligned}
\right.
\end{equation}

For simplicity, we take the notation
\begin{align*}
\n{\wt{f}^{+}}_{X^{s,\ve}}=:\n{\lra{\nabla_{x}}^{s_{c}}\lra{\ve\nabla_{x}}^{s-s_{c}}\lra{\nabla_{\xi}}^{s_{c}+\ga}\lra{\ve \nabla_{\xi}}^{s-s_{c}}\wt{f}}_{L_{x,v}^{2}},
\end{align*}
where $\ve$ is to be determined. Our goal is to prove the well-posedness in the $X^{s,\ve}$ space.
Let
$$B=\lr{\wt{f}^{+}:\n{\wt{f}^{+}}_{U_{L}^{2}X^{s,\ve}}\leq 2\n{\wt{f}_{0}}_{X^{s,\ve}}}$$
and define the nonlinear map
\begin{align}\label{equ:nonlinear map,gwp}
\Phi(\wt{f}^{+})=U(t)\wt{f}_{0}+\int_{0}^{t}U(t-\tau)\wt{Q}^{+}(\wt{f}^{+},\wt{f}^{+})(\tau)d\tau.
\end{align}
By the contraction mapping principle, to prove the existence of a unique solution in $B$, we need to prove that $\Phi:B\mapsto B$ and the map $\Phi$ is a contraction.

First, by the definition of the atomic $U_{L}^{2}$ space, we have
\begin{align}\label{equ:closed estimate,gwp,Linfty}
\n{\Phi(\wt{f}^{+})}_{U_{L}^{2}X^{s,\ve}}\leq & \n{\wt{f}_{0}}_{X^{s,\ve}}+\bbn{\int_{0}^{t}U(t-\tau)\wt{Q}^{+}(\wt{f}^{+},\wt{f}^{+})(\tau)d\tau}_{U_{L}^{2}X^{s,\ve}}.
\end{align}
By estimate \eqref{equ:strichartz estimate,duhamel,U} in Proposition \ref{lemma:strichartz estimate,U-V} and
the scaling-invariant estimate \eqref{equ:bilinear estimate,gain term,gwp} in Lemma \ref{lemma:bilinear estimate,gain term,gwp}, we get
\begin{align}\label{equ:closed estimate,gwp,Linfty,duhamel}
\bbn{\int_{0}^{t}U(t-\tau)\wt{Q}^{+}(\wt{f}^{+},\wt{f}^{+})(\tau)d\tau}_{U_{L}^{2}X^{s,\ve}}\lesssim
\bn{Q^{+}(\wt{f}^{+},\wt{f}^{+})}_{L_{t}^{1}X^{s,\ve}}
\lesssim& \n{\wt{f}^{+}}_{U_{L}^{2}X^{s,\ve}}\n{\wt{f}^{+}}_{U_{L}^{2}X^{s,\ve}}.
\end{align}
Combining estimates \eqref{equ:closed estimate,gwp,Linfty} and \eqref{equ:closed estimate,gwp,Linfty,duhamel}, we reach
\begin{align*}
\n{\Phi(\wt{f}^{+})}_{U_{L}^{2}X^{s,\ve}}\leq \n{\wt{f}_{0}}_{X^{s,\ve}}+C\n{\wt{f}_{0}}_{X^{s,\ve}}^{2},
\end{align*}
where the constant $C$ is independent of $\ve$ and $s$.
By Plancherel identity, we have that
\begin{align*}
\n{\wt{f}_{0}}_{X^{s,\ve}}=
\n{\lra{y}^{s_{c}}\lra{\ve y}^{s-s_{c}}\lra{v}^{s_{c}+\ga}\lra{\ve v}^{s-s_{c}}\mathcal{F}_{x\mapsto y}(f_{0})}_{L_{y,v}^{2}}.
\end{align*}
Since $f_{0}\in L_{v}^{2,s+\ga}H_{x}^{s}$, by the dominated convergence theorem, we can choose $\ve$ sufficiently small to obtain
\begin{align}\label{equ:gain-only,gwp,smallness}
\n{\wt{f}_{0}}_{X^{s,\ve}}\leq 2 \n{\lra{\nabla_{x}}^{s_{c}}\lra{v}^{s_{c}+\ga}f_{0}}_{L_{x,v}^{2}}\leq 2\eta.
\end{align}
Then we choose $\eta$ (only depending on the constant C) small such that $2C\eta\leq \frac{1}{2}$. Thus,
$\Phi$ maps the set $B$ into itself.

To prove $\Phi$ is a contraction, for $\wt{f}^{+}$, $\wt{g}^{+}\in B$, we have
\begin{align*}
\Phi(\wt{f}^{+})-\Phi(\wt{g}^{+})=\int_{0}^{t}U(t-\tau)\lrc{\wt{Q}^{+}(\wt{f}^{+}-\wt{g}^{+},\wt{f}^{+})(\tau)+\wt{Q}^{+}(\wt{g}^{+},\wt{f}^{+}-\wt{g}^{+})(\tau)}d\tau,
\end{align*}
and hence get
\begin{align}\label{equ:closed estimate,gwp,contraction}
\n{\Phi(\wt{f}^{+})-\Phi(\wt{g}^{+})}_{U_{L}^{2}X^{s,\ve}}
\lesssim &\n{Q^{+}(\wt{f}^{+}-\wt{g}^{+},\wt{f}^{+})}_{L_{t}^{1}X^{s,\ve}}+\n{Q^{+}(\wt{g}^{+},\wt{f}^{+}-\wt{g}^{+})}_{L_{t}^{1}X^{s,\ve}}\\
\lesssim &\n{\wt{f}^{+}-\wt{g}^{+}}_{U_{L}^{2}X^{s,\ve}}\lrs{\n{\wt{f}^{+}}_{U_{L}^{2}X^{s,\ve}}+\n{\wt{g}^{+}}_{U_{L}^{2}X^{s,\ve}}}\notag\\
\leq& 4\n{\wt{f}^{+}-\wt{g}^{+}}_{U_{L}^{2}X^{s,\ve}}\n{\wt{f}_{0}}_{X^{s,\ve}}.\notag
\end{align}
Using \eqref{equ:gain-only,gwp,smallness} again, we conclude that $\Phi$ is a contraction map and there exists a unique global solution $\wt{f}^{+}$ to the gain-term-only Boltzmann equation \eqref{equ:gain term,Boltzmann,fourier} in the set $B$. For the non-negativity of the solution, due to that $f_{0}(x,v)\geq 0$, the iteration sequences on the $(x,v)$ side are non-negative, that is,
\begin{align*}
f_{n+1}^{+}(t)=S(t)f_{0}+\int_{0}^{t} S(t-\tau)Q^{+}(f_{n}^{+},f_{n}^{+})d\tau \geq 0, \quad S(t)=e^{-tv\cdot\nabla_{x}},
\end{align*}
which implies that the solution $f^{+}(t)$ is non-negative for all $t\in[0,\infty)$.

By the scaling-invariant estimate \eqref{equ:bilinear estimate,gain term,gwp} in Lemma \ref{lemma:bilinear estimate,gain term,gwp}, we have
\begin{align}
\bn{\wt{Q}^{+}(\wt{f}^{+},\wt{f}^{+})}_{L_{t}^{1}X^{s,\ve}}
\lesssim& \n{\wt{f}^{+}}_{U_{L}^{2}X^{s,\ve}}\n{\wt{f}^{+}}_{U_{L}^{2}X^{s,\ve}}.
\end{align}
Then using Plancherel identity, we arrive at the regularity bounds \eqref{equ:f+,regularity bound}--\eqref{equ:Q+,f+,regularity bound}.

Next, we get into the analysis of integrability bounds \eqref{equ:f+,Lx6}--\eqref{equ:Q-,f+,Lx6}.
By the Sobolev inequality that $H^{s}\hookrightarrow L^{p}$ with $p\in [2,\frac{6}{3-2s}]$, we immediately obtain estimates \eqref{equ:f+,Lx6}--\eqref{equ:Q+,f+,Lx6} by using the regularity bounds \eqref{equ:f+,regularity bound}--\eqref{equ:Q+,f+,regularity bound}.

For \eqref{equ:A,f+,Lxinfty}, by Lemma \ref{lemma:A,f,Lx,estimate} which we postpone to the end of the section, we have
\begin{align*}
\n{A\lrc{f^{+}}}_{L_{t}^{2}(0,\infty;L_{x}^{\infty}L_{v}^{\infty})}\lesssim& \n{\lra{\nabla_{x}}^{s}\lra{v}^{s+\ga}f_{0}}_{L_{x}^{2}L_{v}^{2}}
+\n{\lra{\nabla_{x}}^{s}\lra{v}^{s+\ga}\mathcal{N}\lrc{f^{+}}}_{L_{t}^{1}(0,\infty;L_{x}^{2}L_{v}^{2})}\\
\lesssim& \n{\lra{\nabla_{x}}^{s}\lra{v}^{s+\ga}f_{0}}_{L_{x}^{2}L_{v}^{2}}
+\n{\lra{\nabla_{x}}^{s}\lra{v}^{s+\ga}Q^{+}(f^{+},f^{+})}_{L_{t}^{1}(0,\infty;L_{v}^{2}L_{x}^{2})}.
\end{align*}
By using the regularity bound \eqref{equ:Q+,f+,regularity bound} for $Q^{+}(f^{+},f^{+})$, we complete the proof of \eqref{equ:A,f+,Lxinfty}.

For the integrability bound \eqref{equ:Q-,f+,Lx6}, due to that $Q^{-}(f^{+},f^{+})=f^{+}A[f^{+}]$, we use H\"{o}lder inequality to get
\begin{align*}
&\n{\lra{v}^{s+\ga}Q^{-}(f^{+},f^{+})}_{L_{t}^{1}(0,T;L_{v}^{2}L_{x}^{p})}\\
\lesssim& |T|^{\frac{1}{2}}\n{\lra{v}^{s+\ga}
f^{+}}_{L_{t}^{\infty}(0,T;L_{v}^{2}L_{x}^{p})}\n{A\lrc{f^{+}}}_{L_{t}^{2}(0,T;L_{x}^{\infty}L_{v}^{\infty})}\lesssim C_{T},
\end{align*}
where in the last inequality we have used the bounds \eqref{equ:f+,Lx6} and \eqref{equ:A,f+,Lxinfty}.
\end{proof}

In the end, we give the proof of the estimate on $A[f]$,
which has been used in the analysis of the integrability bounds in Proposition \ref{lemma:gwp,gain-only,small data}.
\begin{lemma}\label{lemma:A,f,Lx,estimate}
For $s>1$, $\ga\in(-1,0]$, we have
\begin{align}\label{equ:A,f,Lx3,estimate}
\n{A\lrc{f}}_{L_{t}^{2}L_{x}^{3}L_{v}^{\infty}}\lesssim \n{\lra{v}^{s+\ga}f_{0}}_{L_{x}^{2}L_{v}^{2}}+
\n{\lra{v}^{s+\ga}\mathcal{N}\lrc{f}}_{L_{t}^{1}L_{x}^{2}L_{v}^{2}},
\end{align}
where $f(t)$ satisfies the Duhamel formula that
$$f(t)=S(t)f_{0}+\int_{0}^{t}S(t-\tau)\mathcal{N}\lrc{f}(\tau)d\tau,\quad S(t)=e^{-tv\cdot\nabla_{x}}.$$
In particular, we have
\begin{align}\label{equ:A,f,Lxinfty,estimate}
\n{A\lrc{f}}_{L_{t}^{2}L_{x}^{\infty}L_{v}^{\infty}}\lesssim& \n{\lra{\nabla_{x}}^{s}A\lrc{f}}_{L_{t}^{2}L_{x}^{3}L_{v}^{\infty}}\\
\lesssim&  \n{\lra{\nabla_{x}}^{s}\lra{v}^{s+\ga}f_{0}}_{L_{x}^{2}L_{v}^{2}}
+\n{\lra{\nabla_{x}}^{s}\lra{v}^{s+\ga}\mathcal{N}\lrc{f}}_{L_{t}^{1}L_{x}^{2}L_{v}^{2}}.\notag
\end{align}

\end{lemma}

\begin{proof}

From the Duhamel formula that
\begin{align}
A\lrc{f}=A\lrc{S(t)f_{0}+\int_{0}^{t}S(t-\tau)\mathcal{N}\lrc{f}(\tau)d\tau},
\end{align}
it is sufficient to prove
\begin{align}\label{equ:A,f,Lx3,linear term}
\n{A\lrc{S(t)f_{0}}}_{L_{t}^{2}L_{x}^{3}L_{v}^{\infty}}\leq \n{\lra{v}^{s+\ga}f_{0}}_{L_{x}^{2}L_{v}^{2}}.
\end{align}
Indeed, for the nonlinear term, by using \eqref{equ:A,f,Lx3,linear term} we have
\begin{align*}
\bbn{A\lrc{\int_{0}^{t}S(t-\tau)\mathcal{N}\lrc{f}(\tau)d\tau}}_{L_{t}^{2}L_{x}^{3}L_{v}^{\infty}}
\leq &
\int_{0}^{\infty}\bn{A\lrc{S(t-\tau)\mathcal{N}\lrc{f}(\tau)}}_{L_{t}^{2}L_{x}^{3}L_{v}^{\infty}}d\tau\\
\lesssim &\int_{0}^{\infty}\n{\lra{v}^{s+\ga}S(-\tau)\mathcal{N}\lrc{f}(\tau)}_{L_{x}^{2}L_{v}^{2}}d\tau\\
\lesssim&\n{\lra{v}^{s+\ga}\mathcal{N}\lrc{f}}_{L_{t}^{1}L_{x}^{2}L_{v}^{2}}.
\end{align*}
To prove \eqref{equ:A,f,Lx3,linear term}, noting that
$$A\lrc{f}=\n{\textbf{b}}_{L^{1}(\mathbb{S}^{2})}\int_{\R^{3}} \frac{f(u)}{|u-v|^{-\ga}}du,$$
 we use the inequality that $\n{f}_{L_{v}^{\infty}}\lesssim \n{\mathcal{F}^{-1}_{v\mapsto \xi}(f)}_{L_{\xi}^{1}}$ to obtain
\begin{align*}
\n{A\lrc{S(t)f_{0}}}_{L_{t}^{2}L_{x}^{3}L_{v}^{\infty}}
\lesssim& \n{\mathcal{F}^{-1}_{v\mapsto \xi}\lrs{A\lrc{S(t)f_{0}}}}_{L_{t}^{2}L_{x}^{3}L_{\xi}^{1}}\\
\lesssim&\bbn{\frac{(U(t)\wt{f}_{0})(x,\xi)}{|\xi|^{3+\ga}}}_{L_{t}^{2}L_{x}^{3}L_{\xi}^{1}}\\
\lesssim&\n{U(t)\wt{f}_{0}}_{L_{t}^{2}L_{x}^{3}L_{\xi}^{\frac{3}{-\ga+\delta}}}^{\frac{1}{2}}
\n{U(t)\wt{f}_{0}}_{L_{t}^{2}L_{x}^{3}L_{\xi}^{\frac{3}{-\ga-\delta}}}^{\frac{1}{2}},
\end{align*}
where in the last inequality we have used the endpoint Hardy-Littlewood-Sobolev inequality that
\begin{align*}
\int  \frac{|u(\xi)|}{|\xi|^{3+\ga}}d\xi \lesssim\n{u}_{L^{\frac{3}{-\ga+\delta}}}^{\frac{1}{2}}\n{u}_{L^{\frac{3}{-\ga-\delta}}}^{\frac{1}{2}}.
\end{align*}
With $s>1$, the Sobolev inequality $W^{s+\ga,3}\hookrightarrow L^{\frac{3}{-\ga\pm \delta}}$ holds for small $\delta$. Hence, we obtain
\begin{align*}
\n{A\lrc{S(t)f_{0}}}_{L_{t}^{2}L_{x}^{3}L_{v}^{\infty}}\lesssim \n{\lra{\nabla_{\xi}}^{s+\ga}U(t)\wt{f}_{0}}_{L_{t}^{2}L_{x}^{3}L_{\xi}^{3}}.
\end{align*}
Then by Strichartz estimate \eqref{equ:strichartz estimate,linear}, we arrive at
\begin{align*}
\n{A\lrc{S(t)f_{0}}}_{L_{t}^{2}L_{x}^{3}L_{v}^{\infty}}\lesssim \n{\lra{\nabla_{\xi}}^{s+\ga}\wt{f}_{0}}_{L_{x,\xi}^{2}}=
\n{\lra{\nabla_{\xi}}^{s+\ga}f_{0}}_{L_{x,v}^{2}}.
\end{align*}

For estimate \eqref{equ:A,f,Lxinfty,estimate}, by the Sobolev inequality $W^{s,3}\hookrightarrow L^{\infty}$ with $s>1$, we get
\begin{align*}
\n{A\lrc{f}}_{L_{t}^{2}L_{x}^{\infty}L_{v}^{\infty}}\lesssim& \n{\lra{\nabla_{x}}^{s}A\lrc{f}}_{L_{t}^{2}L_{v}^{\infty}L_{x}^{3}}\lesssim \n{\lra{\nabla_{x}}^{s}A\lrc{f}}_{L_{t}^{2}L_{x}^{3}L_{v}^{\infty}}.
\end{align*}
Then using \eqref{equ:A,f,Lx3,estimate}, we complete the proof of \eqref{equ:A,f,Lxinfty,estimate}.

\end{proof}

\section{Global Existence of the Boltzmann Equation}\label{section:Global Existence of the Boltzmann Equation}
In the section, we establish the global existence and scattering of the Boltzmann equation by making use of
the iterative method of Kaniel--Shinbrot \cite{kaniel78the}.

\begin{proposition}\label{lemma:global existence,full}
Let $s\in (1,\frac{3}{2})$ and $\ga\in[-\frac{1}{2},0]$. There exists $\eta>0$, such that for all non-negative initial data $f_{0}\in L_{v}^{2,s+\ga}H_{x}^{s}$ with
\begin{align*}
\n{\lra{\nabla_{x}}^{s_{c}}\lra{v}^{s_{c}+\ga}f_{0}}_{L_{x,v}^{2}}\leq \eta,
\end{align*}
there exists a non-negative global solution $f\in C([0,\infty);L_{x}^{2}L_{v}^{2,s+\ga})$ to the Boltzmann equation satisfying
\begin{align}
&\n{\lra{v}^{s+\ga}f}_{L_{t}^{\infty}(0,\infty;L_{v}^{2}L_{x}^{p})}<\infty ,\label{equ:f,Lx6}\\
&\n{\lra{v}^{s+\ga}Q^{+}(f,f)}_{L_{t}^{1}(0,\infty;L_{v}^{2}L_{x}^{p})}<\infty,\label{equ:Q+,f,Lx6}\\
&\n{A\lrc{f}}_{L_{t}^{2}(0,\infty;L_{x}^{\infty}L_{v}^{\infty})}< \infty,\label{equ:A,f,Lxinfty}\\
&\n{\lra{v}^{s+\ga}Q^{-}(f,f)}_{L_{t}^{1}(0,T;L_{v}^{2}L_{x}^{p})}\leq C(p,T),\label{equ:Q-,f,Lx6}
\end{align}
for all $p\in [2,\frac{6}{3-2s}]$ and $T\in(0,\infty)$. Moreover,
\begin{enumerate}
\item The solution $f(t)$ is unique in the class of $C([0,T];L_{v}^{2,s+\ga}L_{x}^{2})$ solutions satisfying the integrability bounds \eqref{equ:f,Lx6}--\eqref{equ:Q-,f,Lx6} on $[0,T]$ with $p=6$.
\item  The solution $f(t)$ scatters in $L_{v}^{2,s+\ga}L_{x}^{2}\bigcap L_{v}^{2,s+\ga}L_{x}^{\frac{6}{3-2s}}$. That is, there exists a function $f_{+\infty}\in L_{v}^{2,s+\ga}L_{x}^{2}\bigcap L_{v}^{2,s+\ga}L_{x}^{\frac{6}{3-2s}}$ such that
    \begin{align*}
    \lim_{t\to+\infty}\n{f(t)-S(t)f_{+\infty}}_{L_{v}^{2,s+\ga}L_{x}^{2}\bigcap L_{v}^{2,s+\ga}L_{x}^{\frac{6}{3-2s}}}=0.
    \end{align*}
\end{enumerate}
\end{proposition}
\begin{proof}

The Kaniel--Shinbrot iteration in \cite{kaniel78the} is as follows:
\begin{equation}\label{equ:k-s iteration}
\left\{
\begin{aligned}
& \partial_t g_{n+1}+v \cdot \nabla_x g_{n+1}+g_{n+1} A\lrc{h_n}=Q^{+}\left(g_n, g_n\right), \\
& \partial_t h_{n+1}+v \cdot \nabla_x h_{n+1}+h_{n+1} A\lrc{g_n}=Q^{+}\left(h_n, h_n\right), \\
& g_{n+1}(0)=h_{n+1}(0)=f_0.
\end{aligned}
\right.
\end{equation}
With the beginning condition that
\begin{align}\label{equ:beginning condition}
0 \leq h_1 \leq h_2 \leq g_2 \leq g_1 ,
\end{align}
this iteration will generate a monotone sequence
$$
0 \leq h_1 \leq h_n \leq h_{n+1} \leq g_{n+1} \leq g_n \leq g_1 .
$$
To verify \eqref{equ:beginning condition}, we choose
$$h_{1}=0,\quad g_{1}=f^{+},$$
where $f^{+}$, constructed in Proposition \ref{lemma:gwp,gain-only,small data}, is the unique non-negative solution to the gain-term-only Boltzmann equation with initial data $f_{0}$.
Then by \eqref{equ:k-s iteration}, we compute $h_{2}$ and $g_{2}$ to get
\begin{align}
h_{2}(t)=&\lrc{S(t)f_{0}}\exp\lrc{-\int_{0}^{t}S(t-\tau)A[g_{1}](\tau)d\tau},\\
g_{2}(t)=&S(t)f_{0}+\int_{0}^{t}S(t-\tau)Q^{+}(g_{1},g_{1})(\tau)\tau=g_{1}(t).&
\end{align}
By the non-negativity of $g_{1}$, the monotonicity condition \eqref{equ:beginning condition} is satisfied for all $t\geq 0$.
Since all the sequences $g_{n}$, $h_{n}$ are bounded by $g_{1}=f^{+}$, the dominated convergence theorem implies their pointwise limits
\begin{align}
g=\lim_{n\to\infty}g_{n},\quad h=\lim_{n\to \infty}h_{n},
\end{align}
which satisfy
$$
\begin{aligned}
& \partial_t g+v \cdot \nabla_x g+g A\lrc{h}=Q^{+}(g, g), \\
& \partial_t h+v \cdot \nabla_x h+h A\lrc{g}=Q^{+}(h, h), \\
& g(0)=h(0)=f_0,
\end{aligned}
$$
in the sense of distributions. Noting that $0 \leq h\leq g\leq f^{+}$,
by the integrability bounds \eqref{equ:f+,Lx6}--\eqref{equ:Q-,f+,Lx6} on $f^{+}$ in Proposition \ref{lemma:gwp,gain-only,small data}, for $p\in [2,\frac{6}{3-2s}]$ we have
\begin{align}
&\n{\lra{v}^{s+\ga}h}_{L_{t}^{\infty}L_{v}^{2}L_{x}^{p}}\leq
\n{\lra{v}^{s+\ga}g}_{L_{t}^{\infty}L_{v}^{2}L_{x}^{p}}<\infty,\label{equ:h,g,Lx6}\\
&\n{A\lrc{h}}_{L_{t}^{2}(0,\infty;L_{x}^{\infty}L_{v}^{\infty})}\leq
\n{A\lrc{g}}_{L_{t}^{2}(0,\infty;L_{x}^{\infty}L_{v}^{\infty})}<\infty,\label{equ:A,h,g,Lxinfty}\\
&\n{\lra{v}^{s+\ga}Q^{+}(h,h)}_{L_{t}^{1}L_{v}^{2}L_{x}^{p}}\leq
\n{\lra{v}^{s+\ga}Q^{+}(g,g)}_{L_{t}^{1}L_{v}^{2}L_{x}^{p}}<\infty,\label{equ:Q+,h,g,Lx6} \\
&\n{\lra{v}^{s+\ga}Q^{-}(h,h)}_{L_{t}^{1}(0,T;L_{v}^{2}L_{x}^{p})}\leq
\n{\lra{v}^{s+\ga}Q^{-}(g,g)}_{L_{t}^{1}(0,T;L_{v}^{2}L_{x}^{p})}\leq C_{T}. \label{equ:Q-,h,g,Lx6}
\end{align}

We are left to prove that $h=g$. Let $w=g-h\geq 0$. We write out the equation for the difference
\begin{equation}\label{equ:limiting equation,ks iteration}
\left\{
\begin{aligned}
&\pa_{t}w+v\cdot \nabla_{x}w=Q^{+}(g,w)+Q^{+}(w,h)+hA\lrc{w}-wA\lrc{h},\\
&w(0)=0.
\end{aligned}
\right.
\end{equation}
With these integrability bounds \eqref{equ:h,g,Lx6}--\eqref{equ:Q-,h,g,Lx6}, we infer that $w(t)\equiv 0$ by using the uniqueness Lemma \ref{lemma:conditional uniqueness,boltzmann}, the proof of which is postponed to Section \ref{section:uniqueness} for smoothness of presentation.

Thus, with $h=g$, we conclude the global existence of the non-negative solution
for the Boltzmann equation. The uniqueness of the solution also follows from Lemma \ref{lemma:conditional uniqueness,boltzmann} in Section \ref{section:uniqueness}.

For the scattering result, by the Duhamel formula that
\begin{align}
f(t)=S(t)\lrc{f_{0}+\int_{0}^{t}S(-\tau)Q^{+}(f,f)(\tau)d\tau-\int_{0}^{t}S(-\tau)Q^{-}(f,f)(\tau)d\tau},
\end{align}
it is enough to prove that the limit
\begin{align*}
\lim_{t\to+\infty}\int_{0}^{t}S(-\tau)Q^{\pm}(f,f)(\tau)d\tau
\end{align*}
exists in $L_{v}^{2,s+\ga}L_{x}^{2}\bigcap L_{v}^{2,s+\ga}L_{x}^{\frac{6}{3-2s}}$. Due to the non-negativity of $f$, we have
\begin{align}
\int_{0}^{t}S(-\tau)Q^{-}(f(\tau),f(\tau))d\tau\leq
f_{0}+\int_{0}^{t}S(-\tau)Q^{+}(f(\tau),f(\tau))d\tau,
\end{align}
which, together with the integrability bound \eqref{equ:Q+,f,Lx6} on the gain term, implies that
\begin{align*}
&\n{S(-\tau)Q^{-}(f,f)(\tau)}_{L_{v}^{2,s+\ga}L_{x}^{p}L_{\tau}^{1}}\\
\leq&
\n{f_{0}}_{L_{v}^{2,s+\ga}L_{x}^{p}}+\n{S(-\tau)Q^{+}(f,f)(\tau)}_{L_{v}^{2,s+\ga}L_{x}^{p}L_{\tau}^{1}}\\
\leq&\n{f_{0}}_{L_{v}^{2,s+\ga}L_{x}^{p}}+\n{S(-\tau)Q^{+}(f,f)(\tau)}_{L_{\tau}^{1}L_{v}^{2,s+\ga}L_{x}^{p}}<\infty.
\end{align*}
Then by the monotone convergence theorem, we conclude that
\begin{align*}
f_{+\infty}=:f_{0}+\lim_{t\to+\infty}\int_{0}^{t}S(-\tau)Q(f,f)(\tau)d\tau
\end{align*}
exists in $L_{v}^{2,s+\ga}L_{x}^{2}\bigcap L_{v}^{2,s+\ga}L_{x}^{\frac{6}{3-2s}}$, and thus
\begin{align*}
\lim_{t\to+\infty}\n{S(-t)f(t)-f_{+\infty}}_{L_{v}^{2,s+\ga}L_{x}^{2}\bigcap L_{v}^{2,s+\ga}L_{x}^{\frac{6}{3-2s}}}=0.
\end{align*}
\end{proof}

\section{Uniqueness of the Boltzmann Equation}\label{section:uniqueness}

In the section, we present the proof of the uniqueness of the Boltzmann equation in an integrable class, which also works the same for the limiting equation \eqref{equ:limiting equation,ks iteration} from the Kaniel--Shinbrot iteration. In the following Lemma \ref{lemma:Q+,bilinear estimate,uniqueness,L2},
we provide an $L_{t}^{1}L_{v}^{2,s+\ga}L_{x}^{2}$ bilinear estimate, which carries no regularity and is the key to the uniqueness in the integrable class.

\begin{lemma}\label{lemma:Q+,bilinear estimate,uniqueness,L2}
Let $s>1$ and $\ga\in[-1,0]$. We have
\begin{align}
\n{\lra{v}^{s+\ga}Q^{+}(S(t)f_{0},S(t)g_{0})}_{L_{t}^{1}(0,T;L_{x,v}^{2})}
\lesssim& |T|^{\frac{1}{2}}\n{\lra{v}^{s+\ga}f_{0}}_{L_{x,v}^{2}}\n{\lra{v}^{s+\ga}g_{0}}_{L_{v}^{2}L_{x}^{6}},\label{equ:Q+,bilinear estimate,uniqueness,L2,g2d}\\
\n{\lra{v}^{s+\ga}Q^{+}(S(t)f_{0},S(t)g_{0})}_{L_{t}^{1}(0,T;L_{x,v}^{2})}
\lesssim& |T|^{\frac{1}{2}}\n{\lra{v}^{s+\ga}f_{0}}_{L_{v}^{2}L_{x}^{6}}\n{\lra{v}^{s+\ga}g_{0}}_{L_{x,v}^{2}}.\label{equ:Q+,bilinear estimate,uniqueness,L2,f2d}
\end{align}
\end{lemma}

\begin{proof}
By Plancherel identity, it suffices to prove that
\begin{align}
\n{\lra{\nabla_{\xi}}^{s+\ga}\wt{Q}^{+}(U(t)\wt{f}_{0},U(t)\wt{g}_{0})}_{L_{t}^{1}(0,T;L_{x,\xi}^{2})} \lesssim&|T|^{\frac{1}{2}}\n{\lra{v}^{s+\ga}f_{0}}_{L_{x,v}^{2}}\n{\lra{v}^{s+\ga}g_{0}}_{L_{v}^{2}L_{x}^{6}},\label{equ:Q+,bilinear estimate,uniqueness,xi,g2d}\\
\n{\lra{\nabla_{\xi}}^{s+\ga}\wt{Q}^{+}(U(t)\wt{f}_{0},U(t)\wt{g}_{0})}_{L_{t}^{1}(0,T;L_{x,\xi}^{2})}
\lesssim& |T|^{\frac{1}{2}}\n{\lra{v}^{s+\ga}f_{0}}_{L_{v}^{2}L_{x}^{6}}\n{\lra{v}^{s+\ga}g_{0}}_{L_{x,v}^{2}}.\label{equ:Q+,bilinear estimate,uniqueness,xi,f2d}
\end{align}
We only handle \eqref{equ:Q+,bilinear estimate,uniqueness,xi,g2d}, as \eqref{equ:Q+,bilinear estimate,uniqueness,xi,f2d} follows similarly.
By duality,
\eqref{equ:Q+,bilinear estimate,uniqueness,xi,g2d} is equivalent to
\begin{align}\label{equ:Q+,bilinear estimate,uniqueness,xi,duality}
\int_{0}^{T}\int \wt{Q}^{+}(U(t)\wt{f}_{0},U(t)\wt{g}_{0}) h dx d\xi dt
\lesssim&|T|^{\frac{1}{2}}\n{\lra{v}^{s+\ga}f_{0}}_{L_{x,v}^{2}}\n{\lra{v}^{s+\ga}g_{0}}_{L_{v}^{2}L_{x}^{6}}.
\end{align}

We denote by $I$ the integral in \eqref{equ:Q+,bilinear estimate,uniqueness,xi,duality}.
Inserting a Littlewood-Paley decomposition gives that
\begin{align*}
I=\sum_{M,M_{1},M_{2}} I_{M,M_{1},M_{2}}
\end{align*}
where
\begin{align*}
I_{M,M_{1},M_{2}}= \int \wt{Q}^{+}(P_{M_{1}}^{\xi}\wt{f},P_{M_{2}}^{\xi}\wt{g}) P_{M}^{\xi}h dx d\xi dt
\end{align*}
with $\wt{f}(t)=U(t)\wt{f}_{0}$ and $\wt{g}(t)=U(t)\wt{g}_{0}$. In the same way as the frequency analysis of \eqref{equ:property,constraint,projector,v,variable}, it gives a constraint that $M\lesssim \max\lrs{M_{1},M_{2}}$.

We divide the sum into two cases, that is, Case A: $M_{1}\geq M_{2}$ and Case B: $M_{1}\leq M_{2}$. We only deal with Case A, as Case B can be treated in a similar way.

Let $I_{A}$ denote the integral restricted to the Case A.
\begin{align*}
I_{A}=&\sum_{M_{1}\geq M_{2},M_{1}\gtrsim M}\int \wt{Q}^{+}(P_{M_{1}}^{\xi}\wt{f},P_{M_{2}}^{\xi}\wt{g}) P_{M}^{\xi}h dx d\xi dt\\
=&\sum_{M_{1}\gtrsim M}\int_{0}^{T}\int \wt{Q}^{+}(P_{M_{1}}^{\xi}\wt{f},P_{\leq M_{1}}^{\xi}\wt{g}) P_{M}^{\xi}h dx d\xi dt.
\end{align*}
where in the last equality we have done the sum in $M_{2}$. By using H\"{o}lder inequality and estimate \eqref{equ:Q+,bilinear estimate,L3,PM,f,g} in Lemma \ref{equ:Q+,bilinear estimate}, we have
\begin{align*}
I_{A}\leq& \int \sum_{M_{1}\gtrsim M} \n{\wt{Q}^{+}(P_{M_{1}}^{\xi}\wt{f},P_{\leq M_{1}}^{\xi}\wt{g})}_{L_{\xi}^{2}} \n{P_{M}^{\xi}h}_{L_{\xi}^{2}} dx  dt\\
\lesssim& \int \sum_{M_{1}\gtrsim M} \n{P_{M_{1}}^{\xi}\wt{f}}_{L_{\xi}^{3}}\n{P_{\leq M_{1}}^{\xi}\wt{g}}_{L_{\xi}^{\frac{6}{1-2\ga}}} \n{P_{M}^{\xi}h}_{L_{\xi}^{2}} dx  dt.
\end{align*}
By Bernstein inequality and Sobolev inequality that $W^{1+\ga,2} \hookrightarrow L^{\frac{6}{1-2\ga}}$,
\begin{align*}
I_{A}\lesssim& \int \sum_{M_{1}\gtrsim M} \frac{M^{s+\ga}}{M_{1}^{s+\ga}}
\n{\lra{\nabla_{\xi}}^{s+\ga}P_{M_{1}}^{\xi}\wt{f}}_{L_{\xi}^{3}}\n{\lra{\nabla_{\xi}}^{s+\ga}\wt{g}}_{L_{\xi}^{2}} \n{\lra{\nabla_{\xi}}^{-s-\ga}P_{M}^{\xi}h}_{L_{\xi}^{2}} dx  dt.
\end{align*}
With $s+\ga>0$, we use Cauchy-Schwarz in $M$ and $M_{1}$ to get
\begin{align*}
I_{A}\leq& \int \n{\lra{\nabla_{\xi}}^{s+\ga}\wt{g}}_{L_{\xi}^{2}}
\lrs{\sum_{M_{1}\gtrsim M}\frac{M^{s+\ga}}{M_{1}^{s+\ga}}
\n{\lra{\nabla_{\xi}}^{s+\ga}P_{M_{1}}^{\xi}\wt{f}}_{L_{\xi}^{3}}^{2}}^{\frac{1}{2}}\\
&\lrs{\sum_{M_{1}\gtrsim M}\frac{M^{s+\ga}}{M_{1}^{s+\ga}}
\n{\lra{\nabla_{\xi}}^{-s-\ga}P_{M}^{\xi}h}_{L_{\xi}^{2}}^{2}}^{\frac{1}{2}}dxdt\\
\lesssim &\int \n{\lra{\nabla_{\xi}}^{s+\ga}\wt{g}}_{L_{\xi}^{2}}
\n{\lra{\nabla_{\xi}}^{s+\ga}P_{M_{1}}^{\xi}\wt{f}}_{l_{M_{1}}^{2}L_{\xi}^{3}}
\n{\lra{\nabla_{\xi}}^{-s-\ga}h}_{L_{\xi}^{2}} dxdt.
\end{align*}
By H\"{o}lder inequality in the $x$-variable and the $t$-variable, we use Minkowski inequality to get
\begin{align*}
I_{A}\lesssim &
\n{\lra{\nabla_{\xi}}^{s+\ga}\wt{g}}_{L_{t}^{2}(0,T;L_{x}^{6}L_{\xi}^{2})}
\n{\lra{\nabla_{\xi}}^{s+\ga}P_{M_{1}}^{\xi}\wt{f}}_{L_{t}^{2}L_{x}^{3}l_{M_{1}}^{2}L_{\xi}^{3}}
\n{\lra{\nabla_{\xi}}^{-s-\ga}h}_{L_{t}^{\infty}(0,T;L_{x}^{2}L_{\xi}^{2})}\\
\lesssim&\n{\lra{\nabla_{\xi}}^{s+\ga}\wt{g}}_{L_{t}^{2}(0,T;L_{x}^{6}L_{\xi}^{2})}
\n{\lra{\nabla_{\xi}}^{s+\ga}P_{M_{1}}^{\xi}\wt{f}}_{l_{M_{1}}^{2}L_{t}^{2}L_{x}^{3}L_{\xi}^{3}}
\n{\lra{\nabla_{\xi}}^{-s-\ga}h}_{L_{t}^{\infty}(0,T;L_{x}^{2}L_{\xi}^{2})}.
\end{align*}
Inserting in $\wt{f}(t)=U(t)\wt{f}_{0}$ and $\wt{g}(t)=U(t)\wt{g}_{0}$, we use Strichartz estimate \eqref{equ:strichartz estimate,linear} and to obtain
\begin{align*}
I_{A}\lesssim&\n{\lra{\nabla_{\xi}}^{s+\ga}U(t)\wt{g}_{0}}_{L_{t}^{2}(0,T;L_{x}^{6}L_{\xi}^{2})}
\n{\lra{\nabla_{\xi}}^{s+\ga}P_{M_{1}}^{\xi}\wt{f}_{0}}_{l_{M_{1}}^{2}L_{x}^{2}L_{\xi}^{2}}
\n{\lra{\nabla_{\xi}}^{-s-\ga}h}_{L_{t}^{\infty}(0,T;L_{x}^{2}L_{\xi}^{2})}\\
\lesssim&\n{\lra{v}^{s+\ga}S(t)g_{0}}_{L_{t}^{2}(0,T;L_{x}^{6}L_{v}^{2})}
\n{\lra{v}^{s+\ga}f_{0}}_{L_{x}^{2}L_{v}^{2}}
\n{\lra{\nabla_{\xi}}^{-s-\ga}h}_{L_{t}^{\infty}(0,T;L_{x}^{2}L_{\xi}^{2})}\\
\lesssim&|T|^{\frac{1}{2}}\n{\lra{v}^{s+\ga}g_{0}}_{L_{v}^{2}L_{x}^{6}}
\n{\lra{v}^{s+\ga}f_{0}}_{L_{x}^{2}L_{v}^{2}}
\n{\lra{\nabla_{\xi}}^{-s-\ga}h}_{L_{t}^{\infty}(0,T;L_{x}^{2}L_{\xi}^{2})},
\end{align*}
where in the last inequality we have used that
\begin{align*}
\n{\lra{v}^{s+\ga}S(t)g_{0}}_{L_{t}^{2}(0,T;L_{x}^{6}L_{v}^{2})}\leq |T|^{\frac{1}{2}}\n{\lra{v}^{s+\ga}S(t)g_{0}}_{L_{t}^{\infty}(0,T;L_{v}^{2}L_{x}^{6})}\leq |T|^{\frac{1}{2}}\n{\lra{v}^{s+\ga}g_{0}}_{L_{v}^{2}L_{x}^{6}}.
\end{align*}
Therefore, we have completed the proof of \eqref{equ:Q+,bilinear estimate,uniqueness,xi,duality}.
\end{proof}

With the $L_{t}^{1}L_{v}^{2,s+\ga}L_{x}^{2}$ bilinear estimate, we are able to prove the uniqueness theorem.
\begin{lemma}[Uniqueness]\label{lemma:conditional uniqueness,boltzmann}
Let $s>1$ and $\ga\in[-1,0]$. There is at most one solution $f(t)\in C([0,T];L_{v}^{2,s+\ga}L_{x}^{2})$ to the Boltzmann equation \eqref{equ:Boltzmann} on $[0,T]$ satisfying
\begin{align}
&\n{\lra{v}^{s+\ga}f}_{L_{t}^{\infty}(0,T;L_{v}^{2}L_{x}^{6})}\leq C_{T},\label{equ:f,Lx6,uniqueness}\\
&\n{A\lrc{f}}_{L_{t}^{2}(0,T;L_{x}^{\infty}L_{v}^{\infty})}\leq C_{T},\label{equ:A,f,Lxinfty,uniqueness}\\
&\n{\lra{v}^{s+\ga}Q^{\pm}(f,f)}_{L_{t}^{1}(0,T;L_{v}^{2}L_{x}^{6})}\leq C_{T}.\label{equ:Q+,f,Lx6,uniquness}
\end{align}
\end{lemma}

\begin{proof}
Let $f$, $g$ be two solutions of the Boltzmann equation on $[0,T]$ satisfying the bounds \eqref{equ:f,Lx6,uniqueness}--\eqref{equ:Q+,f,Lx6,uniquness}. We consider the difference
$w=f-g$,
which satisfies
\begin{align}\label{equ:duhamel formula,w}
w(t)=\int_{0}^{t}S(t-\tau)\mathcal{N}\lrc{w}(\tau)d\tau,
\end{align}
where
$\mathcal{N}\lrc{w}=Q^{+}(w,f)+Q^{+}(g,w)-Q^{-}(w,f)-Q^{-}(g,w)$.
Define the quantity
\begin{align}
W(t_{0})=\n{w}_{L_{t}^{\infty}(0,t_{0};L_{v}^{2,s+\ga}L_{x}^{2})}+\n{\mathcal{N}\lrc{w}}_{L_{t}^{1}(0,t_{0};L_{v}^{2,s+\ga}L_{x}^{2})}.
\end{align}
It suffices to prove that $W(t_{0})=0$ for a sufficiently small $t_{0}$, as the global uniqueness follows from a standard continuity argument.
To do this, we provide a closed estimate for $W(t_{0})$. Noticing that
\begin{align*}
\n{w}_{L_{t}^{\infty}(0,t_{0};L_{v}^{2,s+\ga}L_{x}^{2})}\leq \n{\mathcal{N}\lrc{w}}_{L_{t}^{1}(0,t_{0};L_{v}^{2,s+\ga}L_{x}^{2})},
\end{align*}
we are left to deal with the nonlinear term and divide it into four parts
\begin{align*}
\n{\mathcal{N}\lrc{w}}_{L_{t}^{1}(0,t_{0};L_{v}^{2,s+\ga}L_{x}^{2})}\leq I_{1}+I_{2}+I_{3}+I_{4},
\end{align*}
where
\begin{align*}
I_{1}=\n{Q^{+}(w,f)}_{L_{t}^{1}(0,t_{0};L_{v}^{2,s+\ga}L_{x}^{2})},\\
I_{2}=\n{Q^{+}(g,w)}_{L_{t}^{1}(0,t_{0};L_{v}^{2,s+\ga}L_{x}^{2})},\\
I_{3}=\n{Q^{-}(w,f)}_{L_{t}^{1}(0,t_{0};L_{v}^{2,s+\ga}L_{x}^{2})},\\
I_{4}=\n{Q^{-}(g,w)}_{L_{t}^{1}(0,t_{0};L_{v}^{2,s+\ga}L_{x}^{2})}.\\
\end{align*}

For $I_{1}$, we have
\begin{align*}
I_{1}=&\n{Q^{+}(w,f)}_{L_{t}^{1}(0,t_{0};L_{v}^{2,s+\ga}L_{x}^{2})}\\
\lesssim&\n{Q^{+}(w,S(t)f(0))}_{L_{t}^{1}(0,t_{0};L_{v}^{2,s+\ga}L_{x}^{2})}\\
&+\bbn{Q^{+}\lrs{w,\int_{0}^{t}S(t-\sigma)\mathcal{N}\lrc{f}(\sigma)d\sigma}}_{L_{t}^{1}(0,t_{0};L_{v}^{2,s+\ga}L_{x}^{2})}\\
=:&I_{11}+I_{12}.
\end{align*}

Using Duhamel formula \eqref{equ:duhamel formula,w} of $w(t)$ and Minkowski inequality, we get
\begin{align*}
I_{11}=& \bbn{\int_{0}^{t}Q^{+}\lrs{S(t-\tau)\mathcal{N}\lrc{w}(\tau),S(t)f(0)}d\tau}_{L_{t}^{1}(0,t_{0};L_{v}^{2,s+\ga}L_{x}^{2})}\\
\leq& \int_{0}^{t_{0}}\n{Q^{+}\lrs{S(t-\tau)\mathcal{N}\lrc{w}(\tau),S(t)f(0)}}_{L_{t}^{1}(0,t_{0};L_{v}^{2,s+\ga}L_{x}^{2})}d\tau.
\end{align*}
By the $L_{t}^{1}L_{v}^{2,s+\ga}L_{x}^{2}$ bilinear estimate \eqref{equ:Q+,bilinear estimate,uniqueness,L2,g2d} in Lemma \ref{lemma:Q+,bilinear estimate,uniqueness,L2},
\begin{align*}
I_{11}\lesssim& \int_{0}^{t_{0}} \sqrt{t_{0}}\n{S(-\tau)\mathcal{N}\lrc{w}(\tau)}_{L_{v}^{2,s+\ga}L_{x}^{2}}\n{f(0)}_{L_{v}^{2,s+\ga}L_{x}^{6}}d\tau\\
\lesssim& \sqrt{t_{0}}\n{\mathcal{N}\lrc{w}}_{L_{t}^{1}(0,t_{0};L_{v}^{2,s+\ga}L_{x}^{2})}\n{f(0)}_{L_{v}^{2,s+\ga}L_{x}^{6}}\\
\lesssim& \sqrt{t_{0}}C_{T}W(t_{0}).
\end{align*}

In a similar way, for $I_{12}$ we have
\begin{align*}
I_{12}\leq& \bbn{\int_{0}^{t}\int_{0}^{t}Q^{+}\lrs{S(t-\tau)\mathcal{N}\lrc{w}(\tau),\int_{0}^{t}S(t-\sigma)\mathcal{N}\lrc{f}(\sigma)}d\sigma d\tau}_{L_{t}^{1}(0,t_{0};L_{v}^{2,s+\ga}L_{x}^{2})}\\
\leq& \int_{0}^{t_{0}}\int_{0}^{t_{0}}\n{Q^{+}\lrs{S(t-\tau)\mathcal{N}\lrc{w}(\tau),S(t-\sigma)\mathcal{N}
\lrc{f}(\sigma)}}_{L_{t}^{1}(0,t_{0};L_{v}^{2,s+\ga}L_{x}^{2})}d\tau d\sigma\\
\lesssim& \int_{0}^{t_{0}} \int_{0}^{t_{0}} \sqrt{t_{0}}\n{S(-\tau)\mathcal{N}\lrc{w}(\tau)}_{L_{v}^{2,s+\ga}L_{x}^{2}}\n{S(-\sigma)\mathcal{N}\lrc{f}(\sigma)}_{L_{v}^{2,s+\ga}L_{x}^{6}}d\tau d\sigma\\
\lesssim& \sqrt{t_{0}}\n{\mathcal{N}\lrc{w}}_{L_{t}^{1}(0,t_{0};L_{v}^{2,s+\ga}L_{x}^{2})}
\n{\mathcal{N}\lrc{f}}_{L_{t}^{1}(0,t_{0};L_{v}^{2,s+\ga}L_{x}^{6})}\\
\lesssim& \sqrt{t_{0}}C_{T}W(t_{0}).
\end{align*}
Together with estimates for $I_{11}$ and $I_{12}$, we get
\begin{align}\label{equ:uniqueness,closed estimate,I1}
I_{1}\lesssim& \sqrt{t_{0}}C_{T}W(t_{0}).
\end{align}

Since the term $I_{2}$ can be estimated in the same way as $I_{1}$, we also have
\begin{align}\label{equ:uniqueness,closed estimate,I2}
I_{2}\lesssim& \sqrt{t_{0}}C_{T}W(t_{0}).
\end{align}

Next, we deal with the term $I_{3}$. By H\"{o}lder inequality, we obtain
\begin{align}\label{equ:uniqueness,closed estimate,I3}
I_{3}=&\n{Q^{-}(w,f)}_{L_{t}^{1}(0,t_{0};L_{v}^{2,s+\ga}L_{x}^{2})}\\
\leq&\sqrt{t_{0}}\n{w}_{L_{t}^{\infty}(0,t_{0};L_{v}^{2,s+\ga}L_{x}^{2})}\n{A\lrc{f}}_{L_{t}^{2}L_{x}^{\infty}L_{v}^{\infty}}\notag\\
\lesssim&\sqrt{t_{0}}C_{T}W(t_{0}).\notag
\end{align}

For $I_{4}$, by using H\"{o}lder inequality and then estimate \eqref{equ:A,f,Lx3,estimate} in Lemma \ref{lemma:A,f,Lx,estimate}, we get
\begin{align}\label{equ:uniqueness,closed estimate,I4}
I_{4}=&\n{Q^{-}(g,w)}_{L_{t}^{1}(0,t_{0};L_{v}^{2,s+\ga}L_{x}^{2})}\\
\leq&\sqrt{t_{0}}\n{g}_{L_{t}^{\infty}(0,t_{0};L_{v}^{2,s+\ga}L_{x}^{6})}\n{A\lrc{w}}_{L_{t}^{2}L_{x}^{3}L_{v}^{\infty}}\notag\\
\leq&\sqrt{t_{0}}\n{g}_{L_{t}^{\infty}(0,t_{0};L_{v}^{2,s+\ga}L_{x}^{6})}\lrs{\n{w(0)}_{L_{v}^{2,s+\ga}L_{x}^{2}}+
\n{\mathcal{N}\lrc{w}}_{L_{t}^{1}L_{x}^{2}L_{v}^{2,s+\ga}}}\notag\\
\lesssim&\sqrt{t_{0}}C_{T}W(t_{0}).\notag
\end{align}

Putting estimates \eqref{equ:uniqueness,closed estimate,I1}, \eqref{equ:uniqueness,closed estimate,I2}, \eqref{equ:uniqueness,closed estimate,I3} and
\eqref{equ:uniqueness,closed estimate,I4} together, we arrive at
\begin{align*}
W(t_{0})\lesssim \sqrt{t_{0}}C_{T}W(t_{0}).
\end{align*}
By choosing $t_{0}$ sufficiently small, we conclude that $W(t_{0})=0$. One can
then extend this vanishing to the entire time interval $[0,T]$ by a standard continuity argument.
\end{proof}

\section{Strong Solution and the Blow-up Criterion}\label{section:Strong Solution and Blow-up Criteria}
In the section, we prove the strong local well-posedness of the Boltzmann equation and give the blow-up criterion. We call it strong as the local solution carries the regularity.
\begin{lemma}
Let $T\leq 1$, $s>1$, $\ga\in [-1,0]$. We have
\begin{align}
\n{\lra{\nabla_{x}}^{s}\lra{\nabla_{\xi}}^{s+\ga}\wt{Q}^{-}(\wt{f},\wt{g})}_{L_{t}^{2}(0,T;L_{x\xi}^{2})}\lesssim \n{\wt{f}}_{U_{L}^{2}(0,T;H_{x}^{s}H_{\xi}^{s+\ga})}\n{\wt{g}}_{U_{L}^{2}(0,T;H_{x}^{s}H_{\xi}^{s+\ga})}\label{equ:bilinear estimate,Q-,Lt2},\\
\n{\lra{\nabla_{x}}^{s}\lra{\nabla_{\xi}}^{s+\ga}\wt{Q}^{+}(\wt{f},\wt{g})}_{L_{t}^{2}(0,T;L_{x\xi}^{2})}\lesssim \n{\wt{f}}_{U_{L}^{2}(0,T;H_{x}^{s}H_{\xi}^{s+\ga})}\n{\wt{g}}_{U_{L}^{2}(0,T;H_{x}^{s}H_{\xi}^{s+\ga})}.\label{equ:Q+,bilinear estimate,Lt2}
\end{align}
\end{lemma}
\begin{proof}
The estimates \eqref{equ:bilinear estimate,Q-,Lt2}--\eqref{equ:Q+,bilinear estimate,Lt2} have been established in \cite[Section 2]{chen2023well} in the Fourier restriction space. In the same way, we can extend them to the atomic $U_{L}^{2}$ space and we omit the proof for simplicity.
\end{proof}

By the bilinear estimates \eqref{equ:bilinear estimate,Q-,Lt2}--\eqref{equ:Q+,bilinear estimate,Lt2}, we conclude the local well-posedness of the Boltzmann equation.
\begin{theorem}\label{thm:local well-posedness}
Let $s>1$, $\ga\in[-1,0]$. The Boltzmann equation \eqref{equ:Boltzmann} is locally well-posed in $L_{v}^{2,s+\ga}H_{x}^{s}$.

More precisely,
for each $f_{0}\in L_{v}^{2,s+\ga}H_{x}^{s}$, there exists a time $T_{0}>0$ such that there exists a unique $C([0,T_{0}];L_{v}^{2,s+\ga}H_{x}^{s})$ solution $f(t)$ to the Boltzmann equation satisfying
\begin{align}
\n{f}_{U_{L}(0,T_{0};L_{v}^{2,s+\ga}H_{x}^{s})}<\infty.
\end{align}
Moreover, we have:
\begin{enumerate}[$(1)$]
\item The solution $f(t)$ satisfies
\begin{align}\label{equ:A,f,Lxinfty,local}
 \n{Q^{\pm}(f,f)}_{L_{t}^{1}(0,T_{0};L_{v}^{2,s+\ga}H_{x}^{s})}<\infty,\quad \n{A[f]}_{L_{t}^{2}(0,T_{0};L_{x}^{\infty}L_{v}^{\infty})}<\infty.
\end{align}
\item
The solution map $f_{0}\in L_{v}^{2,s+\ga}H_{x}^{s}\mapsto f\in C([0,T_{0}];L_{v}^{2,s+\ga}H_{x}^{s})$ is Lipschitz continuous.
\item
The lifespan $T(f_{0})$ is bounded from below,
\begin{align*}
T(f_{0})\geq C\n{f_{0}}_{L_{v}^{2,s+\ga}H_{x}^{s}}^{-2},
\end{align*}
which gives a blow-up criterion that
\begin{align}
\lim_{t\to T(f_{0})}\n{f(t)}_{L_{v}^{2,s+\ga}H_{x}^{s}}=\infty.
\end{align}
\end{enumerate}
\end{theorem}
\begin{proof}
We omit the proof, as it follows from the standard contraction mapping principle. See for example \cite[Section 2.3]{chen2023well}. The bounds \eqref{equ:A,f,Lxinfty,local} follow from the bilinear estimates \eqref{equ:bilinear estimate,Q-,Lt2}--\eqref{equ:Q+,bilinear estimate,Lt2} and
estimate \eqref{equ:A,f,Lxinfty,estimate} in Lemma \ref{lemma:A,f,Lx,estimate}.
\end{proof}

\section{Persistence of Regularity}\label{section:Persistence of Regularity}
In the current section, we provide the regularity criteria for the Boltzmann equation and then
 recover the regularity of the global solution $f(t)$ constructed in Proposition \ref{lemma:global existence,full} by using the uniqueness lemma and the strong local well-posedness established in Sections \ref{section:uniqueness}--\ref{section:Strong Solution and Blow-up Criteria}.

For simplicity, we define the integrability bound
\begin{align}\label{equ:the integrability bound}
M_{r}(0,T)=:&\n{\lra{v}^{r}f}_{L_{t}^{\infty}(0,T;L_{v}^{2}L_{x}^{6})}+\n{A\lrc{f}}_{L_{t}^{2}(0,T;L_{x}^{\infty}L_{v}^{\infty})}
+\n{\lra{v}^{r}Q^{\pm}(f,f)}_{L_{t}^{1}(0,T;L_{v}^{2}L_{x}^{6})},
\end{align}
the regularity bound
\begin{align}\label{equ:the regularity bound}
E_{s,r}(t_{0},t_{0}+t):=\n{\lra{\nabla_{x}}^{s}\lra{v}^{r}f}_{L_{t}^{\infty}(t_{0},t_{0}+t;L_{x,v}^{2})}
+\n{\lra{\nabla_{x}}^{s}\lra{v}^{r}Q^{\pm}(f,f)}_{L_{t}^{1}(t_{0},t_{0}+t;L_{x,v}^{2})},
\end{align}
and the maximal time for the regularity bound
\begin{align}\label{equ:definition,T0}
T_{s,r}^{*}=\sup_{T}\lr{T\geq 0:E_{s,r}(0,T)<\infty}.
\end{align}

We first give the regularity criteria in the following Lemma \ref{lemma:regularity criteria}, and then use it to prove the persistence of regularity and the property of finite mass density in Proposition \ref{lemma:persistence of regularity}. Subsequently, we complete the proof of Lemma \ref{lemma:regularity criteria} with the help of an $L_{t}^{1}L_{v}^{2,r}H_{x}^{s}$ bilinear estimate in Lemma \ref{lemma:bilinear estimate,gain term,regularity bound}, the proof of which we postpone to the end of this section.

\begin{lemma}[Regularity Criteria]\label{lemma:regularity criteria}
Let $s>1$, $r>1+\ga>0$, $\be\geq 0$.
\begin{enumerate}[$(1)$]
\item \label{criteria,1} If $T_{s,r}^{*}>0$ and $M_{r}(0,T)<\infty$ for some $T\in[0,\infty)$, then $T_{s,r}^{*}>T$.
\item \label{criteria,2} If $T_{s,r+\be}^{*}>0$ and $E_{s,r}(0,T)<\infty$ for some $T\in[0,\infty)$, then $T_{s,r+\be}^{*}>T$.
\end{enumerate}
\end{lemma}
\begin{remark}\label{remark:conjecture}
Lemma \ref{lemma:regularity criteria} actually holds for $d\geq 2$ provided that $s>1$, $r>1+\ga$, and  $L_{x}^{6}$ in \eqref{equ:the integrability bound} are replaced by $s>\frac{d-1}{2}$, $r>\frac{d-1}{2}+\ga$, and $L_{x}^{2d}$ respectively. We only deal with the $d=3$ case here and other cases follow similarly.
 One novel application of the 2D version of Lemma \ref{lemma:regularity criteria} is to solve the \cite[Conjecture 1.1]{chen2021small}.

As shown in \cite[Theorem 1.2]{chen2021small}, the 2D local solution $f^{(2D)}(t)$ exists up to time $T_{0}(s)$ which depends on the lower regularity norm of initial data, i.e. $\n{\lra{v}^{s}\lra{\nabla_{x}}^{s}f^{(2D)}_{0}}_{L_{x,v}^{2}}$ for $s\in (0,\frac{1}{2})$. It was conjectured that the solution
carries the $\frac{1}{2}+$ regularity of the initial data up to time $T_{0}(s)$, that is,
\begin{align}\label{equ:2d blow up}
\n{\lra{v}^{\frac{1}{2}+}\lra{\nabla_{x}}^{\frac{1}{2}+}f^{(2D)}(t)}_{L_{t}^{\infty}(0,T;L_{x,v}^{2})}<\infty,
\quad \text{for any $T\in(0,T_{0}(s))$.}
\end{align}
 Certainly, it is expected that the regularity is propagated for a short time depending on  $\n{\lra{v}^{\frac{1}{2}+}\lra{\nabla_{x}}^{\frac{1}{2}+}f_{0}^{(2D)}}_{L_{x,v}^{2}}$. The point of the conjecture is that the $\frac{1}{2}+$ regularity persists for a time depending solely on a lower regularity norm of initial data.

The local solution $f^{(2D)}(t)$ is constructed by the Kaniel--Shinbrot iteration and thus satisfies the 2D integrability bounds.
Using the 2D version of regularity criterion \eqref{criteria,1} in Lemma \ref{lemma:regularity criteria}, we deduce \eqref{equ:2d blow up} and hence solve the conjecture.
\end{remark}

Now, we get back to our 3D global regularity problem.
\begin{proposition}\label{lemma:persistence of regularity}
Let $f(t)$ be the global solution to the Boltzmann equation constructed in Proposition \ref{lemma:global existence,full}. Then we have
\begin{align}
&f(t)\in C([0,\infty);L_{v}^{2,s+\ga}H_{x}^{s}),\\
&\n{\lra{\nabla_{x}}^{s}\lra{v}^{s+\ga}Q^{\pm}(f,f)}_{L_{t}^{1}(0,T;L_{x,v}^{2})}<\infty  \quad \text{for all $T\in [0,\infty)$.}
\end{align}
Furthermore, the solution satisfies the properties of the persistence of regularity and finite mass as follows.
\begin{enumerate}
\item \label{propertity,persistence}
If $f_{0}\in L_{v}^{2,s+\ga+\be}H_{x}^{s+\al}$ for some $\al\geq 0$, $\be\geq 0$, then we have
\begin{equation*}
\left\{
\begin{aligned}
&f(t)\in C([0,\infty);L_{v}^{2,s+\ga+\be}H_{x}^{s+\al}),\\
&\n{\lra{\nabla_{x}}^{s+\al}\lra{v}^{s+\ga+\be}Q^{\pm}(f,f)}_{L_{t}^{1}(0,T;L_{x,v}^{2})}<\infty \quad \text{for all $T\in [0,\infty)$.}
\end{aligned}
\right.
\end{equation*}
\item If $f_{0}\in L_{x,v}^{1}$, then $\n{f(t)}_{L_{x,v}^{1}}\leq \n{f_{0}}_{L_{x,v}^{1}}$.
\end{enumerate}
\end{proposition}
\begin{proof}[\textbf{Proof of Proposition $\ref{lemma:persistence of regularity}$}]
Since $f_{0}\in L_{v}^{2,s+\ga}H_{x}^{s}$, by the strong local well-posedness in Theorem \ref{thm:local well-posedness}, there exists a local strong solution $f_{loc}(t)$. By the uniqueness Lemma \ref{lemma:conditional uniqueness,boltzmann}, we have that $f(t)=f_{loc}(t)$ before the lifespan $T(f_{0})$, which implies that $T_{s,s+\ga}^{*}>0$.
On the other hand,
by the a priori bound \eqref{equ:f,Lx6} in Proposition \ref{lemma:global existence,full}, we also have $M_{s+\ga}(0,T)<\infty$ for all $T\in [0,\infty)$. Therefore, using the regularity criterion \eqref{criteria,1} in Lemma \ref{lemma:regularity criteria}, we conclude that $T_{s,s+\ga}^{*}=\infty$ and $T(f_{0})=\infty$.

Furthermore, if $f_{0}\in L_{v}^{2,s+\ga+\be}H_{x}^{s+\al}$ for some $\al\geq 0$, $\be\geq 0$, the same argument shows that $T_{s+\al,s+\ga}^{*}=\infty$. Then using the regularity criterion $\eqref{criteria,2}$ in Lemma \ref{lemma:regularity criteria}, we conclude that $T_{s+\al,s+\ga+\be}^{*}=\infty$.

Next, we follow the approximation scheme in \cite{chen2023derivationboltzmann} to prove the property of finite mass.
If $f_{0}\in L_{x,v}^{1}$, we construct the initial approximation data $f_{0}^{N}(x,v)=\chi(v/N)f_{0}(x,v)$ where $\chi(v)$ is the cutoff function. Then by the global existence in Proposition \ref{lemma:global existence,full} and the property \eqref{propertity,persistence} of persistence of regularity, the global solution $f^{N}(t)$ satisfies that
 \begin{align*}
 \n{f^{N}(t)}_{L_{t}^{\infty}(0,T;L_{v}^{2,s+\ga+\be}H_{x}^{s})}<\infty,
 \end{align*}
 for all $\beta\geq 0$ and $T>0$. Next, we prove the $L_{x,v}^{1}$ conservation law for this solution $f^{N}(t)$.
From the Duhamel formula,
\begin{align}\label{equ:duhamel,approximation solution}
f^{N}(t)=S(t)f_{0}^{N}+\int_{0}^{t}S(t-\tau)Q(f^{N},f^{N})(\tau)d\tau,
\end{align}
thanks to that $f_{0}^{N}\in L_{x,v}^{1}$, we only need to prove the $L_{x,v}^{1}$ integrability for the nonlinear term. By Minkowski and H\"{o}lder inequalities, for $t\in [0,T]$ we have
\begin{align*}
\bbn{\int_{0}^{t}S(t-\tau)Q^{\pm}(f^{N},f^{N})(\tau)d\tau}_{L_{x,v}^{1}}
\leq&
\int_{0}^{t}\bn{Q^{\pm}(f^{N},f^{N})(\tau)}_{L_{x,v}^{1}}d\tau\\
\leq& T\n{f^{N}}_{L_{t}^{\infty}(0,T;L_{x}^{2}L_{v}^{1})}\n{A\lrc{f^{N}}}_{L_{t}^{\infty}(0,T;L_{x}^{2}L_{v}^{\infty})}.
\end{align*}
The weighted estimate gives that
\begin{align*}
\n{f^{N}}_{L_{t}^{\infty}(0,T;L_{x}^{2}L_{v}^{1})}\lesssim \n{\lra{v}^{\frac{3}{2}+}f^{N}}_{L_{t}^{\infty}(0,T;L_{x}^{2}L_{v}^{2})}<\infty.
\end{align*}
For $A\lrc{f^{N}}$, using the endpoint Hardy-Littlewood-Sobolev inequality that
\begin{align}
\int |x-y|^{\ga}|u(y)|dy\lesssim \n{u}_{L^{\frac{3}{3+\ga+\delta}}}^{\frac{1}{2}}\n{u}_{L^{\frac{3}{3+\ga-\delta}}}^{\frac{1}{2}},
\end{align}
we have
\begin{align}
\n{A\lrc{f^{N}}}_{L_{x}^{2}L_{v}^{\infty}}\lesssim \n{f^{N}}_{L_{x}^{2}L_{v}^{\frac{3}{3+\ga-\delta}}}^{\frac{1}{2}}\n{f^{N}}_{L_{x}^{2}L_{v}^{\frac{3}{3+\ga+\delta}}}^{\frac{1}{2}}
\lesssim \n{\lra{v}^{\frac{3}{2}+}f^{N}}_{L_{x}^{2}L_{v}^{2}}<\infty,
\end{align}
where in the last inequality we have used the weighted estimate. Therefore, we have obtained the $L_{x,v}^{1}$ integrability of $f^{N}(t)$. Moreover, taking the $L_{x,v}^{1}$ integration on both side of the Duhamel formula \eqref{equ:duhamel,approximation solution}, we arrive at
\begin{align}\label{equ:L1 conservation,approximation}
\int f^{N}(t,x,v)dxdv=&\int S(t)f_{0}^{N}dxdv +\int \int_{0}^{t}S(t-\tau)Q(f^{N},f^{N})(\tau)d\tau dxdv\\
=&\int f_{0}^{N}(x,v)dxdv ,\notag
\end{align}
where in the last equality we have used the $L_{x,v}^{1}$ conservation law of the flow map $S(t)$ and the cancellation property between the gain and loss terms. Notice that the Lipschitz continuous of the solution map gives that
\begin{align*}
\n{f^{N}(t)-f(t)}_{C([0,T];L_{v}^{2,s+\ga}H_{x}^{s})}\lesssim \n{f_{0}^{N}-f}_{L_{v}^{2,s+\ga}H_{x}^{s}}\to 0,
\end{align*}
which implies the pointwise convergence (up to a subsequence). By the non-negativity of $f^{N}(t)$, Fatou's lemma, and the $L_{x,v}^{1}$ uniform estimate \eqref{equ:L1 conservation,approximation},
we have
\begin{align*}
\int f(t,x,v) dxdv \leq \liminf_{N\to \infty}\int f^{N}(t,x,v) dxdv \leq\int f_{0}(x,v) dxdv,
\end{align*}
which completes the proof of the finite mass density.
\end{proof}

We devote the following to proving Lemma \ref{lemma:regularity criteria}, which plays a crucial role in the proof of Proposition \ref{lemma:persistence of regularity}.

\begin{proof}[\textbf{Proof of Lemma $\ref{lemma:regularity criteria}$}]
To obtain the regularity criteria \eqref{criteria,1} and \eqref{criteria,2}, it suffices to prove the following two local estimates respectively.
\begin{enumerate}[$(1)$]
\item
If $0<T_{s,r}^{*}<\infty$ and $M_{r}(0,T_{s,r}^{*})<\infty$, there exists $T_{1}>0$ $($depending on $M_{r}(0,T_{s,r}^{*}))$ such that
\begin{align}\label{equ:regularity bound,energy inequality}
E_{s,r}(t_{0},t_{0}+t_{1})\leq 2E_{s,r}(0,t_{0})
\end{align}
for all $t_{0}\in [0,T_{s,r}^{*})$, $t_{1}\in [0,T_{1}]$ satisfying $t_{0}+t_{1}<T_{s,r}^{*}$.
\item
If $0<T_{s,r+\be}^{*}<\infty$ and $E_{s,r}(0,T_{s,r+\be}^{*})<\infty$, there exists $T_{1}>0$ $($depending on $E_{s,r}(0,T_{s,r+\be}^{*})$$)$ such that
\begin{align}\label{equ:regularity bound,energy inequality,2}
E_{s,r+\be}(t_{0},t_{0}+t_{1})\leq 2E_{s,r+\be}(0,t_{0})
\end{align}
for all $t_{0}\in [0,T_{s,r+\be}^{*})$, $t_{1}\in [0,T_{1}]$ satisfying $t_{0}+t_{1}<T_{s,r+\be}^{*}$.
\end{enumerate}
Indeed, for the regularity criterion $\eqref{criteria,1}$ in Lemma \ref{lemma:regularity criteria}, by a contradiction argument we might as well assume $0<T_{s,r}^{*}\leq T<\infty$, which implies that $M_{r}(0,T_{s,r}^{*})\leq M_{r}(0,T)<\infty$.
 Then by taking $t_{0}=T_{s,r}^{*}-T_{1}$ and $t_{1}\in[0,T_{1})$, we use the local estimate \eqref{equ:regularity bound,energy inequality} to get an upper bound that
\begin{align*}
E_{s,r}(0,t_{0}+t_{1})\leq E_{s,r}(0,t_{0})+E_{s,r}(t_{0},t_{0}+t_{1})\leq 3E_{s,r}(0,t_{0})<\infty,
\end{align*}
which implies
\begin{align}\label{equ:bounded,blowup point}
\lim_{t\nearrow T_{s,r}^{*}}E_{s,r}(0,t)\leq 3E_{s,r}(0,t_{0})<\infty.
\end{align}
Together with the local well-posedness and the blow-up criterion in Theorem \ref{thm:local well-posedness}, the estimate \eqref{equ:bounded,blowup point} gives a contradiction to the definition \eqref{equ:definition,T0} of $T_{s,r}^{*}$.
In the same argument, we also have the regularity criterion \eqref{criteria,2} by using \eqref{equ:regularity bound,energy inequality,2}.

Next, we are left to prove the local estimates \eqref{equ:regularity bound,energy inequality}--\eqref{equ:regularity bound,energy inequality,2}.
For \eqref{equ:regularity bound,energy inequality},
using the Duhamel formula that
\begin{align}\label{equ:duhamel,t0}
f(t_{0}+t)=S(t)f(t_{0})+\int_{0}^{t}S(t-\tau)\mathcal{N}\lrc{f(t_{0}+\tau)}d\tau,
\end{align}
we have
\begin{align}\label{equ:duhamel,t0,nonlinear estimate}
E_{s,r}(t_{0},t_{0}+t_{1})\leq E_{s,r}(0,t_{0})+2\n{\lra{\nabla_{x}}^{s}\lra{v}^{r}Q^{\pm}(f,f)}_{L_{t}^{1}(t_{0},t_{0}+t_{1};L_{x,v}^{2})}.
\end{align}
Therefore, we are left to estimate the term $\n{\lra{\nabla_{x}}^{s}\lra{v}^{r}Q^{\pm}(f,f)}_{L_{t}^{1}(t_{0},t_{0}+t_{1};L_{x,v}^{2})}$.

We first handle the loss term $Q^{-}$. By the fractional Leibliz rule in Lemma \ref{lemma:generalized leibniz rule} and H\"{o}lder inequality,
\begin{align}\label{equ:regularity bound,loss term,step1}
&\n{\lra{\nabla_{x}}^{s}\lra{v}^{r}Q^{-}(f,f)}_{L_{t}^{1}(t_{0},t_{0}+t_{1};L_{x,v}^{2})}\\
\lesssim& \sqrt{t_{1}}\n{A\lrc{f}}_{L_{t}^{2}(t_{0},t_{0}+t_{1};L_{x}^{\infty}L_{v}^{\infty})}
\n{\lra{\nabla_{x}}^{s}\lra{v}^{r}f}_{L_{t}^{\infty}(t_{0},t_{0}+t_{1};L_{x,v}^{2})}\notag\\
&+\sqrt{t_{1}}\n{\lra{v}^{r}f}_{L_{t}^{\infty}(t_{0},t_{0}+t_{1};L_{v}^{2}L_{x}^{6})}
\n{\lra{\nabla_{x}}^{s}A\lrc{f}}_{L_{t}^{2}(t_{0},t_{0}+t_{1};L_{x}^{3}L_{v}^{\infty})}\notag\\
\leq& \sqrt{t_{1}}M_{r}(0,T_{s,r}^{*})
\lrs{E_{s,r}(t_{0},t_{0}+t_{1})+\n{\lra{\nabla_{x}}^{s}A\lrc{f}}_{L_{t}^{2}(t_{0},t_{0}+t_{1};L_{x}^{3}L_{v}^{\infty})}}.\notag
\end{align}
With $r>1+\ga$, we use estimate \eqref{equ:A,f,Lx3,estimate} in Lemma \ref{lemma:A,f,Lx,estimate} to get
\begin{align}\label{equ:regularity bound,loss term,step2}
&\n{\lra{\nabla_{x}}^{s}A\lrc{f}}_{L_{t}^{2}(t_{0},t_{0}+t_{1};L_{x}^{3}L_{v}^{\infty})}\\
\lesssim& \n{\lra{\nabla_{x}}^{s}\lra{v}^{r}f(t_{0})}_{L_{x}^{2}L_{v}^{2}}
+\n{\lra{\nabla_{x}}^{s}\lra{v}^{r}\mathcal{N}\lrc{f}}_{L_{t}^{1}(t_{0},t_{0}+t_{1};L_{x}^{2}L_{v}^{2})}\notag\\
\lesssim& E_{s,r}(t_{0},t_{0}+t_{1}).\notag
\end{align}
Combining estimates \eqref{equ:regularity bound,loss term,step1} and \eqref{equ:regularity bound,loss term,step2}, we obtain
\begin{align}\label{equ:regularity bound,loss term}
\n{\lra{\nabla_{x}}^{s}\lra{v}^{r}Q^{-}(f,f)}_{L_{t}^{1}(t_{0},t_{0}+t_{1};L_{x,v}^{2})}\lesssim \sqrt{t_{1}}M_{r}(0,T_{s,r}^{*})E_{s,r}(t_{0},t_{0}+t_{1}).
\end{align}

We then deal with the gain term $Q^{+}$.
By the Duhamel formula \eqref{equ:duhamel,t0}, we rewrite
\begin{align}
Q^{+}(f,f)(t_{0}+t)=I_{1}+I_{2}+I_{3}+I_{4},
\end{align}
where
\begin{align}
&I_{1}=Q^{+}(S(t)f(t_{0}),S(t)f(t_{0})),\\
&I_{2}=\int_{0}^{t} Q^{+}(S(t-\tau_{1})\mathcal{N}\lrc{f(t_{0}+\tau_{1})},S(t)f(t_{0}))d\tau_{1},\\
&I_{3}=\int_{0}^{t} Q^{+}(S(t)f(t_{0}),S(t-\tau_{2})\mathcal{N}\lrc{f(t_{0}+\tau)})d\tau_{2},\\
&I_{4}=\int_{0}^{t}\int_{0}^{t} Q^{+}(S(t-\tau_{1})\mathcal{N}\lrc{f(t_{0}+\tau_{1})},S(t-\tau_{2})\mathcal{N}\lrc{f(t_{0}+\tau_{2})})d\tau_{1}d\tau_{2}.
\end{align}

To control these terms $I_{1}$--$I_{4}$, we need an $L_{t}^{1}L_{v}^{2,r}H_{x}^{s}$ bilinear estimate as follows:
\begin{align}\label{equ:bilinear estimate,gain term,regularity bound,used}
&\n{\lra{\nabla_{x}}^{s}\lra{v}^{r}Q^{+}(S(t)f_{0},S(t)g_{0})}_{L_{t}^{1}(0,t_{1};L_{x,v}^{2})}\\
\lesssim& \sqrt{t_{1}}\n{\lra{\nabla_{x}}^{s}\lra{v}^{r}f_{0}}_{L_{x,v}^{2}}
\n{\lra{v}^{1+\ga}g_{0}}_{L_{v}^{2}L_{x}^{6}}+\sqrt{t_{1}}
\n{\lra{v}^{1+\ga}f_{0}}_{L_{v}^{2}L_{x}^{6}}\n{\lra{\nabla_{x}}^{s}\lra{v}^{r}g_{0}}_{L_{x,v}^{2}},\notag
\end{align}
the proof of which relies on the frequency analysis techniques like in Section \ref{section:Scaling-invariant Bilinear Estimate for Gain Term}, and is thus postponed to the end. For convenience, we take the notations
\begin{align}
\n{f}_{X}=\n{\lra{\nabla_{x}}^{s}\lra{v}^{r}f}_{L_{x,v}^{2}},\quad \n{f}_{Y}=\n{\lra{v}^{1+\ga}f}_{L_{v}^{2}L_{x}^{6}}.
\end{align}

For $I_{1}$, by the bilinear estimate \eqref{equ:bilinear estimate,gain term,regularity bound,used}, with $r\geq 1+\ga$ we have
\begin{align}\label{equ:regularity bound,gain term,I1}
\n{I_{1}}_{L_{t}^{1}(t_{0},t_{0}+t_{1};X)}\lesssim \sqrt{t_{1}}\n{f(t_{0})}_{Y}\n{f(t_{0})}_{X}\leq \sqrt{t_{1}}M_{r}(0,T_{s,r}^{*})E_{s,r}(t_{0},t_{0}+s).
\end{align}

For $I_{2}$, by Minkowski inequality, we get
\begin{align*}
\n{I_{2}}_{L_{t}^{1}(t_{0},t_{0}+t_{1};X)}\leq& \bbn{\int_{0}^{t} \bn{Q^{+}(S(t-\tau_{1})\mathcal{N}\lrc{f(t_{0}+\tau_{1})},S(t)f(t_{0}))}_{X}d\tau_{1}}_{L_{t}^{1}(0,t_{1})}\\
\leq& \int_{0}^{t_{1}}\bn{Q^{+}(S(t-\tau_{1})\mathcal{N}\lrc{f(t_{0}+\tau_{1})},S(t)f(t_{0}))}_{L_{t}^{1}(0,t_{1};X)}d\tau_{1}.
\end{align*}
Using again the bilinear estimate \eqref{equ:bilinear estimate,gain term,regularity bound,used}, we obtain
\begin{align}\label{equ:regularity bound,gain term,I2}
&\n{I_{2}}_{L_{t}^{1}(t_{0},t_{0}+t_{1};X)}\\
\lesssim&
\int_{0}^{t_{1}}\sqrt{t_{1}}\n{S(-\tau_{1})\mathcal{N}\lrc{f(t_{0}+\tau_{1})}}_{X}\n{f(t_{0})}_{Y}d\tau_{1}\notag\\
&+\int_{0}^{t_{1}}\sqrt{t_{1}}\n{S(-\tau_{1})\mathcal{N}\lrc{f(t_{0}+\tau_{1})}}_{Y}\n{f(t_{0})}_{X}d\tau_{1}\notag\\
\lesssim&\sqrt{t_{1}}\n{Q^{\pm}(f,f)}_{L_{t}^{1}(t_{0},t_{0}+t_{1};X)}\n{f(t_{0})}_{Y}+
\sqrt{t_{1}}\n{Q^{\pm}(f,f)}_{L_{t}^{1}(t_{0},t_{0}+t_{1};Y)}\n{f(t_{0})}_{X}\notag\\
\lesssim& \sqrt{t_{1}}M_{r}(0,T_{s,r}^{*})E_{s,r}(t_{0},t_{0}+t_{1}).\notag
\end{align}

Since the term $I_{3}$ can be estimated in the same way as $I_{2}$, we also obtain
\begin{align}\label{equ:regularity bound,gain term,I3}
\n{I_{3}}_{L_{t}^{1}(t_{0},t_{0}+t_{1};X)}\lesssim \sqrt{t_{1}}M_{r}(0,T_{s,r}^{*})E_{s,r}(t_{0},t_{0}+t_{1}).
\end{align}

For $I_{4}$, by Minkowski inequality, we get
\begin{align*}
&\n{I_{4}}_{L_{t}^{1}(t_{0},t_{0}+t_{1};X)}\\
\leq& \bbn{\int_{0}^{t}\int_{0}^{t}  \bn{ Q^{+}(S(t-\tau_{1})\mathcal{N}\lrc{f(t_{0}+\tau_{1})},S(t-\tau_{2})\mathcal{N}\lrc{f(t_{0}+\tau_{2})})
}_{X}d\tau_{1}d\tau_{2}}_{L_{t}^{1}(0,t_{1})}\\
\leq& \int_{0}^{t_{1}}\int_{0}^{t_{1}}\bn{Q^{+}(S(t-\tau_{1})\mathcal{N}\lrc{f(t_{0}+\tau_{1})},S(t-\tau_{2})\mathcal{N}\lrc{f(t_{0}+\tau_{2})})}_{L_{t}^{1}(0,t_{1};X)}d\tau_{1}d\tau_{2}.
\end{align*}
By the bilinear estimate \eqref{equ:bilinear estimate,gain term,regularity bound,used} again, we have
\begin{align}\label{equ:regularity bound,gain term,I4}
&\n{I_{4}}_{L_{t}^{1}(t_{0},t_{0}+t_{1};X)}\\
\lesssim&
\int_{0}^{t_{1}}\int_{0}^{t_{1}}\sqrt{t_{1}}\n{S(-\tau_{1})\mathcal{N}\lrc{f(t_{0}+\tau_{1})}}_{X}
\n{S(-\tau_{2})\mathcal{N}\lrc{f(t_{0}+\tau_{2})}}_{Y}d\tau_{1}d\tau_{2}\notag\\
&+\int_{0}^{t_{1}}\int_{0}^{t_{1}}\sqrt{t_{1}}\n{S(-\tau_{1})\mathcal{N}\lrc{f(t_{0}+\tau_{1})}}_{Y}
\n{S(-\tau_{2})\mathcal{N}\lrc{f(t_{0}+\tau_{2})}}_{X}d\tau_{1}d\tau_{2}\notag\\
\lesssim&\sqrt{t_{1}}\n{Q^{\pm}(f,f)}_{L_{t}^{1}(t_{0},t_{0}+t_{1};X)}\n{Q^{\pm}(f,f)}_{L_{t}^{1}(t_{0},t_{0}+t_{1};Y)}\notag\\
\lesssim& \sqrt{t_{1}}M_{r}(0,T_{s,r}^{*})E_{s,r}(t_{0},t_{0}+t_{1}).\notag
\end{align}

Putting estimates \eqref{equ:regularity bound,gain term,I1}--\eqref{equ:regularity bound,gain term,I4} together, we arrive at
\begin{align}\label{equ:equ:regularity bound,gain term,final}
\n{\lra{\nabla_{x}}^{s}\lra{v}^{r}Q^{+}(f,f)}_{L_{t}^{1}(t_{0},t_{0}+t_{1};L_{x,v}^{2})}\lesssim \sqrt{t_{1}}M_{r}(0,T_{s,r}^{*})E_{s,r}(t_{0},t_{0}+t_{1}),
\end{align}
which, together with \eqref{equ:duhamel,t0,nonlinear estimate} and the loss term estimate \eqref{equ:regularity bound,loss term}, implies that
\begin{align*}
E_{s,r}(t_{0},t_{0}+t_{1})\leq& E_{s,r}(0,t_{0})+2\n{\lra{\nabla_{x}}^{s}\lra{v}^{r}Q^{\pm}(f,f)}_{L_{t}^{1}(t_{0},t_{0}+t_{1};L_{x,v}^{2})}\\
\leq& E_{s,r}(0,t_{0})+\sqrt{t_{1}}CM_{r}(0,T_{s,r}^{*})E_{s,r}(t_{0},t_{0}+t_{1}).
\end{align*}
By choosing $\sqrt{t_{1}}\leq (2CM_{r}(0,T_{s,r}^{*}))^{-1}$, we complete the proof of \eqref{equ:regularity bound,energy inequality}.

For $\eqref{equ:regularity bound,energy inequality,2}$, we first deal with the loss term. By the fractional Leibliz rule in Lemma \ref{lemma:generalized leibniz rule} and Sobolev inequality, we have
\begin{align}\label{equ:regularity bound,loss term,2}
&\n{\lra{\nabla_{x}}^{s}\lra{v}^{r+\be}Q^{-}(f,f)}_{L_{t}^{1}(t_{0},t_{0}+t_{1};L_{x,v}^{2})}\\
\lesssim& \sqrt{t_{1}}\n{A\lrc{f}}_{L_{t}^{2}(t_{0},t_{0}+t_{1};L_{x}^{\infty}L_{v}^{\infty})}
\n{\lra{\nabla_{x}}^{s}\lra{v}^{r+\be}f}_{L_{t}^{\infty}(t_{0},t_{0}+t_{1};L_{x,v}^{2})}\notag\\
&+\sqrt{t_{1}}\n{\lra{v}^{r+\be}f}_{L_{t}^{\infty}(t_{0},t_{0}+t_{1};L_{v}^{2}L_{x}^{6})}
\n{\lra{\nabla_{x}}^{s}A\lrc{f}}_{L_{t}^{2}(t_{0},t_{0}+t_{1};L_{x}^{3}L_{v}^{\infty})}\notag\\
\lesssim& \sqrt{t_{1}}\n{\lra{\nabla_{x}}^{s}A\lrc{f}}_{L_{t}^{2}(t_{0},t_{0}+t_{1};L_{x}^{3}L_{v}^{\infty})}
\n{\lra{\nabla_{x}}^{s}\lra{v}^{r+\be}f}_{L_{t}^{\infty}(t_{0},t_{0}+t_{1};L_{x,v}^{2})}\notag\\
\leq& \sqrt{t_{1}}E_{s,r}(0,T_{s,r+\be}^{*})E_{s,r+\be}(t_{0},t_{0}+t_{1}),\notag
\end{align}
where in the last inequality we have used estimate \eqref{equ:A,f,Lx3,estimate} in Lemma \ref{lemma:A,f,Lx,estimate} to get
\begin{align*}
&\n{\lra{\nabla_{x}}^{s}A\lrc{f}}_{L_{t}^{2}(t_{0},t_{0}+t_{1};L_{x}^{3}L_{v}^{\infty})}\\
\lesssim& \n{\lra{\nabla_{x}}^{s}\lra{v}^{r}f(t_{0})}_{L_{x}^{2}L_{v}^{2}}
+\n{\lra{\nabla_{x}}^{s}\lra{v}^{r}\mathcal{N}\lrc{f}}_{L_{t}^{1}(t_{0},t_{0}+t_{1};L_{x}^{2}L_{v}^{2})}\notag\\
\lesssim& E_{s,r}(0,T_{s,r+\be}^{*}).\notag
\end{align*}
For the gain term, repeating the proof of \eqref{equ:equ:regularity bound,gain term,final}, we also have
\begin{align}\label{equ:equ:regularity bound,gain term,final,2}
\n{\lra{\nabla_{x}}^{s}\lra{v}^{r+\be}Q^{+}(f,f)}_{L_{t}^{1}(t_{0},t_{0}+t_{1};L_{x,v}^{2})}\lesssim & \sqrt{t_{1}}M_{r}(0,T_{s,r+\be}^{*})E_{s,r+\be}(t_{0},t_{0}+t_{1})\\
\lesssim&\sqrt{t_{1}}E_{s,r}(0,T_{s,r+\be}^{*})E_{s,r+\be}(t_{0},t_{0}+t_{1}).\notag
\end{align}
Combining estimates \eqref{equ:regularity bound,loss term,2} and \eqref{equ:equ:regularity bound,gain term,final,2}, by choosing $\sqrt{t_{1}}\leq (2 CE_{s,r}(0,T_{s,r+\be}^{*}))^{-1}$, we complete the proof of \eqref{equ:regularity bound,energy inequality,2}.

\end{proof}

In the following, we present the proof of the bilinear estimate \eqref{equ:bilinear estimate,gain term,regularity bound,used}, which is essential for the proof of Lemma \ref{lemma:regularity criteria}.

\begin{lemma}\label{lemma:bilinear estimate,gain term,regularity bound}
Let $s>1$, $r\geq 1+\ga>0$. It holds that
\begin{align}\label{equ:bilinear estimate,gain term,regularity bound}
&\n{\lra{\nabla_{x}}^{s}\lra{v}^{r}Q^{+}(S(t)f_{0},S(t)g_{0})}_{L_{t}^{1}(0,T;L_{x,v}^{2})}\\
\lesssim& |T|^{\frac{1}{2}}\n{\lra{\nabla_{x}}^{s}\lra{v}^{r}f_{0}}_{L_{x,v}^{2}}\n{\lra{v}^{1+\ga}g_{0}}_{L_{v}^{2}L_{x}^{6}}+|T|^{\frac{1}{2}}\n{\lra{v}^{1+\ga}f_{0}}_{L_{v}^{2}L_{x}^{6}}\n{\lra{\nabla_{x}}^{s}\lra{v}^{r}g_{0}}_{L_{x,v}^{2}}.\notag
\end{align}
\end{lemma}
\begin{proof}
By Plancherel identity, it suffices to prove that
\begin{align}\label{equ:bilinear estimate,gain term,regularity bound,xi}
&\n{\lra{\nabla_{x}}^{s}\lra{\nabla_{\xi}}^{r}\wt{Q}^{+}(U(t)\wt{f}(t_{0}),U(t)\wt{g}(t_{0}))}_{L_{t}^{1}(0,T;L_{x,\xi}^{2})}\\
\lesssim& |T|^{\frac{1}{2}}\n{\lra{\nabla_{x}}^{s}\lra{v}^{r}f_{0}}_{L_{x,v}^{2}}\n{\lra{v}^{1+\ga}g_{0}}_{L_{v}^{2}L_{x}^{6}}
+|T|^{\frac{1}{2}}\n{\lra{v}^{1+\ga}f_{0}}_{L_{v}^{2}L_{x}^{6}}\n{\lra{\nabla_{x}}^{s}\lra{v}^{r}g_{0}}_{L_{x,v}^{2}}.\notag
\end{align}
By duality, \eqref{equ:bilinear estimate,gain term,regularity bound,xi} is equivalent to
\begin{align}\label{equ:bilinear estimate,gain term,regularity bound,xi,duality}
&\int \wt{Q}^{+}(U(t)\wt{f}_{0},U(t)\wt{g}_{0}) h dx d\xi dt\\
\lesssim& |T|^{\frac{1}{2}}\n{\lra{\nabla_{x}}^{s}\lra{v}^{r}f_{0}}_{L_{x,v}^{2}}
\n{\lra{v}^{1+\ga}g_{0}}_{L_{v}^{2}L_{x}^{6}}\n{\lra{\nabla_{x}}^{-s}\lra{\nabla_{\xi}}^{-r}h}_{L_{t}^{\infty}(0,T;L_{x,\xi}^{2})}\notag\\
&+|T|^{\frac{1}{2}}\n{\lra{v}^{1+\ga}f_{0}}_{L_{v}^{2}L_{x}^{6}}\n{\lra{\nabla_{x}}^{s}\lra{v}^{r}g_{0}}_{L_{x,v}^{2}}
\n{\lra{\nabla_{x}}^{-s}\lra{\nabla_{\xi}}^{-r}h}_{L_{t}^{\infty}(0,T;L_{x,\xi}^{2})}.\notag
\end{align}

We denote by $I$ the integral in \eqref{equ:bilinear estimate,gain term,regularity bound,xi,duality} and
insert a Littlewood-Paley decomposition such that
\begin{align*}
I=\sum_{\substack{M,M_{1},M_{2}\\N,N_{1},N_{2}}} I_{M,M_{1},M_{2},N,N_{1},N_{2}}
\end{align*}
where
\begin{align*}
I_{M,M_{1},M_{2},N,N_{1},N_{2}}= \int \wt{Q}^{+}(P_{N_{1}}^{x}P_{M_{1}}^{\xi}\wt{f},P_{N_{2}}^{x}P_{M_{2}}^{\xi}\wt{g}) P_{N}^{x}P_{M}^{\xi}h dx d\xi dt,
\end{align*}
with $\wt{f}(t)=U(t)\wt{f}_{0}$ and $\wt{g}(t)=U(t)\wt{g}_{0}$. In the same way as the frequency analysis of \eqref{equ:property,constraint,projector}--\eqref{equ:property,constraint,projector,v,variable}, we have the constraints that $N\lesssim
\max\lrs{N_{1},N_{2}}$ and $M\lesssim\max\lrs{M_{1},M_{2}}$.

We divide the sum into four cases as follows

Case A. $M_{1}\geq M_{2}$, $N_{1}\geq N_{2}$.

Case B. $M_{1}\leq M_{2}$, $N_{1}\geq N_{2}$.

Case C. $M_{1}\geq M_{2}$, $N_{1}\leq N_{2}$.

Case D. $M_{1}\leq  M_{2}$, $N_{1}\leq N_{2}$.

We only handle Cases A and B, as Cases C and D can be dealt with in a similar way.

\textbf{Case A. $M_{1}\leq M_{2}$, $N_{1}\geq N_{2}$.}

Let $I_{A}$ denote the integral restricted to the Case A.

\begin{align*}
I_{A}=&\sum_{\substack{M_{1}\geq M_{2},M_{1}\gtrsim M\\N_{1}\geq N_{2},N_{1}\gtrsim N }}\int \wt{Q}^{+}(P_{N_{1}}^{x}P_{M_{1}}^{\xi}\wt{f},P_{N_{2}}^{x}P_{M_{2}}^{\xi}\wt{g}) P_{N}^{x}P_{M}^{\xi}h dx d\xi dt\\
=&\sum_{\substack{M_{1}\gtrsim M\\N_{1}\gtrsim N }}\int \wt{Q}^{+}(P_{N_{1}}^{x}P_{M_{1}}^{\xi}\wt{f},P_{\leq N_{1}}^{x}P_{\leq M_{1}}^{\xi}\wt{g}) P_{N}^{x}P_{M}^{\xi}h dx d\xi dt
\end{align*}
where in the last equality we have done the sum in $M_{2}$ and $N_{2}$. By using H\"{o}lder inequality and then estimate \eqref{equ:Q+,bilinear estimate,L3,PM,f,g} in Lemma \ref{equ:Q+,bilinear estimate}, we have
\begin{align*}
I_{A}\leq& \int \sum_{\substack{M_{1}\gtrsim M\\N_{1}\gtrsim N }} \n{\wt{Q}^{+}(P_{N_{1}}^{x}P_{M_{1}}^{\xi}\wt{f},P_{\leq N_{1}}^{x}P_{\leq M_{1}}^{\xi}\wt{g})}_{L_{\xi}^{2}} \n{P_{N}^{x}P_{M}^{\xi}h}_{L_{\xi}^{2}} dx  dt\\
\lesssim& \int \sum_{\substack{M_{1}\gtrsim M\\N_{1}\gtrsim N }} \n{P_{N_{1}}^{x}P_{M_{1}}^{\xi}\wt{f}}_{L_{\xi}^{3}}\n{P_{\leq N_{1}}^{x}P_{\leq M_{1}}^{\xi}\wt{g}}_{L_{\xi}^{\frac{6}{1-2\ga}}} \n{P_{N}^{x}P_{M}^{\xi}h}_{L_{\xi}^{2}} dx  dt.
\end{align*}
By Bernstein inequality and Sobolev inequality that $W^{1+\ga,2}\hookrightarrow L^{\frac{6}{1-2\ga}}$,
\begin{align*}
I_{A}\leq& \int \sum_{\substack{M_{1}\gtrsim M\\N_{1}\gtrsim N }} \frac{M^{r}}{M_{1}^{r}}
\n{\lra{\nabla_{\xi}}^{r}P_{N_{1}}^{x}P_{M_{1}}^{\xi}\wt{f}}_{L_{\xi}^{3}}\n{P_{\leq N_{1}}^{x}\lra{\nabla_{\xi}}^{1+\ga}\wt{g}}_{L_{\xi}^{2}} \n{\lra{\nabla_{\xi}}^{-r}P_{N}^{x}P_{M}^{\xi}h}_{L_{\xi}^{2}} dx  dt.
\end{align*}
With $r>0$, we use Cauchy-Schwarz in $M$ and $M_{1}$ to get
\begin{align*}
I_{A}\leq& \int \sum_{N_{1}\gtrsim N}\n{P_{\leq N_{1}}^{x}\lra{\nabla_{\xi}}^{1+\ga}\wt{g}}_{L_{\xi}^{2}}
\lrs{\sum_{M_{1}\gtrsim M}\frac{M^{r}}{M_{1}^{r}}
\n{\lra{\nabla_{\xi}}^{r}P_{N_{1}}^{x}P_{M_{1}}^{\xi}\wt{f}}_{L_{\xi}^{3}}^{2}}^{\frac{1}{2}}\\
&\lrs{\sum_{M_{1}\gtrsim M}\frac{M^{r}}{M_{1}^{r}}
\n{\lra{\nabla_{\xi}}^{-r}P_{N}^{x}P_{M}^{\xi}h}_{L_{\xi}^{2}}^{2}}^{\frac{1}{2}}dxdt\\
\lesssim &\int \sum_{N_{1}\gtrsim N}\n{P_{\leq N_{1}}^{x}\lra{\nabla_{\xi}}^{1+\ga}\wt{g}}_{L_{\xi}^{2}}
\n{\lra{\nabla_{\xi}}^{r}P_{N_{1}}^{x}P_{M_{1}}^{\xi}\wt{f}}_{l_{M_{1}}^{2}L_{\xi}^{3}}\\
&\n{\lra{\nabla_{\xi}}^{-r}P_{N}^{x}h}_{L_{\xi}^{2}} dxdt.
\end{align*}
By H\"{o}lder inequality in the $x$-variable,
\begin{align*}
I_{A}\lesssim &
\int \sum_{N_{1}\gtrsim N}\n{P_{\leq N_{1}}^{x}\lra{\nabla_{\xi}}^{1+\ga}\wt{g}}_{L_{x}^{6}L_{\xi}^{2}}
\n{\lra{\nabla_{\xi}}^{r}P_{N_{1}}^{x}P_{M_{1}}^{\xi}\wt{f}}_{L_{x}^{3}l_{M_{1}}^{2}L_{\xi}^{3}}\n{\lra{\nabla_{\xi}}^{-r}P_{N}^{x}h}_{L_{x}^{2}L_{\xi}^{2}} dt.
\end{align*}
By Minkowski inequality and Bernstein inequality,
\begin{align*}
I_{A}\lesssim &
\int \sum_{N_{1}\gtrsim N}\n{P_{\leq N_{1}}^{x}\lra{\nabla_{\xi}}^{1+\ga}\wt{g}}_{L_{x}^{6}L_{\xi}^{2}}
\n{\lra{\nabla_{\xi}}^{r}P_{N_{1}}^{x}P_{M_{1}}^{\xi}\wt{f}}_{l_{M_{1}}^{2}L_{x}^{3}L_{\xi}^{3}}\n{\lra{\nabla_{\xi}}^{-r}P_{N}^{x}h}_{L_{x}^{2}L_{\xi}^{2}} dt\\
\lesssim &
\int \n{\lra{\nabla_{\xi}}^{1+\ga}\wt{g}}_{L_{x}^{6}L_{\xi}^{2}}
\sum_{N_{1}\gtrsim N} \frac{N^{s}}{N_{1}^{s}}\n{\lra{\nabla_{x}}^{s}\lra{\nabla_{\xi}}^{r}P_{N_{1}}^{x}P_{M_{1}}^{\xi}\wt{f}}_{l_{M_{1}}^{2}L_{x}^{3}L_{\xi}^{3}}\\
&\n{\lra{\nabla_{x}}^{-s}\lra{\nabla_{\xi}}^{-r}P_{N}^{x}h}_{L_{x}^{2}L_{\xi}^{2}} dt
\end{align*}
where in the last inequality we have used that
\begin{align*}
\n{P_{\leq N_{1}}^{x}\lra{\nabla_{\xi}}^{1+\ga}\wt{g}}_{L_{x}^{6}L_{\xi}^{2}}=&\n{\mathcal{F}^{-1}(\vp_{\leq N_{1}}^{x})*\lra{\nabla_{\xi}}^{1+\ga}\wt{g}}_{L_{x}^{6}L_{\xi}^{2}}\\
\lesssim& \bn{\mathcal{F}^{-1}(\vp_{\leq N_{1}}^{x})*\n{\lra{\nabla_{\xi}}^{1+\ga}\wt{g}}_{L_{\xi}^{2}}}_{L_{x}^{6}}\lesssim  \n{\lra{\nabla_{\xi}}^{1+\ga}\wt{g}}_{L_{x}^{6}L_{\xi}^{2}}.
\end{align*}
By Cauchy-Schwarz in $N$ and $N_{1}$,
\begin{align*}
I_{A}\lesssim &\int \n{\lra{\nabla_{\xi}}^{1+\ga}\wt{g}}_{L_{x}^{6}L_{\xi}^{2}}\n{\lra{\nabla_{x}}^{s}\lra{\nabla_{\xi}}^{r}
P_{N_{1}}^{x}P_{M_{1}}^{\xi}\wt{f}}_{l_{N_{1}}^{2}l_{M_{1}}^{2}L_{x}^{3}L_{\xi}^{3}}\n{\lra{\nabla_{x}}^{-s}\lra{\nabla_{\xi}}^{-r}h}_{L_{x}^{2}L_{\xi}^{2}} dt.
\end{align*}
By H\"{o}lder inequality in the $t$-variable,
\begin{align*}
I_{A}\lesssim &\n{\lra{\nabla_{\xi}}^{1+\ga}\wt{g}}_{L_{t}^{2}(0,T;L_{x}^{6}L_{\xi}^{2})}\n{\lra{\nabla_{x}}^{s}\lra{\nabla_{\xi}}^{r}
P_{N_{1}}^{x}P_{M_{1}}^{\xi}\wt{f}}_{L_{t}^{2}(0,T;l_{N_{1}}^{2}l_{M_{1}}^{2}L_{x}^{3}L_{\xi}^{3})}\\
&\n{\lra{\nabla_{x}}^{-s}\lra{\nabla_{\xi}}^{-r}h}_{L_{t}^{\infty}(0,T;L_{x}^{2}L_{\xi}^{2})}.
\end{align*}
Inserting in $\wt{f}(t)=U(t)\wt{f}_{0}$ and $\wt{g}(t)=U(t)\wt{g}_{0}$, we use Strichartz estimate \eqref{equ:strichartz estimate,linear} to obtain
\begin{align*}
I_{A}\lesssim & \n{\lra{\nabla_{\xi}}^{1+\ga}U(t)\wt{g}_{0}}_{L_{t}^{2}(0,T;L_{x}^{6}L_{\xi}^{2})}
\n{P_{N_{1}}^{x}P_{M_{1}}^{\xi}\wt{f}_{0}}_{l_{N_{1}}^{2}l_{M_{1}}^{2}H_{x}^{s}H_{\xi}^{r}}\n{\lra{\nabla_{x}}^{-s}\lra{\nabla_{\xi}}^{-r}h}_{L_{t}^{\infty}(0,T;L_{x}^{2}L_{\xi}^{2})}\\
\lesssim & |T|^{\frac{1}{2}}\n{\lra{\nabla_{\xi}}^{1+\ga}U(t)\wt{g}_{0}}_{L_{t}^{\infty}(0,T;L_{x}^{6}L_{\xi}^{2})}
\n{\wt{f}_{0}}_{H_{x}^{s}H_{\xi}^{r}}\n{\lra{\nabla_{x}}^{-s}\lra{\nabla_{\xi}}^{-r}h}_{L_{t}^{\infty}(0,T;L_{x}^{2}L_{\xi}^{2})}\\
\lesssim&|T|^{\frac{1}{2}}\n{\lra{v}^{1+\ga}g_{0}}_{L_{v}^{2}L_{x}^{6}}
\n{\lra{\nabla_{x}}^{s}\lra{v}^{r}f_{0}}_{L_{x,v}^{2}}\n{\lra{\nabla_{x}}^{-s}\lra{\nabla_{\xi}}^{-r}h}_{L_{t}^{\infty}(0,T;L_{x}^{2}L_{\xi}^{2})},
\end{align*}
where in the last inequality we have used that
\begin{align*}
\n{\lra{\nabla_{\xi}}^{r}U(t)\wt{g}_{0}}_{L_{t}^{\infty}(0,T;L_{x}^{6}L_{\xi}^{2})}=
\n{\lra{v}^{r}S(t)g_{0}}_{L_{t}^{\infty}(0,T;L_{x}^{6}L_{v}^{2})}\lesssim \n{\lra{v}^{r}g_{0}}_{L_{v}^{2}L_{x}^{6}}.
\end{align*}
Therefore, we have completed the proof of \eqref{equ:bilinear estimate,gain term,regularity bound,xi,duality} for Case A.

\textbf{Case B. $M_{1}\leq M_{2}$, $N_{1}\geq N_{2}$.}

Let $I_{B}$ denote the integral restricted to the Case B.

\begin{align*}
I_{B}=&\sum_{\substack{M_{2}\geq M_{1},M_{2}\gtrsim M\\N_{1}\geq N_{2},N_{1}\gtrsim N }}\int \wt{Q}^{+}(P_{N_{1}}^{x}P_{M_{1}}^{\xi}\wt{f},P_{N_{2}}^{x}P_{M_{2}}^{\xi}\wt{g}) P_{N}^{x}P_{M}^{\xi}h dx d\xi dt\\
=&\sum_{\substack{M_{2}\gtrsim M\\N_{1}\gtrsim N }}\int \wt{Q}^{+}(P_{N_{1}}^{x}P_{\leq M_{2}}^{\xi}\wt{f},P_{\leq N_{1}}^{x}P_{M_{2}}^{\xi}\wt{g}) P_{N}^{x}P_{M}^{\xi}h dx d\xi dt
\end{align*}
where we have done the sum in $M_{1}$ and $N_{2}$. By using H\"{o}lder inequality and then estimate \eqref{equ:Q+,bilinear estimate,L3,PM,g,f} in Lemma \ref{equ:Q+,bilinear estimate}, we have
\begin{align*}
I_{B}\leq &\int \sum_{\substack{M_{2}\gtrsim M\\N_{1}\gtrsim N }}\n{\wt{Q}^{+}(P_{N_{1}}^{x}P_{\leq M_{2}}^{\xi}\wt{f},P_{\leq N_{1}}^{x}P_{M_{2}}^{\xi}\wt{g})}_{L_{\xi}^{2}} \n{P_{N}^{x}P_{M}^{\xi}h}_{L_{\xi}^{2}} dx dt\\
\lesssim&\int \sum_{\substack{M_{2}\gtrsim M\\N_{1}\gtrsim N }}\n{P_{N_{1}}^{x}P_{\leq M_{2}}^{\xi}\wt{f}}_{L_{\xi}^{\frac{6}{1-2\ga}}}\n{P_{\leq N_{1}}^{x}P_{M_{2}}^{\xi}\wt{g}}_{L_{\xi}^{3}} \n{P_{N}^{x}P_{M}^{\xi}h}_{L_{\xi}^{2}} dx dt.
\end{align*}
By Bernstein inequality and Sobolev inequality that $W^{1+\ga,2}\hookrightarrow L^{\frac{6}{1-2\ga}}$,
\begin{align*}
I_{B}\leq& \int \sum_{N_{1}\gtrsim N}\n{P_{N_{1}}^{x}\lra{\nabla_{\xi}}^{1+\ga}\wt{f}}_{L_{\xi}^{2}}\sum_{M_{2}\gtrsim M} \frac{M^{r}}{M_{2}^{r}}
\n{\lra{\nabla_{\xi}}^{r}P_{\leq N_{1}}^{x}P_{M_{2}}^{\xi}\wt{g}}_{L_{\xi}^{3}} \\
& \n{\lra{\nabla_{\xi}}^{-r}P_{N}^{x}P_{M}^{\xi}h}_{L_{\xi}^{2}} dx  dt.
\end{align*}
With $r>0$, we use Cauchy-Schwarz in $M$ and $M_{2}$ to get
\begin{align*}
I_{B}\leq& \int \sum_{N_{1}\gtrsim N}
 \n{P_{N_{1}}^{x}\lra{\nabla_{\xi}}^{1+\ga}\wt{f}}_{L_{\xi}^{2}}
\lrs{\sum_{M_{2}\gtrsim M}\frac{M^{r}}{M_{2}^{r}}
\n{\lra{\nabla_{\xi}}^{r}P_{\leq N_{1}}^{x}P_{M_{2}}^{\xi}\wt{g}}_{L_{\xi}^{3}}^{2}}^{\frac{1}{2}}\\
&\lrs{\sum_{M_{2}\gtrsim M}\frac{M^{r}}{M_{2}^{r}}
\n{\lra{\nabla_{\xi}}^{-r}P_{N}^{x}P_{M}^{\xi}h}_{L_{\xi}^{2}}^{2}}^{\frac{1}{2}}dxdt\\
\lesssim &\int \sum_{N_{1}\gtrsim N}
\n{\lra{\nabla_{\xi}}^{1+\ga}P_{N_{1}}^{x}\wt{f}}_{L_{\xi}^{2}}
\n{\lra{\nabla_{\xi}}^{r}P_{\leq N_{1}}^{x}P_{M_{2}}^{\xi}\wt{g}}_{l_{M_{2}}^{2}L_{\xi}^{3}}
\n{\lra{\nabla_{\xi}}^{-r}P_{N}^{x}h}_{L_{\xi}^{2}} dxdt.
\end{align*}
By H\"{o}lder inequality in the $x$-variable,
\begin{align*}
I_{B}\lesssim &
\int \sum_{N_{1}\gtrsim N}
\n{P_{N_{1}}^{x}\lra{\nabla_{\xi}}^{1+\ga}\wt{f}}_{L_{x}^{6}L_{\xi}^{2}}
\n{\lra{\nabla_{\xi}}^{r}P_{\leq N_{1}}^{x}P_{M_{2}}^{\xi}\wt{g}}_{L_{x}^{3}l_{M_{2}}^{2}L_{\xi}^{3}}\n{\lra{\nabla_{\xi}}^{-r}P_{N}^{x}h}_{L_{x}^{2}L_{\xi}^{2}} dt.
\end{align*}
By using that $\n{P_{N_{1}}^{x}\lra{\nabla_{\xi}}^{1+\ga}\wt{f}}_{L_{x}^{6}L_{\xi}^{2}}\lesssim \n{\lra{\nabla_{\xi}}^{1+\ga}\wt{f}}_{L_{x}^{6}L_{\xi}^{2}}$, Minkowski inequality, and Bernstein inequality,
\begin{align*}
I_{B}
\lesssim &
\int \n{\lra{\nabla_{\xi}}^{1+\ga}\wt{f}}_{L_{x}^{6}L_{\xi}^{2}}
\sum_{N_{1}\gtrsim N} \frac{N^{s}}{N_{1}^{s}}
\n{\lra{\nabla_{x}}^{s}\lra{\nabla_{\xi}}^{r}P_{N_{1}}^{x}
P_{M_{2}}^{\xi}\wt{g}}_{l_{M_{2}}^{2}L_{x}^{3}L_{\xi}^{3}}\\
&\n{\lra{\nabla_{x}}^{-s}\lra{\nabla_{\xi}}^{-r}P_{N}^{x}h}_{L_{x}^{2}L_{\xi}^{2}} dt.
\end{align*}
By Cauchy-Schwarz in $N$ and $N_{1}$,
\begin{align*}
I_{B}\lesssim &
\int \n{\lra{\nabla_{\xi}}^{1+\ga}\wt{f}}_{L_{x}^{6}L_{\xi}^{2}}
\n{\lra{\nabla_{x}}^{s}\lra{\nabla_{\xi}}^{r}
P_{N_{1}}^{x}P_{M_{2}}^{\xi}\wt{g}}_{l_{N_{1}}^{2}l_{M_{2}}^{2}L_{x}^{3}L_{\xi}^{3}}\n{\lra{\nabla_{x}}^{-s}\lra{\nabla_{\xi}}^{-r}h}_{L_{x}^{2}L_{\xi}^{2}} dt.
\end{align*}
By H\"{o}lder inequality in the $t$-variable,
\begin{align*}
I_{B}\lesssim &\n{\lra{\nabla_{\xi}}^{1+\ga}\wt{f}}_{L_{t}^{2}(0,T;L_{x}^{6}L_{\xi}^{2})}
\n{\lra{\nabla_{x}}^{s}\lra{\nabla_{\xi}}^{r}
P_{N_{1}}^{x}P_{M_{2}}^{\xi}\wt{g}}_{L_{t}^{2}(0,T;l_{N_{1}}^{2}l_{M_{2}}^{2}L_{x}^{3}L_{\xi}^{3})}\\
&\n{\lra{\nabla_{x}}^{-s}\lra{\nabla_{\xi}}^{-r}h}_{L_{t}^{\infty}(0,T;L_{x}^{2}L_{\xi}^{2})}.
\end{align*}
Inserting in $\wt{f}(t)=U(t)\wt{f}_{0}$ and $\wt{g}(t)=U(t)\wt{g}_{0}$, we use Strichartz estimate \eqref{equ:strichartz estimate,linear} to obtain
\begin{align*}
I_{B}\lesssim &\n{\lra{\nabla_{\xi}}^{1+\ga}U(t)\wt{f}_{0}}_{L_{t}^{2}(0,T;L_{x}^{6}L_{\xi}^{2})}
\n{\lra{\nabla_{x}}^{s}\lra{\nabla_{\xi}}^{r}P_{N_{1}}^{x}P_{M_{2}}^{\xi}\wt{g}_{0}}_{l_{N_{1}}^{2}l_{M_{2}}^{2}
L_{x}^{2}L_{\xi}^{2}}\\
&\n{\lra{\nabla_{x}}^{-s}\lra{\nabla_{\xi}}^{-r}h}_{L_{t}^{\infty}(0,T;L_{x}^{2}L_{\xi}^{2})}\\
\lesssim & |T|^{\frac{1}{2}}\n{\lra{\nabla_{\xi}}^{1+\ga}U(t)\wt{f}_{0}}_{L_{t}^{\infty}(0,T;L_{x}^{6}L_{\xi}^{2})}
\n{\wt{g}_{0}}_{H_{x}^{s}H_{\xi}^{r}}
\n{\lra{\nabla_{x}}^{-s}\lra{\nabla_{\xi}}^{-r}h}_{L_{t}^{\infty}(0,T;L_{x}^{2}L_{\xi}^{2})}\\
\lesssim&|T|^{\frac{1}{2}}\n{\lra{v}^{1+\ga}f_{0}}_{L_{v}^{2}L_{x}^{6}}
\n{\lra{\nabla_{x}}^{s}\lra{v}^{r}g_{0}}_{L_{x,v}^{2}}\n{\lra{\nabla_{x}}^{-s}\lra{\nabla_{\xi}}^{-r}h}_{L_{t}^{\infty}(0,T;L_{x}^{2}L_{\xi}^{2})}.
\end{align*}
Thus, we complete the proof of \eqref{equ:bilinear estimate,gain term,regularity bound,xi,duality} for Case B.

\end{proof}

\noindent \textbf{Acknowledgements}
The authors would like to thank Yan Guo for many encouragements and delightful discussions regarding to this work.
X. Chen was supported in part by NSF grant DMS-2005469 and a Simons fellowship numbered 916862, S. Shen was supported in part by the Postdoctoral Science Foundation of China under Grant 2022M720263, and Z. Zhang was supported in part by NSF of China under Grant 12171010 and 12288101.

\appendix

\section{Sobolev-type and Strichartz Estimates}\label{section:Strichartz Estimates}

\begin{lemma}[Fractional Leibniz rule, {\cite{gulisashvili1996exact}}]\label{lemma:generalized leibniz rule}
Suppose $1<r<\infty$, $s\geq 0$ and $\frac{1}{r}=\frac{1}{p_{i}}+\frac{1}{q_{i}}$ with $i=1,2$, $1<q_{1}\leq \infty$, $1<p_{2}\leq \infty$.
Then
\begin{align}
\n{\lra{\nabla_{x}}^{s}(fg)}_{L^{r}}\leq C\n{\lra{\nabla_{x}}^{s}f}_{L^{p_{1}}}\n{g}_{L^{q_{1}}}+\n{f}_{L^{p_{2}}}
\n{\lra{\nabla_{x}}^{s}g}_{L^{q_{2}}}
\end{align}
where the constant $C$ depends on all of the parameters but not on $f$ and $g$.
\end{lemma}

Recall the abstract Strichartz estimates.
\begin{theorem}[{\cite[Theorem 1.2]{keel1998endpoint}}]\label{lemma:strichartz estimate,keel-tao}
Suppose that for each time
$t$ we have an operator $U(t)$ such that
\begin{align*}
\n{U(t)f}_{L_{x}^{2}}\lesssim& \n{f}_{L_{x}^{2}},\\
\n{U(t)(U(s))^{*}f}_{L_{x}^{\infty}}\lesssim& |t-s|^{-\sigma}\n{f}_{L_{x}^{1}}.
\end{align*}
Then it holds that
\begin{align}\label{equ:strichartz,usual one}
\n{U(t)f}_{L_{t}^{q}L_{x}^{p}}\lesssim \n{f}_{L_{x}^{2}},
\end{align}
for all sharp $\sigma$-admissible exponent pair that
\begin{align}
\frac{2}{q}+\frac{2\sigma}{p}=\sigma,\quad q\geq 2,\ \sigma>1.
\end{align}

\end{theorem}

The symmetric hyperbolic Schr\"{o}dinger equation is
\begin{equation}
\left\{
\begin{aligned}
i\pa_{t}\phi+\nabla_{\xi}\cdot \nabla_{x}\phi=&0,\\
\phi(0)=&\phi_{0}.
\end{aligned}
\right.
\end{equation}
Note that the linear propagator $U(t)=e^{it\nabla_{\xi}\cdot \nabla_{x}}$ satisfies the energy and dispersive estimates
\begin{equation}
\begin{aligned}
&\n{e^{it\nabla_{\xi}\cdot\nabla_{x}}\phi_{0}}_{L_{x\xi}^{2}}\lesssim \n{\phi_{0}}_{L_{x\xi}^{2}},\\
&\n{e^{it\nabla_{\xi}\cdot\nabla_{x}}\phi_{0}}_{L_{x\xi}^{\infty}}\lesssim t^{-3}\n{\phi_{0}}_{L_{x\xi}^{1}}.
\end{aligned}
\end{equation}
Then by Theorem \ref{lemma:strichartz estimate,keel-tao}, this gives the Strichartz estimate that
\begin{align}\label{equ:strichartz estimate,linear}
\n{e^{it\nabla_{\xi}\cdot \nabla_{x}}\phi_{0}}_{L_{t}^{q}L_{x\xi}^{p}}\lesssim \n{\phi_{0}}_{L_{x\xi}^{2}},\quad \frac{2}{q}+\frac{6}{p}=3,\quad q\geq 2.
\end{align}

\bibliographystyle{abbrv}
%\bibliographystyle{plain}
%\nocite{*}
\bibliography{references}

\end{document}